\documentclass{amsart}
\usepackage{color}
\usepackage{amsmath}
\usepackage{amsthm}
\usepackage{amstext}
\usepackage{amssymb}
\usepackage{amsfonts}
\usepackage{graphicx}
\usepackage{young}
\usepackage{mathdots}
\usepackage{multicol}
\usepackage{mathrsfs}
\usepackage{stmaryrd}
\usepackage[all]{xy}
\usepackage{mathabx}
\usepackage{tikz}
\usepackage{soul}
\usepackage{mathtools}

\newcommand{\verteq}{\rotatebox{90}{$\,=$}}
\newcommand{\equalto}[2]{\underset{\scriptstyle\overset{\mkern4mu\verteq}{#2}}{#1}}

\theoremstyle{plain}

\theoremstyle{definition}
\newtheorem{theorem}{Theorem}[section]
\newtheorem{remark}[theorem]{Remark}

\newtheorem{lemma}[theorem]{Lemma}
\newtheorem{definition}[theorem]{Definition}

\newtheorem{problem}[theorem]{Problem}

\newtheorem{example}[theorem]{Example}

\newtheorem{corollary}[theorem]{Corollary}

\DeclareMathAlphabet{\mathpzc}{OT1}{pzc}{m}{it}

\begin{document}
\author{Alexander Garver and Gregg Musiker}
\title{On Maximal Green Sequences for Type $\mathbb{A}$ Quivers}
\begin{abstract}
Given a framed quiver, i.e. one with a frozen vertex associated to each mutable vertex, there is a concept of \emph{green mutation}, as introduced by Keller.  Maximal sequences of such mutations, known as \emph{maximal green sequences}, are important in representation theory and physics as they have numerous applications, including the computations of spectrums of BPS states, Donaldson-Thomas invariants, tilting of hearts in derived categories, and quantum dilogarithm identities.  In this paper, we study such sequences and construct a maximal green sequence for every quiver mutation-equivalent to an orientation of a type $\mathbb{A}$ Dynkin diagram.   
\end{abstract}
\maketitle

\tableofcontents

\section{Introduction}

A very important problem in cluster algebra theory, with connections to polyhedral combinatorics and the enumeration of BPS states in string theory, is to determine when a given quiver has a maximal green sequence.  In particular, it is open to decide which quivers arising from triangulations of surfaces admit a maximal green sequence, although progress for surfaces has been made in \cite{L}, \cite{Bucher}, \cite{BucherMills} and in the physics literature in \cite{ACCERV}.  In \cite{ACCERV}, they give heuristics for exhibiting maximal green  sequences for quivers arising from triangulations of surfaces with boundary and present examples of this for spheres with at least 4 punctures and tori with at least 2 punctures. They write down a particular triangulation of such a surface and show that the quiver defined by this triangulation has a maximal green sequence. In \cite{Bucher, BucherMills}, this same approach is used on surfaces of any genus with at least 2 punctures. In \cite{L}, it is shown that there do not exist maximal green sequences for a quiver arising from any triangulation of a closed once-punctured genus $g$ surface. It is still unknown the exact set of surfaces with the property that each of its triangulations defines a quiver admitting a maximal green sequence. 


Outside the class of quivers defined by triangulated surfaces there has also been progress in proving that certain quivers do not have maximal green sequences. In \cite{BDP}, it is shown that if a quiver has non-degenerate Jacobi-infinite potential, then the quiver has no maximal green sequences. This is used in \cite{BDP} to show that a certain McKay quiver has no maximal green sequences, and in \cite{Seven1} it is shown that the $\mathbb{X}_7$ quiver has no maximal green sequences.  Other work \cite{Muller} illustrates that it is possible to have two mutation-infinite quivers that are mutation equivalent to one another where only one of the two admits a maximal green sequence.

Even for cases where the existence of maximal green sequences is known (e.g. for \emph{quivers of type} $\mathbb{A}$), the problem of exhibiting, classifying or counting maximal green sequences  has been challenging and serves as our motivation.  By a quiver of type $\mathbb{A}$, we mean any quiver that is \emph{mutation-equivalent} to an orientation of a type $\mathbb{A}$ Dynkin diagram.  In the case where $Q$ is acyclic, one can find a maximal green  sequence whose length is the number of vertices of $Q$,  by mutating at sources and iterating until all vertices have been mutated exactly once. In general, maximal green sequences must have length at least the number of vertices of $Q$. However, even for the smallest non-acyclic quiver, i.e. the oriented 3-cycle (of type $\mathbb{A}_3$), a shortest maximal green  sequence is of length 4. (While we were in the process of revising this paper, it was shown in \cite{CDRSW} that the shortest possible length of a maximal green sequence for a quiver $Q$ of type $\mathbb{A}_n$ is $n+t$ where $t = \#\{\text{3-cycles of $Q$}\}$.  See Remark~\ref{cormierresult1} and Sections~\ref{Sec:OtherSurf} and \ref{Sec:Enum} for more details.) With a goal of gaining a better understanding of such sequences, in this paper we \emph{explicitly} construct a maximal green sequence for every quiver of type $\mathbb{A}$. As any triangulation of the disk with $n+3$ marked points on the boundary defines a quiver of type $\mathbb{A}_n$, our construction shows that the disk belongs to the set of surfaces each of whose triangulations define a quiver admitting a maximal green sequence. We remark that the latter result has also been proved in \cite{CDRSW} by constructing maximal green sequences of type $\mathbb{A}_n$ quivers of shortest possible length. Additionally, the maximal green sequences constructed in \cite{CDRSW} are almost never the same as the maximal green sequences constructed in this paper.


\vspace{0.5em}

In Section \ref{Sec:Prelim}, we begin with background on quivers and their mutations.  This section includes the definition of \emph{maximal green sequences}, which is our principal object of study in this paper.
Section \ref{Sec:DirectSum} describes how to decompose quivers into \emph{direct sums} of strongly connected components, 
which we call irreducible quivers. We remark that this definition of direct sum of quivers, which is based on a quiver gluing rule from \cite{ACCERV}, coincides with the definition of a triangular extension of quivers appearing in \cite{A}. Using this notion of direct sums in Theorem~\ref{tcolordirsum}, we show that for certain direct sums of quivers, to 
construct a maximal green sequence, it suffices to construct a maximal green  sequence for each of their 
irreducible components. We refer to the class of such quivers for which Theorem~\ref{tcolordirsum} holds as $t$\emph{-colored direct sums} of quivers (see Definition~\ref{dirsum}). In Section~\ref{Sec:SurfTriang}, we show that almost all quivers arising from triangulated surfaces (with $1$ connected component) which are a direct sum of at least 2 irreducible components are in fact a $t$-colored directed sum.

For type $\mathbb{A}$ quivers, irreducible quivers have an especially nice form as trees of $3$-cycles, as described 
by Corollary \ref{Cor:Irr}.  This allows us to restrict our attention to \emph{signed irreducible} quivers of \emph{type $\mathbb{A}$}, which are defined in and studied in Section \ref{Sec:Signed}. We then construct a special mutation sequence for every signed irreducible quiver of type $\mathbb{A}$ in Section \ref{Sec:MutSeq}, which we call an \emph{associated mutation sequence}.  This brings us to the main theorem of the paper, Theorem \ref{main2}, which states that this associated mutation sequence is a maximal green sequence. Section \ref{Sec:MutSeq} also highlights how the results of Section \ref{Sec:DirectSum} can be combined with Theorem \ref{main2} to get maximal green sequences for \emph{any} quiver of type $\mathbb{A}$ (see Corollary~\ref{anytypeA}).

The proof of Theorem~\ref{main2} is somewhat involved. The proof of Theorem \ref{main2} essentially follows from two important lemmas (see Lemma~\ref{newmainlemma} and Lemma~\ref{Lemma:mukR}). Our proof begins by attaching \emph{frozen vertices} to a signed irreducible type $\mathbb{A}$ quiver $Q$ to get a \emph{framed quiver} $\widehat{Q}$ (see Section \ref{Sec:Prelim} for more details).  We then apply the associated mutation sequence $\underline{\mu}$ alluded to above, which is constructed in Section \ref{Sec:MutSeq}, but decompose it into certain subsequences as $\underline{\mu} =  \underline{\mu}_n \circ \cdots \circ \underline{\mu}_1\circ \underline{\mu}_0$ and apply each mutation subsequence $\underline{\mu}_k$ one after the other.  In Lemma \ref{newmainlemma} we explicitly describe, for the resulting intermediate quivers, the full subquiver that will be affected by the next iteration of mutations $\underline{\mu}_k$.  We will refer to this full subquiver of $\underline{\mu}_{k-1}\circ \cdots \circ \underline{\mu}_1\circ \underline{\mu}_0(\widehat{Q})$ affected by $\underline{\mu}_k$ as $\overline{R}_{k}$.  Lemma~\ref{Lemma:mukR} then explicitly describes how each of these full subquivers, $\overline{R}_k$, is affected by the mutation sequence $\underline{\mu}_k$.  Together these lemmas lead us to conclude that the associated mutation sequence $\underline{\mu} = \underline{\mu}_{n}\circ \cdots \circ \underline{\mu}_1\circ \underline{\mu}_0$ is a maximal green sequence.

Furthermore, these two lemmas imply that the final quiver $\underline{\mu}_{n}\circ \cdots \circ \underline{\mu}_1\circ \underline{\mu}_0(\widehat{Q})$ is isomorphic (as a directed graph) to $\widecheck{Q}$, the \emph{co-framed} quiver where the directions of arrows between vertices of $Q$ and frozen vertices have all been reversed.
In particular, such an isomorphism is known as a \emph{frozen isomorphism} since it permutes the vertices of $Q$ while leaving the frozen vertices fixed. We refer to this permutation, of vertices of $Q$, as the \emph{permutation induced} by a maximal green sequence (we refer the reader to Section~\ref{Sec:Prelim} for precise definitions of these notions). One of the benefits of proving Theorem~\ref{main2} using the two lemmas mentioned in the previous paragraph is that we exactly describe the permutation that is induced by an associated mutation sequence of a signed irreducible quiver of type $\mathbb{A}$. (See the last paragraph in Section~\ref{Sec:Prelim} and Definition \ref{def:associatedperm}.)  This is a result that may be of independent interest.

Finally, Section \ref{Sec:Conc} ends with further remarks and ideas for future directions, including extensions to quivers arising from triangulations of surfaces other than the disk with marked points on the boundary.

\vspace{1em}

\noindent {\bf Acknowledgements.~} The authors would like to thank T. Br\"{u}stle, M. Del Zotto, B. Keller, S. Ladkani, R. Patrias, V. Reiner, and H. Thomas for useful discussions.  We also thank the referees for their careful reading and numerous suggestions.  The authors  were supported by NSF Grants DMS-1067183, DMS-1148634, and DMS-1362980.

\section{Preliminaries and Notation} \label{Sec:Prelim}

The reader may find excellent surveys on the theory of cluster algebras and maximal green sequences in \cite{BDP,K1}.  For our purposes, we recall a few of the relevant definitions.

A \textbf{quiver} $Q$ is a directed graph without loops or 2-cycles. In other words, $Q$ is a 4-tuple $((Q)_0,(Q)_1,s,t)$, where $(Q)_0 = [M] := \{1,2, \ldots, M\}$ is a set of \textbf{vertices}, $(Q)_1$ is a set of \textbf{arrows}, and two functions $s, t:(Q)_1 \to (Q)_0$ are defined so that for every $\alpha \in (Q)_1$, we have $s(\alpha) \xrightarrow{\alpha} t(\alpha)$. An \textbf{ice quiver} is a pair $(Q,F)$ with $Q$ a quiver and $F \subset (Q)_0$ \textbf{frozen vertices} with the additional restriction that any $i,j \in F$ have no arrows of $Q$ connecting them. We refer to the elements of $(Q)_0\backslash F$ as \textbf{mutable vertices}. By convention, we assume $(Q)_0\backslash F = [N]$ and $F = [N+1,M] := \{N+1, N+2, \ldots, M\}.$ Any quiver $Q$ can be regarded as an ice quiver by setting $Q = (Q, \emptyset)$.

The {\bf mutation} of an ice quiver $(Q,F)$ at a mutable vertex $k$, denoted $\mu_k$, produces a new ice quiver $(\mu_kQ,F)$ by the three step process:

(1) For every $2$-path $i \to k \to j$ in $Q$, adjoin a new arrow $i \to j$.

(2) Reverse the direction of all arrows incident to $k$ in $Q$.

(3) Delete any $2$-cycles created during the first two steps.

\noindent We show an example of mutation in Figure~\ref{example ofmutation} depicting the mutable (resp. frozen) vertices in black (resp. blue).

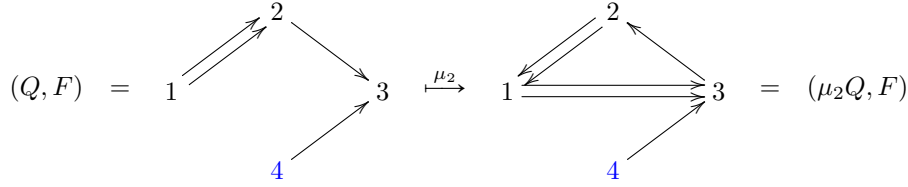
\begin{figure}
$$\begin{array}{c c c c c c c c c}
\raisebox{-.4in}{$(Q,F)$} & \raisebox{-.4in}{=} & {\begin{xy} 0;<1pt,0pt>:<0pt,-1pt>:: 
(0,30) *+{1} ="0",
(40,0) *+{2} ="1",
(80,30) *+{3} ="2",
(40,60) *+{\textcolor{blue}{4}} ="3",
"0", {\ar@<-.5ex>"1"},
"0", {\ar@<.5ex>"1"},
"1", {\ar"2"},
"3", {\ar"2"},
\end{xy}} & \raisebox{-.4in}{$\stackrel{\mu_2}{\longmapsto}$} & {\begin{xy} 0;<1pt,0pt>:<0pt,-1pt>:: 
(0,30) *+{1} ="0",
(40,0) *+{2} ="1",
(80,30) *+{3} ="2",
(40,60) *+{\textcolor{blue}{4}} ="3",
"2", {\ar"1"},
"0", {\ar@<-.5ex>"2"},
"0", {\ar@<.5ex>"2"},
"1", {\ar@<-.5ex>"0"},
"1", {\ar@<.5ex>"0"},
"3", {\ar"2"},
\end{xy}} & \raisebox{-.4in}{=} & \raisebox{-.4in}{$(\mu_2Q,F)$}
\end{array}
$$\caption{An example of quiver mutation.}\label{example ofmutation}\end{figure}

\noindent Since we will focus on quiver mutation in this paper, it will be useful to define a notation for arrows obtained by reversing their direction. Given $\alpha \in (Q)_1$ where $s(\alpha) \stackrel{\alpha}{\rightarrow} t(\alpha)$, formally define $\alpha^{\text{op}}$ to be the arrow where $s(\alpha^\text{op}) = t(\alpha)$ and $t(\alpha^\text{op}) = s(\alpha).$ With this notation, step (2) in the definition of mutation can be rephrased as: if $\alpha \in (Q)_1$ and $s(\alpha) = k$ or $t(\alpha) = k$, replace $\alpha$ with $\alpha^\text{op}$.

The information of an ice quiver can be equivalently described by its (skew-symmetric) \textbf{exchange matrix}. Given $(Q,F),$ we define $B = B_{(Q,F)} = (b_{ij}) \in \mathbb{Z}^{N\times M} := \{N \times M \text{ integer matrices}\}$ by $b_{ij} := \#\{(i \stackrel{\alpha}{\to} j) \in (Q)_1\} - \#\{(j \stackrel{\alpha}{\to} i) \in (Q)_1\}.$ Furthermore, ice quiver mutation can equivalently be defined  as \textbf{matrix mutation} of the corresponding exchange matrix. Given an exchange matrix $B \in \mathbb{Z}^{N\times M}$, the \textbf{mutation} of $B$ at $k \in [N]$, also denoted $\mu_k$, produces a new exchange matrix $\mu_k(B) = (b^\prime_{ij})$ with entries
\[
b^\prime_{ij} := \left\{\begin{array}{ccl}
-b_{ij} & : & \text{if $i=k$ or $j=k$} \\
b_{ij} + \text{sgn}(b_{ik})[b_{ik}b_{kj}]_+ & : & \text{otherwise}
\end{array}\right.
\]
\begin{flushleft}where $[x]_+ = \max(x,0)$. For example, the mutation of the ice quiver above (here $M=4$ and $N=3$) translates into the following matrix mutation. Note that mutation of matrices {(or of ice quivers)} is an involution (i.e. $(\mu_k\circ\mu_k)(B) = B$).\end{flushleft}
\[
\begin{array}{c c c c c c c c c c}
B_{(Q,F)} & = & \left[\begin{array}{c c c | r}
0 & 2 & 0 & 0 \\
-2 & 0 & 1 & 0\\
0 & -1 & 0 & -1\\
\end{array}\right]
& \stackrel{\mu_2}{\longmapsto} &
\left[\begin{array}{c c c | r}
0 & -2 & 2 & 0 \\
2 & 0 & -1 & 0\\
-2 & 1 & 0 & -1\\
\end{array}\right] 
& = & B_{(\mu_2Q,F)}.
\end{array}
\]

In this paper, we focus on successively applying mutations to a fixed ice quiver. As such, if $(Q,F)$ is a given ice quiver we define an \textbf{admissible sequence} of $(Q,F)$, denoted $\textbf{i} = (i_1,\ldots, i_d),$ to be a sequence of mutable vertices of $(Q,F)$ such that $i_j \neq i_{j+1}$ for all $j \in [d-1]$. An admissible sequence $\textbf{i}=(i_1,\ldots, i_d)$ also gives rise to a \textbf{mutation sequence}, which we define to be an expression $\underline{\mu} = \mu_{i_d} \circ \cdots \circ \mu_{i_1}$ with $i_j \neq i_{j+1}$ for all $j \in [d-1]$ that maps an ice quiver $(Q,F)$ to a mutation-equivalent one\footnote{In the sequel, we will identify an admissible sequence with the mutation sequence it defines.}. Let Mut($(Q,F)$) denote the collection of ice quivers obtainable from $(Q,F)$ by a mutation sequence of \textbf{finite length} where the \textbf{length} of a mutation sequence is defined to be $d$, the number of vertices appearing in the associated admissible sequence  $\textbf{i} = (i_1, \ldots, i_d)$.  Given a mutation sequence $\underline{\nu}$ of $\overline{Q} \in \text{Mut}(\widehat{Q})$ we define the $\textbf{support}$ of $\underline{\nu}$, denoted $\text{supp}(\underline{\nu})$, to be the set of mutable vertices of $\overline{Q}$ 
appearing in the admissible sequence which gives rise to $\underline{\nu}$.

Given a quiver $Q$, we focus on successively mutating the framed quiver of Q, which we now define. Following references such as \cite[Section 2.3]{BDP}, given a quiver $Q$, we define its \textbf{framed} (resp. \textbf{coframed}) quiver to be the ice quiver $\widehat{Q}$ (resp. $\widecheck{Q}$) where $(\widehat{Q})_0\ (= (\widecheck{Q})_0) := (Q)_0 \sqcup [N+1, 2N]$, $F = [N+1, 2N]$, and $(\widehat{Q})_1 := (Q)_1 \sqcup \{i \to N+i: i \in [N]\}$ (resp. $(\widecheck{Q})_1 := (Q)_1 \sqcup \{N+i \to i: i \in [N]\}$). We will denote elements of Mut($\widehat{Q}$) by $\overline{Q}$.  In the sequel, we will often write frozen vertices $N+ i$  of $\overline{Q} \in \text{Mut(}\widehat{Q}\text{)}$ as $i^\prime.$ Thus $F = [N]^\prime$ where $S^\prime := \{x^\prime : x \in S\}$ for any $S \subset (Q)_0$. In this paper, we consider ice quivers $\overline{Q} \in \text{Mut}(\widehat{Q})$ up to \textbf{frozen isomorphism} (i.e. an isomorphism of quivers that fixes frozen vertices). Such an isomorphism is equivalent to a simultaneous permutation of the rows and of the first $N$ columns of the exchange matrix $B_{\overline{Q}}$.

A mutable vertex $i$ of an ice quiver $\overline{Q} \in \text{Mut}(\widehat{Q})$ is said to be {\bf green} (resp. {\bf red}) if all arrows of $\overline{Q}$ connecting an element of $[N]^\prime$ and $i$ point away from (resp. towards) $i$. Note that all mutable vertices of $\widehat{Q}$ are green and all vertices of $\widecheck{Q}$ are red.  By the Sign-Coherence of {\bf c}- and {\bf g}-vectors for cluster algebras \cite[Theorem 1.7]{DWZ2}, it follows that given any $\overline{Q} \in \text{Mut}(\widehat{Q})$ each mutable vertex of $\overline{Q}$ is either red or green.

Let $\underline{\mu} = \mu_{i_d}\circ\cdots\circ \mu_{i_1}$ be a mutation sequence of $\widehat{Q}$.
Define $\{\overline{Q}{(k)}\}_{0\le k \le d}$ to be the sequence of ice quivers where $\overline{Q}{(0)} := \widehat{Q}$ and $\overline{Q}{(j)} := (\mu_{i_j} \circ \cdots \circ \mu_{i_1})(\widehat{Q})$.  (In particular, throughout this paper, we apply any sequence of mutations in order from right-to-left.)  A {\bf green sequence} of $Q$ is an {admissible sequence $\textbf{i} = (i_1,i_2,\ldots i_d)$} of $\widehat{Q}$ such that ${i_j}$ is a green vertex of 
$\overline{{Q}}{(j-1)}$ for each $1 \leq j \leq d$. The admissible sequence $\textbf{i}$ is a {\bf maximal green sequence} of $Q$ if it is a green sequence of $Q$ such that in the final quiver $\overline{Q}{(d)}$, the vertices $1,2,\dots, N$ are all red\footnote{Note that since we identify an admissible sequence with the mutation sequence defined by it, we have that maximal green sequences are identified with maximal green mutation sequences, as they are referred to in \cite{Qiu}.}.  In other words, $\overline{Q}{(d)}$ contains no green vertices.  Following \cite{BDP}, we let $\text{green}(Q)$ denote the set of maximal green sequences of ${Q}$\footnote{By identifying maximal green sequences with maximal green mutation sequences, we abuse notation and write an element of $\text{green}(Q)$ either as an admissible sequence $\textbf{i} = (i_1,\ldots, i_d)$ or as its corresponding mutation sequence $\underline{\mu} = \mu_{i_d} \circ \cdots \circ \mu_{i_1}$.}.

Proposition 2.10 of \cite{BDP} shows that given any maximal green sequence $\underline{\mu}$ of $Q$, one has a frozen isomorphism $\overline{Q}{(d)} \cong \widecheck{Q}$. Such an isomorphism amounts to a permutation of the {mutable} vertices of $\widecheck{Q}$, (i.e. $\overline{Q}{(d)} = \widecheck{Q \sigma}$ for some permutation $\sigma \in \mathfrak{S}_{{N}}$ where $\widecheck{Q\sigma}$ is defined by the exchange matrix $B\sigma = B_{\widecheck{Q}}\sigma$ that has entries $(B\sigma)_{i,j} {=} B_{i\cdot\sigma, j\cdot\sigma}$).  We call this the {\bf permutation induced} by $\underline{\mu}$. Note that we can regard $\sigma$ as an element $\mathfrak{S}_{2N}$ where $i\cdot\sigma = i$ for any $i \in [N+1,2N].$

\section{Direct Sums of Quivers} \label{Sec:DirectSum}
In this section, we define a direct sum of quivers based on notation appearing in \cite[Section 4.2]{ACCERV}. We also show that, under certain restrictions, if a quiver $Q$ can be written as a direct sum of quivers where each summand has a maximal green sequence, then the maximal green sequences of the summands can be concatenated in some way to give a maximal green sequence for $Q$. Throughout this section, we let $(Q_1,F_1)$ and $(Q_2,F_2)$ be finite ice quivers with $N_1$ and $N_2$ vertices, respectively. Furthermore, we assume $(Q_1)_0\backslash F_1 = [N_1]$ and $(Q_2)_0\backslash F_2 = [N_1 + 1, N_1 + N_2].$

\begin{definition}\label{dirsum}
Let $(a_1,\ldots,a_k)$ denote a $k$-tuple of elements from $(Q_1)_0\backslash F_1$ and $(b_1,\ldots, b_k)$ a $k$-tuple of elements from $(Q_2)_0\backslash F_2$. (By convention, we assume that the  $k$-tuple $(a_1,\ldots, a_k)$ is ordered so that $a_i \le a_j$ if $i < j$ unless stated otherwise.)  Additionally, let $({R}_1,F_1) \in \text{Mut}(({Q_1},F_1))$ and $(R_2, F_2) \in \text{Mut}(({Q_2},F_2))$. We define the \textbf{direct sum} of $({R}_1,F_1)$ and $(R_2,F_2)$, denoted $(R_1,F_1)\oplus_{(a_1,\ldots, a_k)}^{(b_1,\ldots,b_k)}(R_2,F_2)$, to be the ice quiver with vertices$$\left((R_1,F_1)\oplus_{(a_1,\ldots, a_k)}^{(b_1,\ldots,b_k)}(R_2,F_2)\right)_0 := ({R}_1)_0 \sqcup ({R}_2)_0 = ({Q}_1)_0 \sqcup ({Q}_2)_0 = [N_1+N_2] \sqcup F_1 \sqcup F_2$$ and arrows $$\left((R_1,F_1)\oplus_{(a_1,\ldots, a_k)}^{(b_1,\ldots,b_k)}(R_2,F_2)\right)_1 := (R_1,F_1)_1 \sqcup (R_2,F_2)_1 \sqcup \left\{a_i \stackrel{\alpha_i}{\to} b_i: i \in [k]\right\}.$$ Observe that we have the identification of ice quivers $$\widehat{Q_1\oplus_{(a_1,\ldots, a_k)}^{(b_1,\ldots,b_k)}Q_2} \cong \widehat{Q_1}\oplus_{(a_1,\ldots, a_k)}^{(b_1,\ldots,b_k)}\widehat{Q_2},$$ 
where the total number of vertices $M=2(N_1+N_2)$ in both cases. 

We say that $(R_1,F_1)\oplus_{(a_1,\ldots, a_k)}^{(b_1,\ldots,b_k)}(R_2,F_2)$ is a \textbf{$t$-colored direct sum} if $t = \#\{\text{distinct elements of $\{a_1\ldots, a_k\}$}\}$ and there does not exist $i$ and $j$ such that $$\#\{a_i \stackrel{\alpha}{\to}b_j\} \ge 2.$$ 
\end{definition}

\begin{remark}
Our definition of the direct sum of two quivers coincides with the definition of a \textbf{triangular extension} of two quivers introduced by C. Amiot in \cite{A}, except that we consider quivers as opposed to quivers with potential.  We thank S. Ladkani for bringing this to our attention.  He uses this terminology to study the representation theory of a related class of quivers with potential, called class $\mathcal{P}$ by M. Kontsevich and Y. Soibelman \cite[Section 8.4]{KS}.
\end{remark}

\begin{remark}
The direct sum of two ice quivers is a non-associative operation as is shown in Example~\ref{directsum1}.
\end{remark}

\begin{definition}
We say that a quiver $Q$ is \textbf{irreducible} if $$Q = Q_1\oplus_{(a_1,\ldots, a_k)}^{(b_1,\ldots,b_k)}Q_2$$ for some $k$-tuple $(a_1\ldots,a_k)$ on $(Q_1)_0$ and some $k$-tuple $(b_1,\ldots, b_k)$ on $(Q_2)_0$ implies that $Q_1$ or $Q_2$ is the empty quiver. Note that we define irreducibility only for quivers rather than for ice quivers because we later only study reducibility when $F=\emptyset$. 
\end{definition}

\begin{example}\label{directsum1}
Let $Q$ denote the quiver shown in Figure~\ref{Fig_dirsumexample}. Define $Q_1$ to be the full subquiver of $Q$ on the vertices $1,\ldots, 4$, $Q_2$ to be the full subquiver of $Q$ on the vertices $6,\ldots, 11$, and $Q_3$ to be the full subquiver of $Q$ on the vertex 5. Note that $Q_1, Q_2,$ and $Q_3$ are each irreducible. Then $$Q = Q_1\oplus_{(1,1,1,3,4,4)}^{(5,8,11,8,9,11)}Q_{23}$$ where $Q_{23} = Q_2\oplus_{(6)}^{(5)}Q_3$ so $Q$ is a 3-colored direct sum. On the other hand, we could write $$Q = Q_{12} \oplus_{(1,6)}^{(5,5)}Q_3$$ where $Q_{12} = Q_1\oplus_{(1,1,3,4,4)}^{(8,11,8,9,11)}Q_2$ so $Q$ is a 2-colored direct sum. Additionally, note that $$Q_1\oplus_{(1,1,1,3,4,4)}^{(5,8,11,8,9,11)}Q_{23} = Q_1\oplus_{(1,1,1,3,4,4)}^{(5,8,11,8,9,11)}\left(Q_2\oplus_{(6)}^{(5)}Q_3\right) \neq \left(Q_1\oplus_{(1,1,1,3,4,4)}^{(5,8,11,8,9,11)}Q_2\right)\oplus_{(6)}^{(5)}Q_3$$
where the last equality does not hold because $Q_1\oplus_{(1,1,1,3,4,4)}^{(5,8,11,8,9,11)}Q_2$ is not defined as $5$ is not a vertex of $Q_2$. This shows that the direct sum of two quivers, in the sense of this paper, is not associative.
\end{example}

\begin{figure}
$$\begin{xy} 0;<1pt,0pt>:<0pt,-1pt>:: 
(25,25) *+{1} ="0",
(0,75) *+{2} ="1",
(50,75) *+{3} ="2",
(25,125) *+{4} ="3",
(150,0) *+{5} ="4",
(200,25) *+{6} ="5",
(150,50) *+{7} ="6",
(200,75) *+{8} ="7",
(150,100) *+{9} ="8",
(200,125) *+{10} ="9",
(150,150) *+{11} ="10",
"1", {\ar"0"},
"0", {\ar"2"},
"0", {\ar"4"},
"0", {\ar"7"},
"0", {\ar"10"},
"3", {\ar"1"},
"2", {\ar"3"},
"2", {\ar"7"},
"3", {\ar"8"},
"3", {\ar"10"},
"5", {\ar"4"},
"6", {\ar"5"},
"5", {\ar"7"},
"7", {\ar"6"},
"7", {\ar"8"},
"8", {\ar"9"},
"10", {\ar"8"},
"9", {\ar"10"},
\end{xy}$$
\caption{The quiver $Q$ used in Example \ref{directsum1}.}
\label{Fig_dirsumexample}
\end{figure}
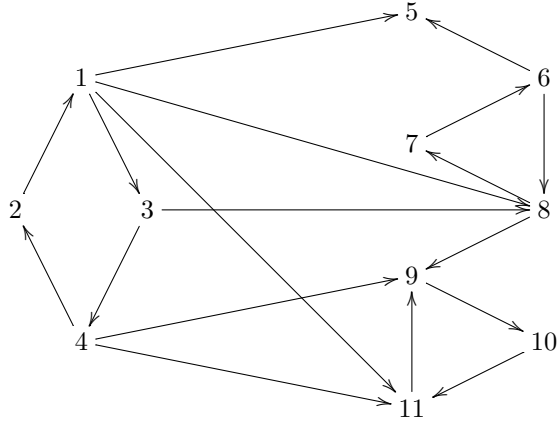

Our next goal is to prove that $Q$ has a maximal green sequence if $Q$ is a $t$-colored direct sum and each of its summands has a maximal green sequence (see Proposition~\ref{tcolordirsum}). Before proving this, we introduce a standard form of $t$-colored direct sums of ice quivers with which we will work:

\begin{eqnarray}
(R,F) & = & \widehat{Q_1}\oplus_{(a_1,\ldots, a_1,\ldots, a_t, \ldots, a_t)}^{(b_{1}^{(1)},\ldots,b_{r_1}^{(1)},\ldots, b_{1}^{(t)}, \ldots, b_{r_t}^{(t)})}\overline{Q}_2 \label{Qeqn}
\end{eqnarray}

\noindent where $\overline{Q}_2 \in \text{Mut}(\widehat{Q_2})$,  $a_1,\ldots, a_t \in (Q_1)_0\backslash [N_1]^\prime$, $b_1^{(1)},\ldots, b_{r_1}^{(1)},\ldots, b_1^{(t)}, \ldots, b_{r_t}^{(t)} \in (Q_2)_0\backslash [N_2]^\prime$, and $\underline{\mu}$ is a fixed mutation sequence $\mu_{i_d}\circ\cdots\circ \mu_{i_1}$ where $\text{supp}(\underline{\mu}) \subset (Q_1)_0$.

We consider the sequence of mutated quivers $(R^{(k)},F)$, for each $k \in [0,d]$, where $R^{(k)} := \mu_{i_k} \circ \cdots \circ \mu_{i_1}R$. By convention, $k = 0$ implies that the empty mutation sequence has been applied to $(R,F)$ so $R^{(0)} = R$.  
For every $k \in [0,d]$, we define $\overline{Q}_1(k) := (\mu_{i_k}\circ \cdots \circ \mu_{i_1})(\widehat{Q}_1)$ and the following set of arrows $$A(k) := \left\{\alpha \in (R^{(k)},F)_1 : \begin{array}{c}
s(\alpha) \mathrm{~or~} t(\alpha) \in (\overline{Q_1}(k))_0\setminus [N_1]^\prime \text{~and the other end of} \\ \alpha \text{~is in } \{a_1^\prime,\ldots,a_t^\prime\} \cup \{b_1^{(1)},\ldots, b_{r_1}^{(1)},\ldots, b_1^{(t)}, \ldots, b_{r_t}^{(t)}\}  \end{array} \right\}.$$

Observe that the sets $A(k)$ only contain arrows
in the partially mutated quivers which have exactly one of their two ends incident to a vertex in $(Q_1)_0$.  The next lemma illustrates how the set of arrows $A(k-1)$ transforms into the set $A(k)$.

\begin{lemma} \label{lem:inductive_color}
If $(i \stackrel{\alpha}{\rightarrow} j) \in A(k)$, but $\alpha, \alpha^\text{op} \not \in A(k-1)$, then there is a 2-path $i \stackrel{\alpha_1}{\rightarrow} i_k \stackrel{\alpha_2}{\rightarrow} j$ in $(R^{(k-1)},F)$ and exactly one of the arrows $\alpha_1, \alpha_2 \in (R^{(k-1)},F)_1$ belongs to $A(k-1).$
\end{lemma}
\begin{proof}
By the definition of quiver mutation, the arrow $(i \stackrel{\alpha}{\rightarrow} j) \in A(k) \subset (R^{(k)},F)_1 = (\mu_{i_k}R^{(k-1)},F)_1$ was originally in $A(k-1)$, was the reversal of an arrow originally in $A(k-1)$, or resulted from a $2$-path.

By hypothesis, we must be in the last case.  By the definition of $A(k)$, either the source or target of $\alpha$ is in 
$(\overline{Q_1}(k))_0\setminus[N_1^\prime]$ but not both.  Hence the $2$-path 
$i \stackrel{\alpha_1}{\rightarrow} i_k \stackrel{\alpha_2}{\rightarrow} j$ must contain one arrow from $(\overline{Q_1}(k))_0\setminus[N_1^\prime]$ to itself and one arrow in $A(k-1)$.
\end{proof}

In the context of this lemma, we refer to this unique arrow in $A(k-1)$ as $\overline{\alpha}$.  We use Lemma \ref{lem:inductive_color} to define a \textbf{coloring function} to stratify the set of arrows $A(k)$. This will allow us to keep track of their orientations as will be needed to prove a crucial lemma (see Lemma~\ref{tcolorind}).

\begin{definition}\label{colorfunction}
Let $(R,F)$ be a $t$-colored direct sum with a direct sum decomposition of the form shown in (\ref{Qeqn}) and let $\underline{\mu}$ be a mutation sequence where $\text{supp}(\underline{\mu}) \subset (Q_1)_0$. Define a \textbf{coloring function} 
by $$\begin{array}{rcl}
f^0: A(0) & \longrightarrow & \{a_1,\ldots, a_t\} \\ 
\alpha & \longmapsto & s(\alpha).
\end{array}$$ We say that $\alpha \in A(0)$ has \textbf{color} $f^0(\alpha)$ in $(R,F)$. Now, inductively we define a coloring function on each ice quiver $(R^{(k)},F)$ where $k \in [0,d].$   Define $f^k : A(k) \to \{a_1,\ldots, a_t\}$ by 

$$\begin{array}{rcl}
f^k(\alpha) & = &  \left\{\begin{array}{rcl} f^{k-1}(\overline{\alpha}) & : & \text{if } \alpha, \alpha^{\text{op}} \not \in A(k-1)\\ 
f^{k-1}(\alpha^\text{op}) & : & \text{if } \alpha \not \in A(k-1), \ \alpha^{\text{op}} \in A(k-1),\\
f^{k-1}(\alpha) & : & \text{if } \alpha \in A(k-1).  \end{array}\right.
\end{array}$$

\noindent We say that $\alpha \in A(k)$ has \textbf{color} $f^k(\alpha)$ in $(R^{(k)},F).$ 
\end{definition}

\begin{example}
Using the notation from Example~\ref{directsum1} and writing $$\widehat{Q} = \widehat{Q}_1\oplus_{(1,1,1,3,4,4)}^{(5,8,11,8,9,11)}\left(\widehat{Q}_2\oplus_{(6)}^{(5)}\widehat{Q}_3\right),$$ we have $a_1 = 1$ and $b^{(1)}_{1} = 5, b^{(1)}_{2} = 8, b^{(1)}_{3}=11$, $a_2 = 3$ and $b^{(2)}_{1} = 8$, and $a_3 = 4$ and $b^{(3)}_{1} = 9, b^{(3)}_{2} = 11$. In Figure~\ref{colorevolve}, we show $\widehat{Q}$ and $\mu_3\widehat{Q}$. The label written on an arrow $\alpha$ of $\widehat{Q}$ or $\mu_3\widehat{Q}$ indicates its color with respect to $Q_1$.  
\end{example}

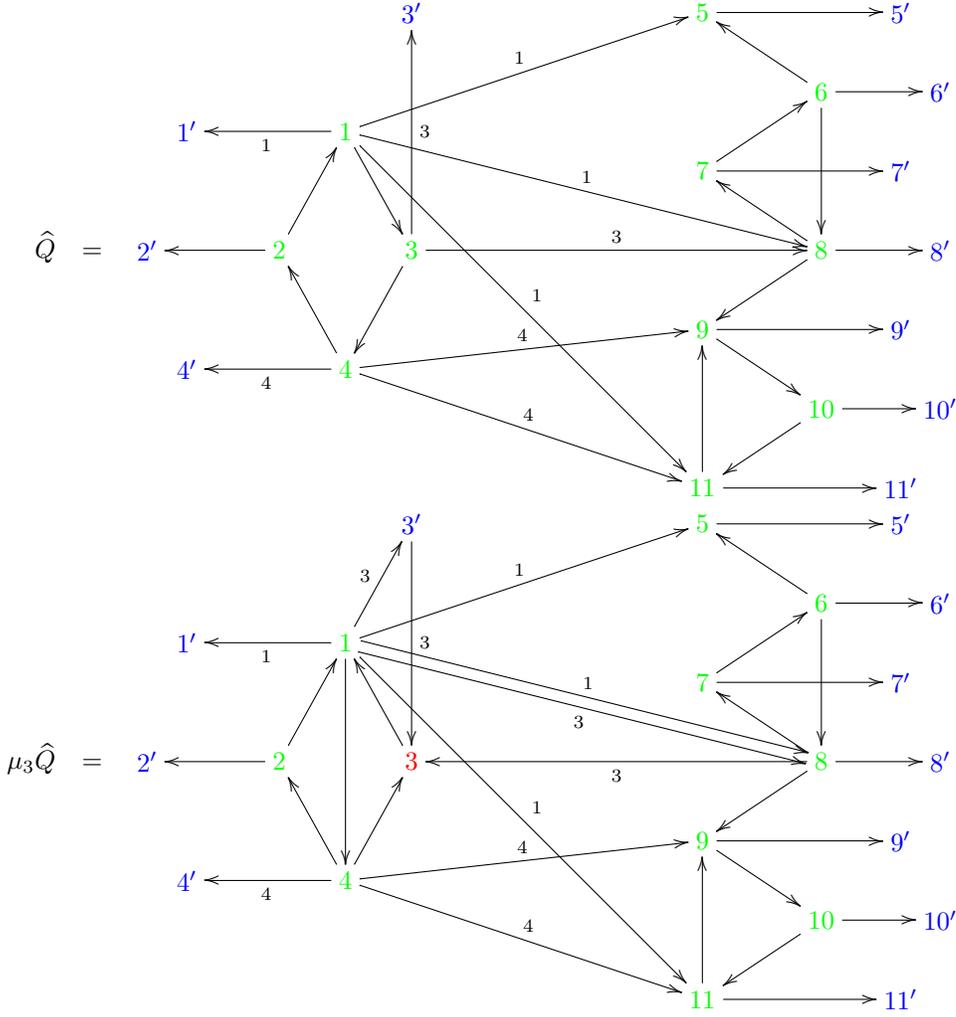
\begin{figure}[h]
$\begin{array}{rcl}
\raisebox{-1.25in}{$\widehat{Q}$} & \raisebox{-1.25in}{=} & \begin{xy} 0;<1pt,0pt>:<0pt,-1pt>:: 
(75,45) *+{\textcolor{green}{1}} ="0",
(50,90) *+{\textcolor{green}{2}} ="1",
(100,90) *+{\textcolor{green}{3}} ="2",
(75,135) *+{\textcolor{green}{4}} ="3",
(210,0) *+{\textcolor{green}{5}} ="4",
(255,30) *+{\textcolor{green}{6}} ="5",
(210,60) *+{\textcolor{green}{7}} ="6",
(255,90) *+{\textcolor{green}{8}} ="7",
(210,120) *+{\textcolor{green}{9}} ="8",
(255,150) *+{\textcolor{green}{10}} ="9",
(210,180) *+{\textcolor{green}{11}} ="10",
(15,45) *+{\textcolor{blue}{1^\prime}} ="11",
(0,90) *+{\textcolor{blue}{2^\prime}} ="12",
(100,0) *+{\textcolor{blue}{3^\prime}} ="13",
(15,135) *+{\textcolor{blue}{4^\prime}} ="14",
(285,60) *+{\textcolor{blue}{7^\prime}} ="15",
(300,30) *+{\textcolor{blue}{6^\prime}} ="16",
(285,0) *+{\textcolor{blue}{5^\prime}} ="17",
(300,90) *+{\textcolor{blue}{8^\prime}} ="18",
(285,120) *+{\textcolor{blue}{9^\prime}} ="19",
(300,150) *+{\textcolor{blue}{10^\prime}} ="20",
(285,180) *+{\textcolor{blue}{11^\prime}} ="21",
"1", {\ar"0"},
"0", {\ar"2"},
"0", {\ar^1"4"},
"0", {\ar^1"7"},
"0", {\ar^1"10"},
"0", {\ar^1"11"},
"3", {\ar"1"},
"1", {\ar"12"},
"2", {\ar"3"},
"2", {\ar^3"7"},
"2", {\ar_3"13"},
"3", {\ar^4"8"},
"3", {\ar^4"10"},
"3", {\ar^4"14"},
"5", {\ar"4"},
"4", {\ar"17"},
"6", {\ar"5"},
"5", {\ar"7"},
"5", {\ar"16"},
"7", {\ar"6"},
"6", {\ar"15"},
"7", {\ar"8"},
"7", {\ar"18"},
"8", {\ar"9"},
"10", {\ar"8"},
"8", {\ar"19"},
"9", {\ar"10"},
"9", {\ar"20"},
"10", {\ar"21"},
\end{xy}\\
\raisebox{-1.25in}{$\mu_3\widehat{Q}$} & \raisebox{-1.25in}{=} &\begin{xy} 0;<1pt,0pt>:<0pt,-1pt>:: 
(75,45) *+{\textcolor{green}{1}} ="0",
(50,90) *+{\textcolor{green}{2}} ="1",
(100,90) *+{\textcolor{red}{3}} ="2",
(75,135) *+{\textcolor{green}{4}} ="3",
(210,0) *+{\textcolor{green}{5}} ="4",
(255,30) *+{\textcolor{green}{6}} ="5",
(210,60) *+{\textcolor{green}{7}} ="6",
(255,90) *+{\textcolor{green}{8}} ="7",
(210,120) *+{\textcolor{green}{9}} ="8",
(255,150) *+{\textcolor{green}{10}} ="9",
(210,180) *+{\textcolor{green}{11}} ="10",
(15,45) *+{\textcolor{blue}{1^\prime}} ="11",
(0,90) *+{\textcolor{blue}{2^\prime}} ="12",
(100,0) *+{\textcolor{blue}{3^\prime}} ="13",
(15,135) *+{\textcolor{blue}{4^\prime}} ="14",
(285,60) *+{\textcolor{blue}{7^\prime}} ="15",
(300,30) *+{\textcolor{blue}{6^\prime}} ="16",
(285,0) *+{\textcolor{blue}{5^\prime}} ="17",
(300,90) *+{\textcolor{blue}{8^\prime}} ="18",
(285,120) *+{\textcolor{blue}{9^\prime}} ="19",
(300,150) *+{\textcolor{blue}{10^\prime}} ="20",
(285,180) *+{\textcolor{blue}{11^\prime}} ="21",
"1", {\ar"0"},
"2", {\ar"0"},
"0", {\ar"3"},
"0", {\ar^1"4"},
"0", {\ar_3@<-.5ex>"7"},
"0", {\ar^1@<.5ex>"7"},
"0", {\ar^1"10"},
"0", {\ar^1"11"},
"0", {\ar^3"13"},
"3", {\ar"1"},
"1", {\ar"12"},
"3", {\ar"2"},
"7", {\ar^3"2"},
"13", {\ar^3"2"},
"3", {\ar^4"8"},
"3", {\ar^4"10"},
"3", {\ar^4"14"},
"5", {\ar"4"},
"4", {\ar"17"},
"6", {\ar"5"},
"5", {\ar"7"},
"5", {\ar"16"},
"7", {\ar"6"},
"6", {\ar"15"},
"7", {\ar"8"},
"7", {\ar"18"},
"8", {\ar"9"},
"10", {\ar"8"},
"8", {\ar"19"},
"9", {\ar"10"},
"9", {\ar"20"},
"10", {\ar"21"},
\end{xy}
\end{array}$
\caption{The quivers $\widehat{Q}$ and $\mu_3\widehat{Q}$ with the coloring functions $f^1$ and $f^2$, respectively.}
\label{colorevolve}
\end{figure}

Our next result shows how the coloring functions $\{f^k\}_{0\le k \le d}$ defined by an ice quiver $(R,F)$ of the form in (\ref{Qeqn}) and a mutation sequence $\underline{\mu} = \mu_{i_d}\circ \cdots \circ \mu_{i_1}$ partition the arrows connecting a mutable vertex $x \in (Q_1)_0$ and a vertex in $\{b^{(j)}_i: \ {i \in [r_j]}\} \sqcup \{a^\prime_j\}.$

\begin{lemma}\label{colorpartit}
Let $(R,F)$ be a $t$-colored direct sum with a direct sum decomposition of the form shown in (\ref{Qeqn}) and let $\underline{\mu}$ be a mutation sequence where $\text{supp}(\underline{\mu}) \subset (Q_1)_0$.  For any $k \in [0,d]$, we have that the coloring function $f^k$ is defined on each $\alpha \in A(k)$. 

\end{lemma}
\begin{proof}
We proceed by induction on $k$. If $k = 0$, no mutations have been applied so the desired result holds. Suppose the result holds for $(R^{(k-1)},F)$ and we will show that the result also holds for $(R^{(k)},F)$. We can write $(R^{(k)},F) = (\mu_yR^{(k-1)},F)$ for some $y \in (Q_1)_0.$ Let $\alpha \in A(k)$ such that $s(\alpha) = x\in (Q_1)_0$ and $t(\alpha) = z \in \{b^{(j)}_i: i \in [r_j]\} \sqcup \{a^\prime_j\}$ or vice-versa. 
There are three cases to consider:

$\begin{array}{rll}
a)& x = y,\\
b) & \text{$x$ is connected to $y$ and there is a $2$-path $x\to y\to z$ or $x \leftarrow y \leftarrow  z$ in $(R^{(k-1)},F)$,} \\
c) & \text{$x$ does not satisfy $a)$ or $b)$.}
\end{array}$

In \textit{Case a)}, we have that all arrows $\alpha \in A(k-1)$ connecting $x$ and $z$ are replaced by $\alpha^{\text{op}} \in A(k)$. By the definition of the coloring functions, these reversed arrows obtain color $f^k(\alpha) = f^{k-1}(\alpha^\text{op})$. 

In \textit{Case b)}, it follows by Lemma \ref{lem:inductive_color} that an arrow $\alpha \in A(k)$ resulting from mutation of the middle of a $2$-path has a well-defined color given by $f^{k-1}(\overline{\alpha})$.  Further, mutation at $y$ would reverse both arrows of such a $2$-path hence vertex $y$ is in the middle of a $2$-path in $A(k)$ if and only if it is in the middle of a $2$-path in $A(k-1)$.

Finally, in \textit{Case c)}, the mutation at $y$ does not affect the arrows $\alpha$ connecting $x$ and $z$ and therefore the colors of such an arrow is inherited from its color as an arrow in $A(k-1)$.  Note that an arrow between $x$ and $y$ would connect vertices of $(Q_1)_0$ and thus has no color.
\end{proof}

For the proofs in the remainder of this section, we denote the exchange matrix of $(R^{(k)},F)$, as $B_{(R^{(k)},F)} = (b^{k}(x,y))_{x \in [N_1+N_2], y \in [2(N_1+N_2)]}$. Here $b^k(x,y):= \#\{(x \stackrel{\alpha}{\to} y) \in (R^{(k)},F)_1\} - \#\{(y \stackrel{\alpha}{\to} x) \in (R^{(k)},F)_1\}.$  (This differs from the notation of Section \ref{Sec:Prelim} to differentiate it from our notation for the set of vertices $\{b_i^{(j)}: i \in [r_j]\}$.)  Furthermore, we refine this enumeration according to color using the following terminology.
$$b^k(x,y,\ell):= \#\{(x \stackrel{\alpha}{\rightarrow} y) \in (R^{(k)}, F)_1:  \ \alpha \text{ has color $\ell$}\} - \#\{(y \stackrel{\alpha}{\rightarrow} x) \in (R^{(k)}, F)_1:  \ \alpha \text{ has color $\ell$}\}.$$
We proceed with the following two technical lemmas.

\begin{lemma}\label{tcolorind}
Let $(R,F)$ be a $t$-colored direct sum with a direct sum decomposition of the form shown in (\ref{Qeqn}) and let $\underline{\mu}$ be a mutation sequence of $(R,F)$ where $\text{supp}(\underline{\mu}) \subset (Q_1)_0$. For any $k \in [0,d]$, $\ell \in [t]$, and $x \in (Q_1)_0$, all of the arrows of $A(k)$ with color $a_\ell$ and incident to vertex $x$ either all point towards vertex $x$ or all point away from vertex $x$.  Moreover they do so with the same multiplicity.

\end{lemma}

\begin{proof}
We need to show that for any $x \in (Q_1)_0$, $k \in [0,d]$, $j \in [t]$, and $\ell \in \{a_1,\ldots, a_t\}$ we have that $b^k(x,b_i^{(j)},\ell) = b^k(x,a_j^\prime,\ell)$ for all $i \in [r_j]$. We proceed by induction on $k$. If $k = 0$, no mutations have been applied so the desired results holds. Suppose the result holds for $(R^{(k-1)},F)$ and we will show that the result also holds for $(R^{(k)},F)$. We can write $(R^{(k)},F) = (\mu_yR^{(k-1)},F)$ for some $y \in (Q_1)_0.$ Let $x \in (Q_1)_0$ and $z \in \{b^{(j)}_i: i \in [r_j]\} \sqcup \{a^\prime_j\}$ be given. There are three cases to consider:

$\begin{array}{rll}
a)& x = y,\\
b) & \text{$x$ is connected to $y$ and $\text{sgn}(b^{k-1}(x,y)) = \text{sgn}(b^{k-1}(y,z)) \neq 0,$}\\
c) & \text{$x$ does not satisfy $a)$ or $b)$.}
\end{array}$

\noindent By Lemma~\ref{colorpartit}, we know that $b^k(x, z) = \sum_{\ell \in \{a_i:i \in [t]\}} b^k(x,z,\ell)$ and $b^{k-1}(y, z) = \sum_{\ell \in \{a_i:i \in [t]\}} b^{k-1}(y,z,\ell)$. Thus, from the definition of $\mu_y$ and the proof of Lemma~\ref{colorpartit}, we have that $$\begin{array}{rcl}
   b^k(x,z,\ell) & = & \left\{
     \begin{array}{lrl}
       -b^{k-1}(x,z,\ell) & : &  Case \ a) \\
       b^{k-1}(x,y)b^{k-1}(y,z,\ell) + b^{k-1}(x,z,\ell) & : & Case \ b)\\
       b^{k-1}(x,z,\ell) & : & Case \ c).\\
     \end{array}
   \right. \nonumber 
\end{array}$$

By induction, each expression on the right hand side of the equality is independent of the choice of $z \in \{b_i^{(j)}: i \in [r_j]\} \sqcup \{a_j^\prime\}.$ Thus $b^k(x,z,\ell)$ is independent of of the choice of $z \in \{b_i^{(j)}: i \in [r_j]\} \sqcup \{a_j^\prime\}.$ 
\end{proof}

\begin{lemma}\label{Lem:aiprimecolor} Let $(R,F)$ be a $t$-colored direct sum with a direct sum decomposition of the form shown in (\ref{Qeqn}), let $\underline{\mu}$ be a mutation sequence of $(R,F)$ where $\text{supp}(\underline{\mu}) \subset (Q_1)_0$, and let $k \in [0,d]$. In any $(R^{(k)},F)$, the arrows incident to the frozen vertex $a_i^\prime$ (for all $i\in[t]$) have color $a_i$.
\end{lemma}
\begin{proof}
Let $a_i^\prime \in \{a_1^\prime, \ldots, a_t^\prime\}$ be given. We proceed by induction on $k$. If $k = 0$, no mutations have been applied so the desired result holds. Suppose the result holds $(R^{(k-1)}, F)$ and we will show that the result holds for $(R^{(k)}, F).$ We can write $(R^{(k)},F) = (\mu_yR^{(k-1)},F)$ for some $y \in (Q_1)_0.$ As $y \neq a_i^\prime$, there are only two cases to consider: 

$\begin{array}{rll}
b) & \text{$a_i^\prime$ is connected to $y$ and there is a $2$-path $a_i^\prime\to y\to z$ or $a_i^\prime \leftarrow y \leftarrow  z$ in $(R^{(k-1)},F)$,} \\
c) & \text{$a_i^\prime$ does not satisfy $b)$.}
\end{array}$

First, in \textit{Case b)}, if there is a 2-path $a_i^\prime\to y\to z$ in $(R^{(k-1)},F)$ (resp. $a_i^\prime\leftarrow y\leftarrow z$ in $(R^{(k-1)},F)$), then by induction the arrow $(a_i^\prime \to y) \in (R^{(k-1)},F)_1$ (resp. $(a_i^\prime \leftarrow y) \in (R^{(k-1)},F)_1$) has color $a_i$. Thus if there is a 2-path $a_i^\prime\to y\to z$ in $(R^{(k-1)},F)$ (resp. $a_i^\prime\leftarrow y\leftarrow z$ in $(R^{(k-1)},F)$), then there is an arrow $a_i^\prime \to z \in (R^{(k)}, F)_1$ (resp. $a_i^\prime \leftarrow z \in (R^{(k)}, F)_1$) of color $a_i.$

In \textit{Case c)}, the mutation at $y$ does not affect the arrows $\alpha$ connecting $a_i^\prime$ and any vertex $z \in (R^{(k)},F)_0$. Therefore the color of such an arrow is inherited from its color as an arrow in $A(k-1)$. By induction, such arrows have color $a_i$.
\end{proof}

We now arrive at the main result of this section. It shows that if $\widehat{Q}$ is a $t$-colored direct sum each of whose summands has a maximal green sequence, then one can build a maximal green sequence for $Q$ using the maximal green sequences for each of its summands.

\begin{theorem}\label{tcolordirsum}
Let $Q = {Q_1\oplus_{(a_1,\ldots, a_1,\ldots, a_t, \ldots, a_t)}^{(b_{1}^{(1)},\ldots,b_{{r_1}}^{(1)},\ldots, b_{1}^{(t)}, \ldots, b_{{r_t}}^{(t)})}Q_2}$ be a $t$-colored direct sum of quivers. If $\underline{\mu}_1 \in \text{green}\left(Q_1\right)$ and $\underline{\mu}_2 \in \text{green}\left(Q_2\right)$, then $\underline{\mu}_2\circ\underline{\mu}_1 \in \text{green}\left(Q\right).$
\end{theorem}
\begin{proof}[Proof of Theorem~\ref{tcolordirsum}]
 Let $\sigma_i$ denote the permutation of the vertices of $Q_i$ induced by $\underline{\mu}_i$. Observe that under the identification in Definition~\ref{dirsum}, we let

$$\begin{array}{rclll}
\widehat{Q} & = & \widehat{Q_1}\oplus_{(a_1,\ldots, a_1,\ldots, a_t, \ldots, a_t)}^{(b_{1}^{(1)},\ldots,b_{{r_1}}^{(1)},\ldots, b_{1}^{(t)}, \ldots, b_{{r_t}}^{(t)})}\widehat{Q_2}.  \\
\end{array}$$

\noindent We also have that $\widehat{Q_1}\oplus_{(a_1,\ldots, a_1,\ldots, a_t, \ldots, a_t)}^{(b_{1}^{(1)},\ldots,b_{{r_1}}^{(1)},\ldots, b_{1}^{(t)}, \ldots, b_{{r_t}}^{(t)})}\widehat{Q_2}$ is a $t$-colored direct sum of the form shown in (\ref{Qeqn}).

We first show that $\underline{\mu}_1\widehat{Q}$ is a $s$-colored direct sum (for some $s$). Let $\underline{\mu}_1 = \mu_{i_d}\circ \cdots \circ \mu_{i_1}.$  Since $\underline{\mu}_1 \in \text{green}(Q_1),$ we have that $\underline{\mu}_1\widehat{Q_1} = \widecheck{Q_1\sigma_1}$ and so for each frozen vertex $a_j^\prime$ with $j \in [t]$, we obtain that $x_j := a_j\cdot \sigma_1 \in (Q_1)_0$ is the unique mutable vertex of $\widehat{Q}$ that is connected to  $a_j^\prime$ by an arrow. Furthemore, $(x_j \stackrel{\alpha}{\leftarrow} a_j^\prime) \in (\underline{\mu}_1\widehat{Q})_1$ is the unique arrow of $\underline{\mu}_1\widehat{Q}$ connecting these two vertices.

By Lemma~\ref{colorpartit}, for any $a_j^\prime$ we have that $b^d(x_j, a_j^\prime) = \sum_{\ell \in \{a_i: \ i \in [t]\}} b^d(x_j,a_j^\prime,\ell).$ Since $\widecheck{Q_1\sigma_1}$ has no 2-cycles, $\text{sgn}(b^d(x_j,a_j^\prime,\ell)) \le 0$ for any $\ell \in \{a_i: \ i \in [t]\}.$ By Lemma~\ref{Lem:aiprimecolor}, $\alpha_j$ has color $a_j$ so $b^d(x_j,a_j^\prime) = b^d(x_j,a_j^\prime,a_j).$ By Lemma~\ref{tcolorind}, given any $x_j := a_j\cdot \sigma_1\in (Q_1)_0$ we have that $b^d(x_j,z) = b^d(x_j,z,a_j) = -1$ for any $z \in \{b^{(j)}_i : \ i \in [r_j]\}\sqcup \{a_j^\prime\}.$ Thus we have that $\underline{\mu}_1\widehat{Q} = \widehat{Q_2} \oplus_{(b^{(1)}_1, \ldots, b^{(1)}_{r_1},\ldots, b^{(t)}_1,\ldots, b^{(t)}_{r_t})}^{(x_1,\ldots,x_1,\ldots, x_t,\ldots, x_{t})}\widecheck{Q_1\sigma_1}$ is a $s$-colored direct sum where $\{b_1^{(1)},\ldots, b_{r_1}^{(1)}, \ldots, b_1^{(t)}, \ldots, b_{r_t}^{(t)}\}$ is a multiset on $(Q_2)_0\backslash F_2$ (with $s$ distinct elements) and $\{x_1,\ldots, x_1,\ldots, x_t,\ldots, x_t\}$ is a multiset on $(Q_1)_0\backslash F_1$.  Note that in this $s$-colored direct sum, the $b_i^{(j)}$'s are not necessarily given in increasing order.

Next, we show that $\underline{\mu}_2( \underline{\mu}_1(\widehat{Q}))$ is a $t$-colored direct sum. Since $\underline{\mu}_1\widehat{Q}$ is a $s$-colored direct sum 
and $\underline{\mu}_2 = \mu_{j_{d^\prime}}\circ \cdots \circ \mu_{j_1}$ is a mutation sequence with $\text{supp}(\underline{\mu}_2) \subset (Q_2)_0,$ one defines coloring functions $\{g^k\}_{0 \le k \le d^\prime}$ on $\underline{\mu}_1\widehat{Q}$ with respect to $Q_2$ in the sense of Definition~\ref{colorfunction}. Now an analogous argument to that of the previous two paragraphs shows that $$\underline{\mu}_2(\underline{\mu}_1(\widehat{Q})) = \widecheck{Q_1\sigma_1} \oplus^{(y^{(1)}_1, \ldots, y^{(1)}_{r_1},\ldots, y^{(t)}_1,\ldots, y^{(t)}_{r_t})}_{(x_1,\ldots,x_1,\ldots, x_t,\ldots, x_{t})}\widecheck{Q_2\sigma_2}$$
where $y^{(i)}_{j}:= b^{(i)}_{j}\cdot \sigma_ 2$ with $i \in [t],$ $j \in [r_i].$ One now observes that $$(\underline{\mu}_2\circ\underline{\mu}_1)(\widehat{Q}) = \widecheck{Q_1\sigma_1} \oplus^{(y^{(1)}_1, \ldots, y^{(1)}_{r_1},\ldots, y^{(t)}_1,\ldots, y^{(t)}_{r_t})}_{(x_1,\ldots,x_1,\ldots, x_t,\ldots, x_{t})}\widecheck{Q_2\sigma_2} \cong \widecheck{Q}$$ and thus all mutable vertices of $(\underline{\mu}_2\circ\underline{\mu}_1)(\widehat{Q})$ are red.

Finally, since $\underline{\mu}_i \in \text{green}(Q_i)$ for $i = 1,2$, each mutation of $\widehat{Q}$ along $\underline{\mu}_2\circ\underline{\mu}_1$ takes place at a green vertex. Thus $\underline{\mu}_2\circ\underline{\mu}_1 \in \text{green}(Q).$\end{proof}

\begin{remark}We believe that Theorem~\ref{tcolordirsum} holds for any quiver that can be realized as the direct sum of two non-empty quivers, but we do not have a proof. 
\end{remark}

\section{Quivers Arising from Triangulated Surfaces}\label{Sec:SurfTriang} In this section, we show that Theorem~\ref{tcolordirsum} can be applied to quivers that arise from triangulated surfaces. Our main result of this section is that quivers $Q$ arising from triangulated surfaces can be realized as $t$-colored direct sums (see Corollary~\ref{cor:tsurf}). Before presenting this result and its proof, we recall for the reader how a triangulated surface defines a quiver. For more details on this construction, we refer the reader to \cite{FST}.

Let $\textbf{S}$ denote an oriented Riemann surface that may or may not have a boundary and let $\textbf{M} \subset \textbf{S}$ be a finite subset of $\textbf{S}$ where we require that for each component $\textbf{B}$ of $\partial \textbf{S}$ we have $\textbf{B}\cap \textbf{M} \neq \emptyset.$ We call the elements of $\textbf{M}$ \textbf{marked points}, we call the elements of $\textbf{M}\backslash(\textbf{M}\cap \partial \textbf{S})$ \textbf{punctures}, and we call the pair $(\textbf{S},\textbf{M})$ a \textbf{marked surface}. We require that $(\textbf{S},\textbf{M})$ is not one of the following \textbf{degenerate marked surfaces}: a sphere with one, two, or three punctures; a disc with one, two, or three marked points on the boundary; or a punctured disc with one marked point on the boundary.

Given a marked surface $(\textbf{S},\textbf{M}),$ we consider curves on $\textbf{S}$ up to isotopy. We define an $\textbf{arc}$ on \textbf{S} to be a simple curve $\gamma$ in \textbf{S} whose endpoints are marked points and which is not isotopic to a boundary component of \textbf{S}. We say two arcs $\gamma_1$ and $\gamma_2$ on \textbf{S} are \textbf{compatible} if they are isotopic relative to their endpoints to curves that are nonintersecting except possibly at their endpoints. A \textbf{triangulation} of \textbf{S} is defined to be a maximal collection of pairwise compatible arcs, denoted $\textbf{T}$. Each triangulation $\textbf{T}$ of \textbf{S} defines a quiver $Q_\textbf{T}$ by associating vertices to arcs and arrows based on oriented adjacencies (see Figure~\ref{msquiv}).

\begin{figure}[h]
$\begin{array}{rcl}\raisebox{-1in}{\includegraphics[scale=2]{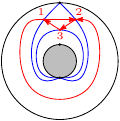}} & \raisebox{-.25in}{$\rightsquigarrow$} & \begin{xy} 0;<1pt,0pt>:<0pt,-1pt>:: 
(0,40) *+{1} ="0",
(60,40) *+{2} ="1",
(30,0) *+{3} ="2",
"2", {\ar"0"},
"1", {\ar"2"},
"0", {\ar@<-.5ex>"1"},
"0", {\ar@<.5ex>"1"},
\end{xy}\end{array}$
\caption{The quiver $Q_\textbf{T}$ defined by a triangulation \textbf{T}.}
\label{msquiv}
\end{figure}

One can also move between different triangulations of a given marked surface $(\textbf{S},\textbf{M}).$ Define the $\textbf{flip}$ of an arc $\gamma \in \textbf{T}$ to be the unique arc $\gamma^\prime \neq \gamma$ that produces a triangulation of $(\textbf{S},\textbf{M})$ given by $\textbf{T}^\prime = (\textbf{T}\backslash\{\gamma\}) \sqcup \{\gamma^\prime\}$ (see Figure~\ref{msflip}). If $(\textbf{S},\textbf{M})$ is a marked surface where $\textbf{M}$ contains punctures, there will be triangulations of \textbf{S} that contain \textbf{self-folded triangles} (the region of \textbf{S} bounded by $\gamma_3$ and $\gamma_4$ in Figure~\ref{taggedarcs} is an example of a self-folded triangle). We refer to the arc $\gamma_3$ (resp. $\gamma_4$) shown in the triangulation in Figure~\ref{taggedarcs} as a \textbf{loop} (a \textbf{radius}). As the flip of a radius of a self-folded triangle is not defined, Fomin, Shapiro, and Thurston introduced \textbf{tagged arcs}, a generalization of arcs, in order to develop such a notion. 

\begin{figure}[h]
$\begin{array}{rcl}\raisebox{-1in}{\includegraphics[scale=2]{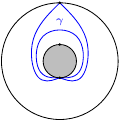}} & \raisebox{-.25in}{$\longleftrightarrow$} & \raisebox{-1in}{\includegraphics[scale=2]{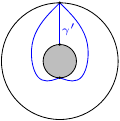}} \end{array}$
\caption{A flip connecting two triangulations of an annulus.}
\label{msflip}
\end{figure}

\begin{figure}[h]
{\includegraphics[scale=1]{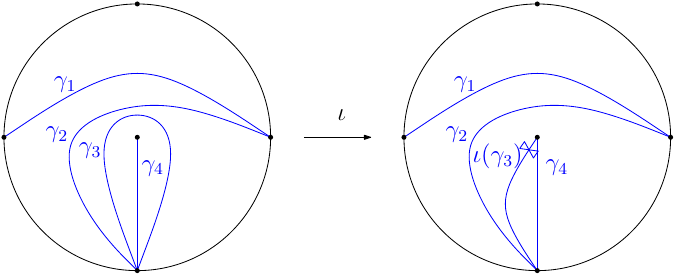}}
\caption{The map identifying a triangulation of a punctured disk as a tagged triangulation of a punctured disk.}
\label{taggedarcs}
\end{figure}

We will not review the details of tagged arcs in this paper, but we remark that any triangulation can be regarded as a \textbf{tagged triangulation} of $(\textbf{S},\textbf{M})$ (i.e. a maximal collection of pairwise compatible tagged arcs). In Figure~\ref{taggedarcs}, we show how one regards a triangulation of $(\textbf{S},\textbf{M})$ as a tagged triangulation of $(\textbf{S},\textbf{M})$. We also note that any tagged triangulation \textbf{T} of $(\textbf{S},\textbf{M})$ gives rise to a quiver $Q_\textbf{T}$ (see Example~\ref{apptosurf} for a quiver defined by a tagged triangulation or see \cite{FST} for more examples and details).

We now review the notion of blocks, which was introduced in \cite{FST} and used to classify quivers defined by a triangulation of some surface.

\begin{definition}\cite[Def. 13.1]{FST}\label{blockdecomp} A \textbf{block} is a directed graph isomorphic to one of the graphs shown in Figure~\ref{blocks}. Depending on which graph it is, we call it a block of type I, II, III, IV, or V. The vertices marked by unfilled circles in Figure~\ref{blocks} are called outlets. A directed graph $\Gamma$ is called \textbf{block-decomposable} if it can be obtained from a collection of disjoint blocks by the following procedure. Take a partial matching of the combined set of outlets; matching an outlet to itself or to another outlet from the same block is not allowed. Identify (or ÒglueÓ) the vertices within each pair of the matching. We require that the resulting graph $\Gamma^\prime$ be connected. If $\Gamma^\prime$ contains a pair of edges connecting the same pair of vertices but going in opposite directions, then remove each such a pair of edges. The result is a block-decomposable graph $\Gamma$.
\begin{figure}[h]
$$\includegraphics[scale=1]{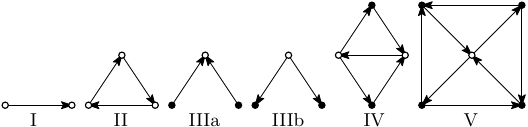}$$
\caption{The Fomin-Shapiro-Thurston blocks.}
\label{blocks}
\end{figure}
\end{definition}

As quivers are examples of directed graphs, one can ask if there is a description of the class of block-decomposable quivers. The following theorem answers this question completely.

\begin{theorem}\cite[Thm. 13.3]{FST}\label{blockdecompdes} Block-decomposable quivers are exactly those quivers defined by a triangulation of some surface.
\end{theorem}
\begin{remark}\label{numberofarrows}
Let $Q_\textbf{T}$ be a quiver defined by a triangulated surface with no frozen vertices. In other words, we are assuming that every $v \in (Q_\textbf{T})_0$ is a mutable vertex.  Then $$\begin{array}{rcl} \#\{\alpha \in (Q_\textbf{T})_1: \ x\stackrel{\alpha}{\longrightarrow} y \text{ for some } y \in (Q_\textbf{T})_0\} & \le & 2\\ \#\{\alpha \in (Q_\textbf{T})_1: \ y\stackrel{\alpha}{\longrightarrow} x \text{ for some } y \in (Q_\textbf{T})_0\} & \le & 2. \end{array}$$
\end{remark}

We now consider the quivers that are defined by triangulations, but are not irreducible. We show that any such quiver is a $t$-colored direct sum. The following lemma is a crucial step in showing that a quiver defined by a triangulation that is not irreducible will not have a double arrow connecting two summands of $Q.$ 

\begin{lemma}\label{2path}
Assume that $Q$ is defined by a triangulated surface (with 1 connected component) and that
$\begin{array}{c}
\begin{xy}0;<1pt,0pt>:<0pt,-1pt>::
(0,0) *+{a} ="1",
(50,0) *+{b} ="2",
"1", {\ar@<1.ex>^{\alpha_1} "2"},
"1", {\ar@<-1.ex>_{\alpha_2} "2"}
\end{xy}\\
\end{array}$
is a proper subquiver of $Q$.  Then there exists a path of length 2 from $b$ to $a$.
\end{lemma}

\begin{proof}
Since $Q$ is defined by a triangulated surface, there exists a block decomposition $\{R_j\}_{j \in [m]}$ of $Q$ by Theorem~\ref{blockdecompdes}. By definition of the blocks, $\alpha_1$ and $\alpha_2$ come from distinct blocks. Without loss of generality, $\alpha_1$ is an arrow of $R_1$ and $\alpha_2$ is an arrow of $R_2$. Furthermore, in $R_i$ with $i = 1,2$ we must have that $s(\alpha_i)$ and $t(\alpha_i)$ are outlets. Thus $R_i$ with $i= 1,2$ is of type I, II, or IV, but by assumption $R_1$ and $R_2$ are not both of type I. When we glue the $R_1$ to $R_2$ to using the identifications associated with $Q$, a case by case analysis shows that there exists a path of length 2 from $b$ to $a$. Furthermore, the vertices corresponding to $a$ and $b$ are no longer outlets. Thus attaching the remaining $R_j$'s will not delete any arrows from this path.
\end{proof}

\begin{corollary} \label{cor:tsurf}
Let $Q$ be a quiver defined by a triangulated surface (with 1 connected component) that is not irreducible.  If $Q \neq\begin{array}{c}
\begin{xy}0;<1pt,0pt>:<0pt,-1pt>::
(0,0) *+{a} ="1",
(50,0) *+{b} ="2",
"1", {\ar@<1.ex>^{\alpha_1} "2"},
"1", {\ar@<-1.ex>_{\alpha_2} "2"}
\end{xy}\\
\end{array}$, then $Q$ is a $t$-colored direct sum for some $t \in \mathbb{N}.$
\end{corollary}
\begin{proof}
Since we are assuming that $Q$ is not irreducible, there exists subquivers $Q_1$ and $Q_2$ of $Q$ such that we can write $Q$ as the direct sum $Q = Q_1 \oplus_{ (a_1,\ldots,a_k)}^{(b_1,\ldots, b_k)} Q_2$ where $\{a_1,\ldots, a_k\}$ is a multiset on $(Q_1)_0$ and $\{b_1,\ldots, b_k\}$ is a multiset on $(Q_2)_0.$ Let $a_i \in \{a_1,\ldots, a_k\}$ and $b_j \in \{b_1,\ldots, b_k\}$ be given. We claim that $\#\{\alpha \in (Q)_1: \ a_i \stackrel{\alpha}{\longrightarrow} b_j\} \le 1.$ Suppose this were not the case, then $Q$ would have a proper subquiver of the form $\begin{array}{c}
\begin{xy}0;<1pt,0pt>:<0pt,-1pt>::
(0,0) *+{a_i} ="1",
(50,0) *+{b_j} ="2",
"1", {\ar@<1.ex>^{\alpha_1} "2"},
"1", {\ar@<-1.ex>_{\alpha_2} "2"}
\end{xy}\\
\end{array}$.  By Lemma~\ref{2path}, there must be a path of length 2 from $b_j$ to $a_i$. This contradicts the fact that all arrows between $\{a_1,\ldots, a_k\}$ and $\{b_1,\ldots, b_k\}$ point towards the latter. Hence, $Q$ is not only a direct sum but is a $t$-colored direct sum.
\end{proof}

\section{Signed Irreducible Type $\mathbb{A}$ Quivers}\label{Sec:Signed}
In this section, we focus our attention on \textbf{type $\mathbb{A}_n$ quivers}, which are defined to be quivers $R \in \text{Mut}(1 \leftarrow 2 \leftarrow \cdots \leftarrow n)$ where $n \ge 1$ is a positive integer.  We begin by classifying irreducible type $\mathbb{A}_n$ quivers. After that, we explain how almost any irreducible type $\mathbb{A}_n$ quiver carries the structure of a binary tree of 3-cycles. In section~\ref{Sec:MutSeq}, we will show how regarding irreducible type $\mathbb{A}_n$ quivers as trees of 3-cycles allows us to construct maximal green sequences for such quivers.

Our first step in classifying irreducible type $\mathbb{A}_n$ quivers is to present the following theorem of Buan and Vatne, which classifies quivers in $\text{Mut}(1\leftarrow 2 \leftarrow \cdots \leftarrow n)$ where $n \ge 1$ is a positive integer. We will say that a quiver is of \textbf{type} $\mathbb{A}$ if it is of type $\mathbb{A}_n$ for some positive integer $n\ge 1.$ 

\begin{lemma}\cite[Prop. 2.4]{BV}\label{BV}
A quiver $Q$ is of type $\mathbb{A}$ if and only if $Q$ satisfies the following:
\begin{itemize}
\item[\textit{i)}] all non-trivial cycles in the underlying graph of $Q$ are of length 3 and are oriented in $Q$,
\item[\textit{ii)}] any vertex has at most four neighbors,
\item[\textit{iii)}] if a vertex has four neighbors, then two of its adjacent arrows belong to one 3-cycle, and the other two belong to another 3-cycle,
\item[\textit{iv)}] if a vertex has exactly three neighbors, then two of its adjacent arrows belong to one 3-cycle, and the third arrow does not belong to any 3-cycle. 
\end{itemize}
\end{lemma}

\begin{corollary} \label{Cor:Irr}
Besides the quiver of type $\mathbb{A}_1$, the \textit{irreducible} quivers of type $\mathbb{A}$ are exactly those quivers $Q$ obtained by gluing together a finite number of Type II blocks $\{S_\alpha\}_{\alpha \in {[n]}}$ in such a way that the cycles in the underlying graph of $Q$ are in bijection with the elements of $\{S_\alpha\}_{\alpha \in [n]}$. Additionally, each $S_\alpha$ shares a vertex with at most three other $S_\beta$'s.
(We say that $S_\alpha$ is \textbf{connected} to $S_\beta$ in such a situation.)
\end{corollary}

\begin{proof}
Assume that $Q$ is a quiver obtained by gluing together a finite number of Type II blocks $\{S_{\alpha}\}_{\alpha \in [n]}$ in such a way that the cycles in the underlying graph of $Q$ are in bijection with the elements of $\{S_\alpha\}_{\alpha \in [n]}$. Then $Q$ satisfies \textit{i)} in Lemma~\ref{BV}. By the rules for gluing blocks together, each vertex $i \in (Q)_0$ has either two or four neighbors so \textit{ii)} and \textit{iv)} in Lemma~\ref{BV} hold. It also follows from the gluing rules that if $i$ has four neighbors, then two of its adjacent arrows belong to one 3-cycle and the other two belong to another 3-cycle so \textit{iii)} in Lemma~\ref{BV} holds. Additionally, since each arrow of $Q$ is contained in an oriented $3$-cycle, there is no way to partition the vertices into two components so that the arrows connecting them coherently point from one to the other. Thus the quiver $Q$ is irreducible.

Conversely, let $Q$ be an irreducible type $\mathbb{A}$ quiver that is not the quiver of type $\mathbb{A}_1$. We first show that any arrow of $Q$ belongs to a (necessarily) oriented 3-cycle of $Q$. Suppose $(i \stackrel{\alpha}{\longrightarrow} j) \in (Q)_1$ does not belong to an oriented 3-cycle of $Q$. Then there exist nonempty full subquivers $Q_1$ and $Q_2$ of $Q$ such that $Q = Q_1 \oplus_{(i)}^{(j)} Q_2.$ (By property \textit{i)}, there cannot be an (undirected) cycle of length larger than $3$.)  This contradicts the fact that $Q$ is irreducible.

Not only is it true that every arrow of $Q$ belongs to an oriented $3$-cycle of $Q$, property \textit{i)} also ensures that $Q$ is obtained by identifying certain vertices of Type II blocks in a finite set of Type II blocks $\{S_\alpha\}_{\alpha \in [n]}$.
Furthermore, property \textit{ii)} in Lemma~\ref{BV} implies these identifications are such that all vertices have two or four neighbors. By properties \textit{i)} and \textit{iii)}, these identifications do not create any new cycles in the underlying graph of $Q$. Thus $Q$ is obtained by gluing together a finite number of Type II blocks $\{S_\alpha\}_{\alpha \in [n]}$ in such a way that the cycles in the underlying graph of $Q$ are in bijection with the elements of $\{S_{\alpha}\}_{\alpha \in [n]}$.
\end{proof}

\begin{definition}
Let $Q$ be an irreducible type $\mathbb{A}$ quiver with at least one 3-cycle. Define a $\textbf{leaf}$ 3-cycle $S_\alpha$ in $Q$ to be a 3-cycle in $Q$ that is connected to at most one other 3-cycle in $Q$. We define a $\textbf{root}$ 3-cycle to be a chosen leaf 3-cycle. 
\end{definition}

\begin{lemma}
Suppose $Q$ is an irreducible type $\mathbb{A}$ quiver with at least one 3-cycle. Then $Q$ has a leaf 3-cycle.
\end{lemma}

\begin{proof}
If $Q$ has exactly one 3-cycle $R$, then $Q = R$ is a leaf 3-cycle. If $Q$ is obtained from the Type II blocks $\{S_i\}_{i\in[n]}$, consider the block $S_{i_1}$. If $S_{i_1}$ is connected to only one other 3-cycle, then $S_{i_1}$ is a leaf 3-cycle. If $S_{i_1}$ is connected to more than one 3-cycle, let $S_{i_2}$ denote one of the 3-cycles to which $S_{i_1}$ is connected. If $S_{i_2}$ is only connected to $S_{i_1}$, then $S_{i_2}$ is a leaf 3-cycle. Otherwise, there exists a 3-cycle $S_{i_3} \neq S_{i_1}$ connected to $S_{i_2}$. By Lemma~\ref{BV} there are no non-trivial cycles in the underlying graph of $Q$ besides those determined by the blocks $\{S_i\}_{i \in [n]}$ so this process will end. Thus $Q$ has a leaf 3-cycle.
\end{proof}

Consider a pair, $(Q,S)$ where $Q$ is an irreducible type $\mathbb{A}$ quiver $Q$ with at least one 3-cycle, and $S$ denotes a root 3-cycle in $Q$. We now define a labeling of the arrows of $Q$, an ordering of the 3-cycles, and a sign function on the set of 3-cycles of $Q$. Adding this additional data to $(Q,S)$ yields a binary tree structure on the set of 3-cycles $\{S_\alpha\}_{\alpha\in [n]}$.

We begin by letting $S_1 :=S$ denote the chosen root 3-cycle, $S_2$ denote the unique 3-cycle connected to $S_1$, and $z_1$ denote the vertex shared by $S_1$ and $S_2$.  (In the event that $Q$ is a single 3-cycle, we choose $z_1$ to be a vertex of $S_1$ arbitrarily.)  Next, we let $\alpha_1$, $\beta_1$ and $\gamma_1$ denote the three arrows of $S_1$ in cyclic order such that $s(\gamma_1)=z_1=t(\beta_1)$, $s(\beta_1)=t(\alpha_1)$, and $s(\alpha_1)=t(\gamma_1)$.  We next label the arrows of $S_2$ such that $s(\alpha_2)=z_1 = t(\gamma_2)$, $t(\alpha_2)= s(\beta_2)$, and $t(\beta_2)=s(\gamma_2)$.  See Figure \ref{Slabel} for examples of this labeling.


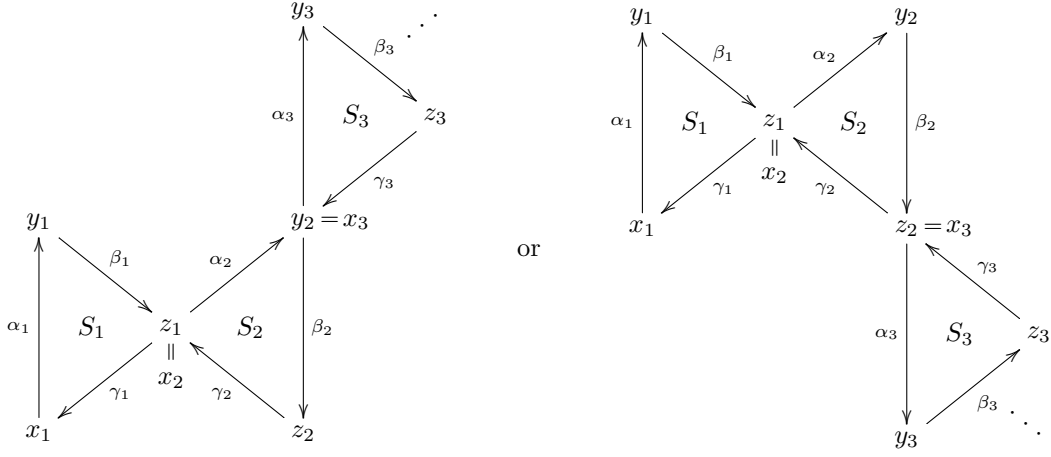
\begin{figure}[h]
$$\begin{array}{ccc}{\begin{xy} 0;<1pt,0pt>:<0pt,-1pt>:: 
(0,160) *+{x_1} ="0",
(0,80) *+{y_1} ="1",
(50,120) *+{z_1} ="2",
(100,80) *+{y_2} ="3",
(100,160) *+{z_2} ="4",
(100,0) *+{y_3} ="5",
(150,40) *+{z_3} ="6",
(140,10) *+{\cdot} = "14",
(145,6) *+{\cdot} = "15",
(150,2) *+{\cdot} = "16",
(50,130) *+{\equalto{}{}} ="7",
(50,140) *+{x_2} ="8",
(110,80) *+{=} ="9",
(120,80) *+{x_3} ="10",
(120,40) *+{S_3} ="11",
(80,120) *+{S_2} ="12",
(20,120) *+{S_1} ="13",
"0", {\ar^{\alpha_1}"1"},
"2", {\ar^{\gamma_1}"0"},
"1", {\ar^{\beta_1}"2"},
"2", {\ar^{\alpha_2}"3"},
"4", {\ar^{\gamma_2}"2"},
"3", {\ar^{\beta_2}"4"},
"3", {\ar^{\alpha_3}"5"},
"6", {\ar^{\gamma_3}"3"},
"5", {\ar^{\beta_3}"6"},
\end{xy}} & \raisebox{-1.25in}{\hspace{1em} \text{or} \hspace{1em}} & \raisebox{.0in}{\begin{xy} 0;<1pt,0pt>:<0pt,-1pt>:: 
(0,80) *+{x_1} ="0",
(0,0) *+{y_1} ="1",
(50,40) *+{z_1} ="2",
(100,0) *+{y_2} ="3",
(100,80) *+{z_2} ="4",
(100,160) *+{y_3} ="5",
(150,120) *+{z_3} ="6",
(140,150) *+{\cdot} = "14",
(145,154) *+{\cdot} = "15",
(150,158) *+{\cdot} = "16",
(50,50) *+{\equalto{}{}} ="7",
(50,60) *+{x_2} ="8",
(110,80) *+{=} ="9",
(120,80) *+{x_3} ="10",
(120,120) *+{S_3} ="11",
(80,40) *+{S_2} ="12",
(20,40) *+{S_1} ="13",
"0", {\ar^{\alpha_1}"1"},
"2", {\ar^{\gamma_1}"0"},
"1", {\ar^{\beta_1}"2"},
"2", {\ar^{\alpha_2}"3"},
"4", {\ar^{\gamma_2}"2"},
"4", {\ar_{\alpha_3}"5"},
"3", {\ar^{\beta_2}"4"},
"6", {\ar_{\gamma_3}"4"},
"5", {\ar_{\beta_3}"6"},
\end{xy}}\end{array}$$
\caption{Labeling arrows of an irreducible quiver of type $\mathbb{A}$.}
\label{Slabel}
\end{figure}

For $i\geq 2$, we order the remaining 3-cycles by a depth-first ordering where we 

\begin{itemize}
\item[(1)] inductively define $S_{i+1}$ to be the 3-cycle attached to the vertex $t(\alpha_i)$,
\item[(2)] define $\alpha_{i+1}$ such that $s(\alpha_{i+1})=t(\alpha_i)$ and then $\beta_{i+1}$, $\gamma_{i+1}$ follow $\alpha_{i+1}$ in cyclic order,
\item[(3)] if no 3-cycle is attached to $t(\alpha_i)$, define $S_{i+1}$ to be the 3-cycle attached to $t(\beta_i)$ and $s(\alpha_{i+1}) =t(\beta_i)$ instead, and finally
\item[(4)] minimally backtrack and continue the depth-first ordering until all arrows and 3-cycles have been labeled.
\end{itemize}

Given a 3-cycle $S_i$ in the block decomposition of $Q$, define $x_i := s(\alpha_i),$ $y_i := s(\beta_i)$, and $z_i := s(\gamma_i)$. The vertex $z_1$ of $S_1$ was already defined in the previous paragraph and that definition of $z_1$ clearly agrees with this one. We say that a 3-cycle $S_i$ is \textbf{positive} (resp. \textbf{negative}) if $s(\alpha_i) = t(\alpha_j)$ (resp. $s(\alpha_i)= t(\beta_j)$) for some $j<i$. We define $\text{sgn}(S_i) := +$ (resp. $-$) if $S_i$ is positive (resp. negative). By convention, we set $\text{sgn}(S_1) = +$.  We define $T_i := (S_i, \text{sgn}(S_i))$ to be a 3-cycle in the block decomposition of $Q$ and its sign. We will refer to $T_i$ where $i \in [n]$ as a \textbf{signed} 3\textbf{-cycle} of $Q$. 
For graphical convenience, we will consistently draw 3-cycles as shown in Figure~\ref{+and-} with the convention that $\text{sgn}(S_i) = +$ (resp. $-$) in the former figure (resp. latter figure). We refer to the data $\mathcal{Q} := (Q,S,\{T_i\}_{i \in [n]})$ as a \textbf{signed} irreducible type $\mathbb{A}$ quiver.

\begin{figure}[h]
\[
\xymatrix
{
& y_i \ar@{->}[dr]^{\beta_i} &  & x_i \ar@{->}[rr]_{T_i}^{\alpha_i} & & y_i \ar@{->}[dl]^{\beta_i}\\
x_i \ar@{->}[ur]^{\alpha_i} & & z_i\ar@{->}[ll]_{T_i}^{\gamma_i} & & z_i \ar@{->}[ul]^{\gamma_i} &\\
}
\]
\caption{A positive (resp. negative) 3-cycle is shown on the left (resp. right).}
\label{+and-}
\end{figure}
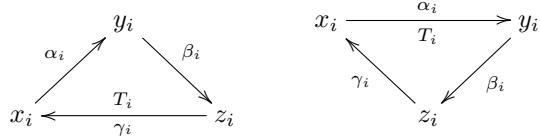

\begin{remark}\label{signdet}
If $Q$ is an irreducible type $\mathbb{A}$ quiver with more than one 3-cycle, then the choice of a root 3-cycle completely determines the sign of each 3-cycle of $Q$. Thus $\mathcal{Q} = (Q, S, \{T_i\}_{i \in [n]})$ depends only on $(Q,S)$ and thus it makes sense to refer to the signed irreducible type $\mathbb{A}$ quiver \textit{defined by} $(Q,S).$
\end{remark}

The next lemma follows immediately from Corollary~\ref{Cor:Irr} and from our definition of the sign of a 3-cycle $S_i$ in $Q$.

\begin{lemma}\label{binarytreelemma}
If $Q$ is an irreducible type $\mathbb{A}$ quiver with at least one 3-cycle, $S$ is a root 3-cycle of $Q$ and $\mathcal{Q} = (Q,S,\{T_i\}_{i \in [n]})$ is a signed irreducible type $\mathbb{A}$ quiver defined by $(Q,S)$, then $\mathcal{Q}$ is equivalent to a labeled binary tree with vertex set $\{S_i\}_{i\in[n]}$ where $S_i$ is connected to $S_j$ by an edge if and only if $S_i$ is connected to $S_j$ (i.e. $S_i$ and $S_j$ share a vertex). Furthermore, a 3-cycle $S_j \in \{S_i\}_{i \in [n]}$ has a right child (resp. left child) if and only if $S_j$ shares the vertex $y_j$ (resp. $z_j$) with another 3-cycle, .
\end{lemma}

For the remainder of this section, we assume that $Q$ is a given irreducible type $\mathbb{A}$ quiver and $S$ a root 3-cycle of $Q$. We also assume $\mathcal{Q}$ is a signed irreducible type $\mathbb{A}$ defined by the data $(Q, S)$. For convenience, we will abuse notation and refer to the \textbf{vertices}, \textbf{arrows}, \textbf{3-cycles}, etc. of $\mathcal{Q}$ with the understanding that we are referring to the vertices, arrows, 3-cycles, etc. of ${Q}$, respectively. Since we will often work with $\widehat{Q}$, the framed quiver of $Q$, it will also be useful to define $\widehat{\mathcal{Q}}$ to be framed quiver of $Q$ with the additional data of $S$, the root 3-cycle of $Q$, and the data of a sign associated with each 3-cycle of $Q$. Now for convenience, we will abuse notation and refer to the \textbf{mutable vertices}, \textbf{frozen vertices}, \textbf{arrows}, and \textbf{3-cycles} of $\widehat{\mathcal{Q}}$ with the understanding that we are referring to the mutable vertices, frozen vertices, arrows, and 3-cycles of $\widehat{Q}$, respectively. We will refer to $\widehat{\mathcal{Q}}$ as a \textbf{signed irreducible type} $\mathbb{A}$ \textbf{framed quiver}. Additionally, we define a \textbf{full subquiver} $\mathcal{R}$ of $\mathcal{Q}$ or $\widehat{\mathcal{Q}}$ to be a full subquiver of $Q$ or $\widehat{Q}$, respectively, with the property that the sign of any 3-cycle $C$ of $\mathcal{R}$ is the same as the sign of $C$ when regarded as a 3-cycle of $\mathcal{Q}$ or $\widehat{\mathcal{Q}}$.

\begin{example}\label{signirred1}
In Figure~\ref{A23}, we show an example of a signed irreducible type $\mathbb{A}_{23}$ quiver, which we denote by $\mathcal{Q}$. The positive 3-cycles of $\mathcal{Q}$ are $T_1, T_3, T_4, T_5, T_7.$ For clarity, we have labeled the arrows of $\mathcal{Q}$ in Figure~\ref{A23}, but we will often suppress these labels in later examples.  We also note that many of the vertices, e.g. $z_1, y_2, y_3, z_3$ could also be labeled as $x_2,x_3,x_4, x_{11}$, respectively, but we suppress the vertex labels $x_i$ (which are shorthand for $s(\alpha_i)$) except for $x_1$.

\begin{figure}[h]
$$\begin{xy} 0;<1pt,0pt>:<0pt,-1pt>:: 
(0,90) *+{x_1} ="0",
(60,90) *+{z_1} ="1",
(30,60) *+{y_1} ="2",
(90,120) *+{z_2} ="3", 
(120,90) *+{y_2} ="4",
(150,60) *+{y_3} ="5",
(180,30) *+{y_4} ="6",
(210,60) *+{z_4} ="7",
(180,90) *+{z_3} ="8",
(210,0) *+{y_5} ="9",
(240,30) *+{z_5} ="10",
(300,30) *+{y_6} ="11",
(270,60) *+{z_6} ="12",
(330,0) *+{y_7} ="13",
(360,30) *+{z_7} ="14",
(420,30) *+{y_8} ="15",
(390,60) *+{z_8} ="16",
(450,60) *+{y_9} ="17",
(420,90) *+{z_9} ="18",
(480,90) *+{y_{10}} ="19",
(450,120) *+{z_{10}} ="20",
(210, 120) *+{z_{11}} = "21", 
(240, 90), *+{y_{11}} = "22", 
"1", {\ar_{T_1}^{\gamma_1}"0"},
"0", {\ar^{\alpha_1}"2"},
"2", {\ar^{\beta_1}"1"},
"3", {\ar^{\gamma_2}"1"},
"1", {\ar_{T_2}^{\alpha_2}"4"},
"4", {\ar^{\beta_2}"3"},
"4", {\ar^{\alpha_3}"5"},
"8", {\ar_{T_3}^{\gamma_3}"4"},
"5", {\ar^{\alpha_4}"6"},
"7", {\ar_{T_4}^{\gamma_4}"5"},
"5", {\ar^{\beta_3}"8"},
"6", {\ar^{\beta_4}"7"},
"6", {\ar^{\alpha_5}"9"},
"10", {\ar_{T_5}"6"},
"10", {\ar^{\gamma_5}"6"},
"9", {\ar^{\beta_5}"10"},
"10", {\ar_{T_6}"11"},
"10", {\ar^{\alpha_6}"11"},
"12", {\ar^{\gamma_6}"10"},
"11", {\ar^{\beta_6}"12"},
"11", {\ar^{\alpha_7}"13"},
"14", {\ar_{T_7}^{\gamma_7}"11"},
"13", {\ar^{\beta_7}"14"},
"14", {\ar_{T_8}^{\alpha_8}"15"},
"16", {\ar^{\gamma_8}"14"},
"15", {\ar^{\beta_8}"16"},
"16", {\ar_{T_9}^{\alpha_9}"17"},
"18", {\ar^{\gamma_9}"16"},
"17", {\ar_{\beta_9}"18"},
"18", {\ar_{T_{10}}^{\alpha_{10}}"19"},
"20", {\ar^{\gamma_{10}}"18"},
"19", {\ar^{\beta_{10}}"20"},
"8", {\ar_{T_{11}}^{\alpha_{11}}"22"},
"22", {\ar^{\beta_{11}}"21"},
"21", {\ar^{\gamma_{11}}"8"},
\end{xy}$$
\caption{A signed irreducible type $\mathbb{A}_{23}$ quiver.}
\label{A23}
\end{figure}
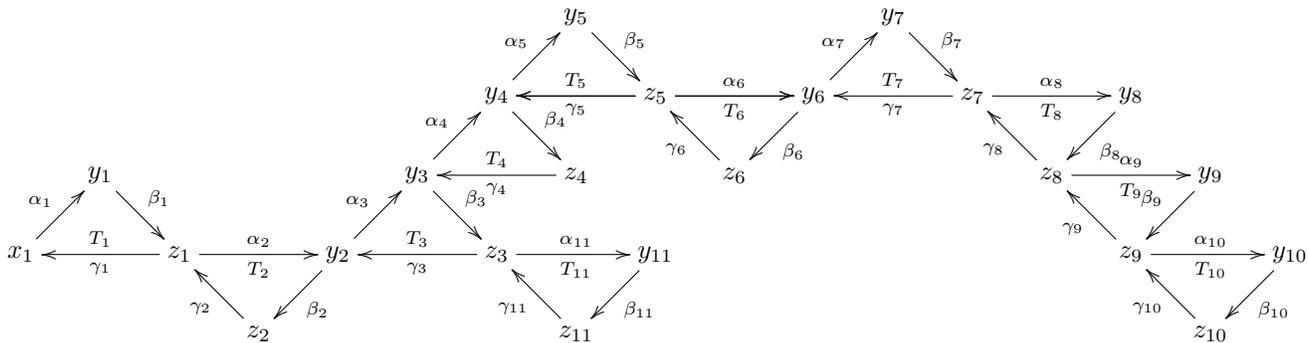

\end{example}

It will be helpful to define an ordering on the vertices of $\widehat{\mathcal{Q}}$.
We label the mutable vertices of $\widehat{\mathcal{Q}}$ according to the linear order $$1=s(\alpha_1) \hspace{1em} <  t(\alpha_1) < t(\beta_1) \hspace{1em} < t(\alpha_2) < t(\beta_2) \hspace{1em} < \dots  \hspace{1em} < t(\alpha_n) < t(\beta_n) = N$$ and the frozen vertices of $\widehat{\mathcal{Q}}$ according to the linear order $$N+1=s(\alpha_1)^\prime \hspace{1em} <  t(\alpha_1)^\prime < t(\beta_1)^\prime \hspace{1em} < t(\alpha_2)^\prime < t(\beta_2)^\prime \hspace{1em} < \dots  \hspace{1em} < t(\alpha_n)^\prime < t(\beta_n)^\prime = 2N.$$ We call this the \textbf{standard ordering} of the vertices of $\widehat{\mathcal{Q}}$.


\begin{example}\label{branchesexample}
Let $\widehat{\mathcal{Q}}$ denote the signed irreducible type $\mathbb{A}_{23}$ framed quiver shown in Figure~\ref{A23prime}. We have labeled the vertices of $\widehat{\mathcal{Q}}$ in Figure~\ref{A23prime} according to the standard ordering. Note that we have suppressed the arrow labels in Figure~\ref{A23prime}.

\begin{figure}[h]
$$\begin{xy} 0;<1pt,0pt>:<0pt,-1pt>:: 
(0,90) *+{\textcolor{green}{1}} ="0",
(0,60) *+{\textcolor{blue}{1^\prime}} = "23",
(30,30) *+{\textcolor{blue}{2^\prime}} = "24",
(60,60) *+{\textcolor{blue}{3^\prime}} = "25",
(60,90) *+{\textcolor{green}{3}} ="1",
(30,60) *+{\textcolor{green}{2}} ="2",
(90,120) *+{\textcolor{green}{5}} ="3", 
(90,150) *+{\textcolor{blue}{5^\prime}} = "27",
(120,90) *+{\textcolor{green}{4}} ="4",
(120,60) *+{\textcolor{blue}{4^\prime}} = "26",
(150,60) *+{\textcolor{green}{6}} ="5",
(150,30) *+{\textcolor{blue}{6^\prime}} = "28",
(180,30) *+{\textcolor{green}{8}} ="6",
(180,0) *+{\textcolor{blue}{8^\prime}} = "29",
(210,60) *+{\textcolor{green}{9}} ="7",
(240,60) *+{\textcolor{blue}{9^\prime}} = "43",
(180,90) *+{\textcolor{green}{7}} ="8",
(180,120) *+{\textcolor{blue}{7^\prime}} = "38",
(210,0) *+{\textcolor{green}{10}} ="9",
(210,-30) *+{\textcolor{blue}{10^\prime}} = "30",
(240,30) *+{\textcolor{green}{11}} ="10",
(240,0) *+{\textcolor{blue}{11^\prime}} = "31",
(300,30) *+{\textcolor{green}{12}} ="11",
(300,0) *+{\textcolor{blue}{12^\prime}} = "32",
(270,60) *+{\textcolor{green}{13}} ="12",
(270,90) *+{\textcolor{blue}{13^\prime}} = "39",
(330,0) *+{\textcolor{green}{14}} ="13",
(330,-30) *+{\textcolor{blue}{14^\prime}} = "33",
(360,30) *+{\textcolor{green}{15}} ="14",
(360,0) *+{\textcolor{blue}{15^\prime}} = "34",
(420,30) *+{\textcolor{green}{16}} ="15",
(420,0) *+{\textcolor{blue}{16^\prime}} = "35",
(390,60) *+{\textcolor{green}{17}} ="16",
(390,90) *+{\textcolor{blue}{17^\prime}} = "40",
(450,60) *+{\textcolor{green}{18}} ="17",
(450,30) *+{\textcolor{blue}{18^\prime}} = "36",
(420,90) *+{\textcolor{green}{19}} ="18",
(420,120) *+{\textcolor{blue}{19^\prime}} = "41",
(480,90) *+{\textcolor{green}{20}} ="19",
(480,60) *+{\textcolor{blue}{20^\prime}} = "37",
(450,120) *+{\textcolor{green}{21}} ="20",
(450,150) *+{\textcolor{blue}{21^\prime}} = "42",
(210, 120) *+{\textcolor{green}{23}} = "21",
(210,150) *+{\textcolor{blue}{23^\prime}} = "45",
(240, 90), *+{\textcolor{green}{22}} = "22",
(240,120) *+{\textcolor{blue}{22^\prime}} = "44",
"1", {\ar_{T_1}"0"},
"0", {\ar"2"},
"2", {\ar"1"},
"3", {\ar"1"},
"1", {\ar_{T_2}"4"},
"4", {\ar"3"},
"4", {\ar"5"},
"8", {\ar_{T_3}"4"},
"5", {\ar"6"},
"7", {\ar_{T_4}"5"},
"5", {\ar"8"},
"6", {\ar"7"},
"6", {\ar"9"},
"10", {\ar_{T_5}"6"},
"10", {\ar"6"},
"9", {\ar"10"},
"10", {\ar_{T_6}"11"},
"10", {\ar"11"},
"12", {\ar"10"},
"11", {\ar"12"},
"11", {\ar"13"},
"14", {\ar_{T_7}"11"},
"13", {\ar"14"},
"14", {\ar_{T_8}"15"},
"16", {\ar"14"},
"15", {\ar"16"},
"16", {\ar_{T_9}"17"},
"18", {\ar"16"},
"17", {\ar"18"},
"18", {\ar_{T_{10}}"19"},
"20", {\ar"18"},
"19", {\ar"20"},
"8", {\ar_{T_{11}}"22"},
"22", {\ar"21"},
"21", {\ar"8"},
"0", {\ar "23"},
"2", {\ar "24"},
"1", {\ar "25"},
"4", {\ar "26"},
"3", {\ar "27"},
"5", {\ar "28"},
"6", {\ar "29"},
"9", {\ar "30"},
"10", {\ar "31"},
"11", {\ar "32"},
"13", {\ar "33"},
"14", {\ar "34"},
"15", {\ar "35"},
"17", {\ar "36"},
"19", {\ar "37"},
"8", {\ar "38"},
"7", {\ar "43"},
"12", {\ar "39"},
"16", {\ar "40"},
"18", {\ar "41"},
"20", {\ar "42"},
"21", {\ar "45"},
"22", {\ar "44"},
\end{xy}$$
\caption{The framed quiver of a signed irreducible type $\mathbb{A}_{23}$ quiver with vertices labeled using the standard ordering.}
\label{A23prime}
\end{figure}
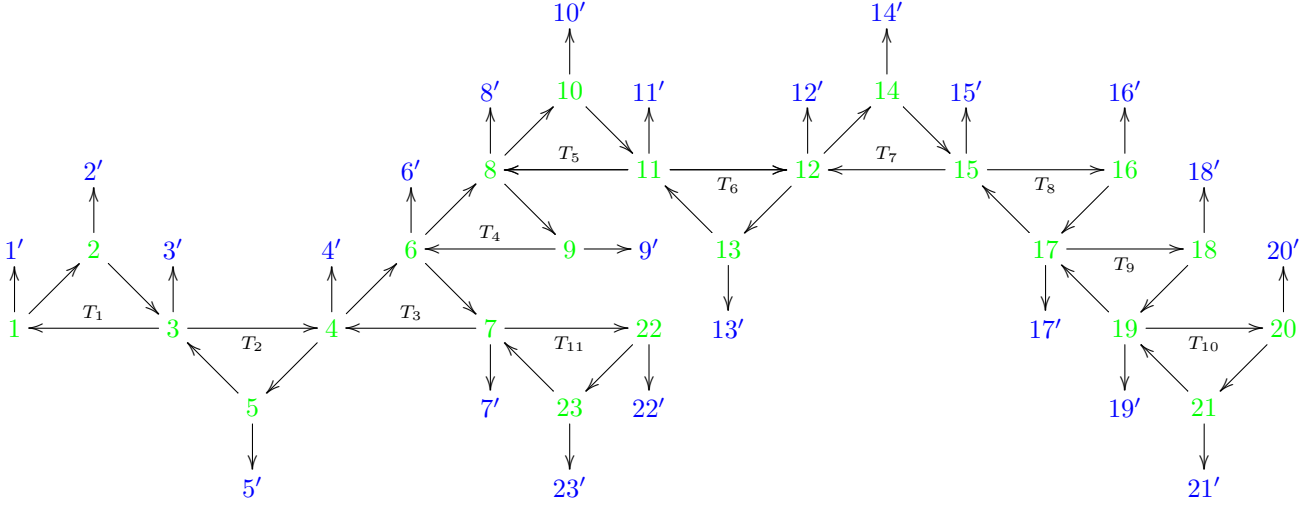

\end{example}

\section{Associated Mutation Sequences} \label{Sec:MutSeq}

Throughout this section we work with a given signed irreducible type $\mathbb{A}$ quiver $\mathcal{Q}$ with respect to a fixed root 3-cycle $S$. Based on the data defining the signed irreducible type $\mathbb{A}$ quiver $\mathcal{Q}$, we construct a mutation sequence of $Q$ that we will call the $\textbf{associated mutation sequence}$ of $\mathcal{Q}$. After that we state our main theorem which says that the associated mutation sequence of $\mathcal{Q}$ is a maximal green sequence (see Theorem~\ref{main2}). We then apply our main theorem to construct a maximal green sequence for any type $\mathbb{A}$ quiver $Q$ (see Corollary~\ref{anytypeA}).

\subsection{Definition of Associated Mutation Sequences} \label{Sec:Irr}

Before defining the associated mutation sequence of $\mathcal{Q}$, we need to develop some terminology. 

\begin{definition}\label{specialT}
Let $T_k$ be a signed 3-cycle of $\mathcal{Q}$. Define the sequence of vertices $(x({0,k}),x({1,k}), \ldots, x({d,k}))$ of $\mathcal{Q}$ by
$$\begin{array}{rcl}
x({j,k}) & := & \left\{\begin{array}{rcl}z_k &:& \text{if $j = 0$,}\\
t(\gamma_{m_j}) &:& \text{$\gamma_{m_j}$ is the unique arrow of $\mathcal{Q}$ satisfying $s(\gamma_{m_j}) = x({{j-1,k}})$}.
\end{array}\right.
\end{array}$$

\noindent Note that such a sequence is necessarily finite, and we choose $d$ to be maximal, or equivalently so that \text{sgn}$(S_{m_d}) = +$. When $k$ is clear from context, we abbreviate $x(s,k)$ as $x(s)$.  It follows from the definition of $x(j)$ that $x(j) = x_{m_j}$ for any $j \in [d]$, and that $x(d) = x_1$ or $y_{m_d-1}$. However, $x(0)$ can be expressed as $x_s$ for some $s \in [n]$ only if $\deg(x(0)) = \deg(z_k) = 4.$ See Figure~\ref{special}. Note that if sgn$(S_k)=+$, then this sequence of vertices is simply $(x(0),x(1))$.

\end{definition}

\begin{figure}
$$\begin{array}{c}
\xymatrix
{
& & y_{m_d} \ar@{->}[dr] & &\\
& \underbracket{x(d)}_{=\text{tr}(y_k)} \ar@{->}[ur] & & x({d-1})\ar@{->}[ll]_{T_{m_d}}^{\gamma_{m_d}}\ar@{->}[rr]_{T_{m_{d-1}}} & & y_{m_{d-1}} \ar@{->}[dl]\\
&  & & & x({d-2}) \ar@{->}^{\gamma_{m_{d-1}}}[ul]\\
& & & & & \ddots\\
& & & & & & x({1}) \ar@{->}[rr]_{T_k=T_{m_1}} & & y_{k} \ar@{->}[dl]\\
& & & & & & & x({0})\ar@{->}^{\gamma_{m_1}}[ul]\\
}
\end{array}$$
\caption{The sequence $(x(0), x(1), \ldots, x(d))$ defined by $T_k$ where $\text{sgn}(S_k) = -.$ The transport of $y_k$ is also illustrated for quivers where there is no sequence of the form in (\ref{alphabetaseq}) of Definition \ref{transportDEF}.}
\label{special}
\end{figure}
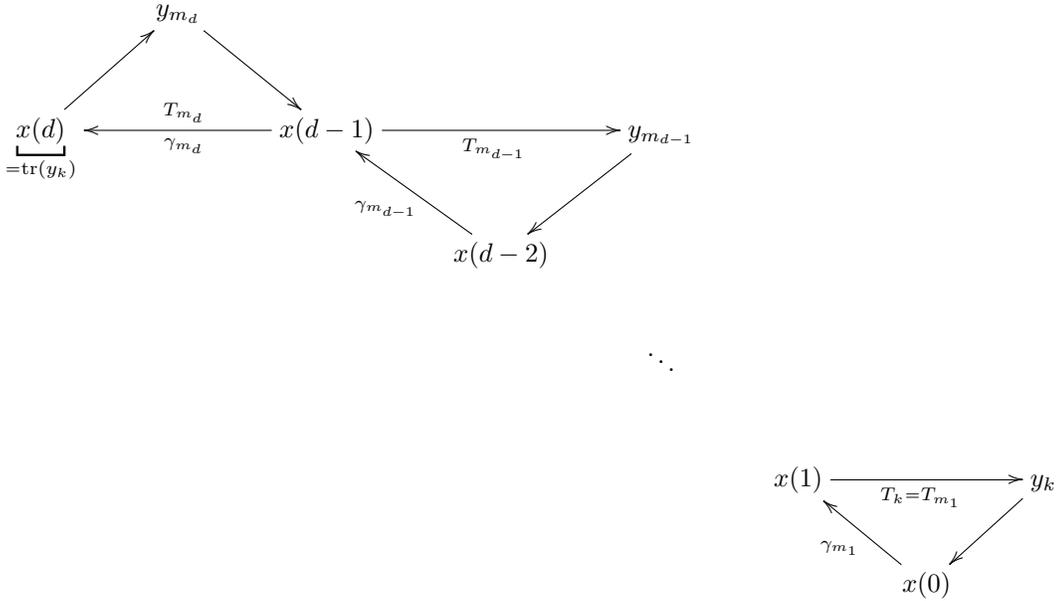

\begin{definition}\label{transportDEF}
For any vertex $v$ of $Q$ which can be expressed as $v=y_k$, i.e. as a point of some signed 3-cycle $T_k$ of $\mathcal{Q}$, we define the {\bf transport} of $y_k$ by the following procedure.  We will denote the image of the transport as $\text{tr}(v)$.  Consider the full subquiver of $\mathcal{Q}$ on the vertices of the signed $3$-cycles $T_1$,$T_2$,$\dots$,$T_k$, which we denote by $\mathcal{Q}_k$.  Inside this subquiver,  

\begin{itemize}
\item[$i)$]  move from $y_k$ along $\beta_k$ to $t(\beta_k)$,
\item[$ii)$] move from $t(\beta_k)$ along the sequence of arrows $\gamma_{m_1}, \gamma_{m_2}, \ldots, \gamma_{m_d}$ of maximal length to $t(\gamma_{m_d})$ where the integers $\{m_i\}_{i \in [d]}$ are those defined by the signed 3-cycle $T_k$ (see Definition~\ref{specialT}),
\item[$iii)$] if possible, move from $t(\gamma_{m_d})$ to $\text{tr}(y_k) := t(\beta_{k_s})$ along the sequence of arrows of the form shown in (\ref{alphabetaseq}) each of which belongs to a signed 3-cycle $T_i$ for some $i < k$, under the assumption that the
subsequences $A_1$ and $A_2$ are of maximal length, and $A_2$ must be nonempty.  If no such sequence exists of this form, we instead define $\text{tr}(y_k) := t(\gamma_{m_d})$.
\end{itemize}

\begin{eqnarray}\text{$\underbracket{\alpha_{k_1},\beta_{k_1},\alpha_{k_2},\beta_{k_2}, \ldots, \alpha_{k_{\ell-1}}, \beta_{k_{\ell-1}}}_{A_1}, \alpha_{k_\ell},\underbracket{\alpha_{k_{\ell+1}}, \beta_{k_{\ell+1}}, \alpha_{k_{\ell+2}}, \beta_{k_{\ell+2}}, \ldots, \alpha_{k_{s}}, \beta_{k_s}}_{A_2}$} \label{alphabetaseq} \end{eqnarray}

\noindent See Figures~\ref{special}, \ref{st1}, and \ref{st2}.
\end{definition}

\begin{figure}
$$\begin{xy} 0;<1pt,0pt>:<0pt,-1pt>:: 
(0,30) *+{x({d})} ="0",
(30,0) *+{y_{k_1}} ="1",
(60,30) *+{x{(d-1)}} ="2",
(240,90) *+{y_{k_{\ell+1}}} ="3",
(270,120) *+{x_{k_{\ell+2}}} ="4",
(120,30) *+{y_{k_2}} ="5",
(90,60) *+{x({d-2})} ="6",
(120,90) *+{\ddots} ="7",
(270,240) *+{x(0)} ="8",
(240,210) *+{x(1)} ="9",
(300,210) *+{y_{k}} ="10",
(330,120) *+{y_{{k_{\ell+2}}}} ="11",
(300,150) *+{x_{{k_{\ell+3}}}} ="12",
(330,180) *+{\ddots} ="13",
(390,240) *+{z_{k_s}} ="14",
(360,210) *+{x_{k_s}} ="15",
(420,210) *+{y_{k_s}} ="16",
(210,120) *+{x_{k_{\ell+1}}} ="17",
(150,120) *+{x(i)} ="18",
(180,150) *+{x(i-1)} ="19",
(210,180) *+{\ddots} ="20",
(390,250) *+{\equalto{}{}} ="21",
(390,260) *+{\text{tr}(y_k)} ="22",
"0", {\ar^{\alpha_{k_1}}"1"},
"2", {\ar^{\gamma_{m_d}}_{T_{m_d}=T_{k_1}}"0"},
"1", {\ar^{\beta_{k_1}}"2"},
"2", {\ar^{\alpha_{k_2}}_{T_{m_{d-1}}}"5"},
"6", {\ar^{\gamma_{m_{d-1}}}"2"},
"3", {\ar^{\beta_{k_{\ell+1}}}"4"},
"17", {\ar^{\alpha_{k_{\ell+1}}}"3"},
"4", {\ar^{\alpha_{k_{\ell+2}}}_{T_{{k_{\ell+2}}}}"11"},
"12", {\ar"4"},
"4", {\ar_{T_{k_{\ell+1}}}"17"},
"5", {\ar^{\beta_{k_2}}"6"},
"8", {\ar^{\gamma_{m_1}}"9"},
"10", {\ar^{\beta_k}"8"},
"9", {\ar_{T_{k} = T_{m_1}}"10"},
"11", {\ar^{\beta_{k_{\ell+2}}}"12"},
"14", {\ar"15"},
"16", {\ar^{\beta_{k_{s}}}"14"},
"15", {\ar^{\alpha_{k_{s}}}_{T_{k_s}}"16"},
"18", {\ar^{\alpha_{k_\ell}}_{T_{m_i} = T_{k_\ell}}"17"},
"17", {\ar"19"},
"19", {\ar^{\gamma_{m_i}}"18"},
\end{xy}$$
\caption{The sequence of arrows one follows to compute the transport of $y_k$. Note that in this case, the sequence $A_1$ is non-empty.}
\label{st1}
\end{figure}
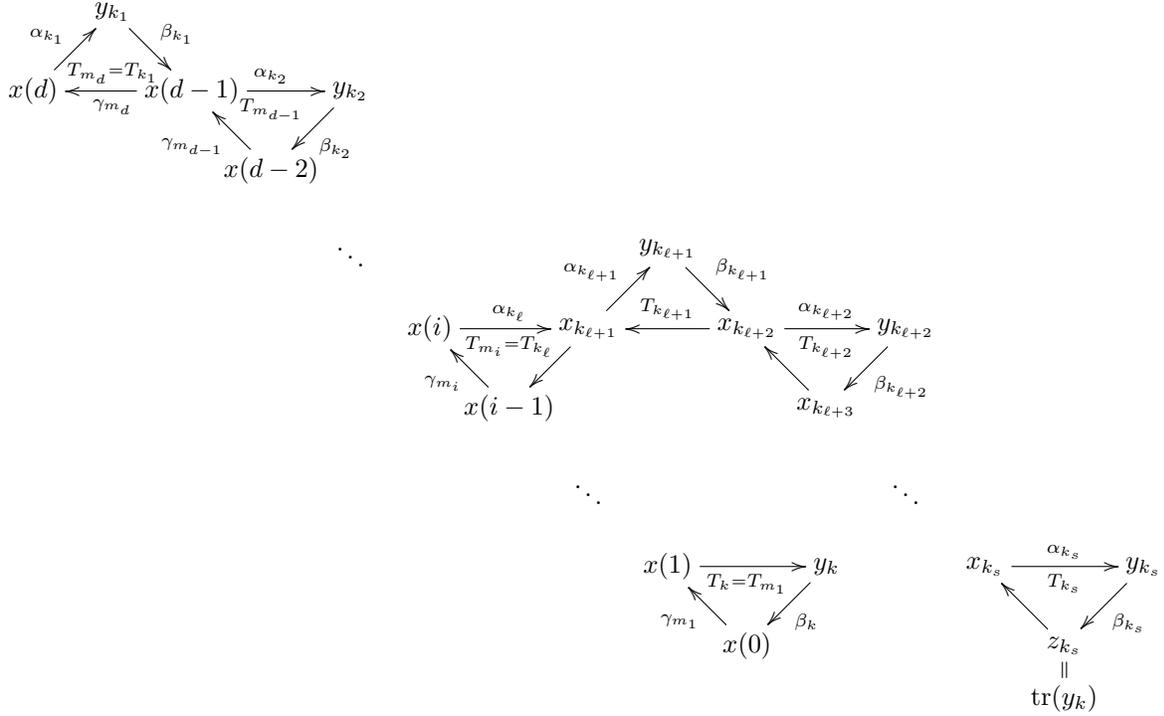

\begin{figure}[t]
$$\begin{xy} 0;<1pt,0pt>:<0pt,-1pt>:: 
(-10,65) *+{x(d)} ="0",
(30,30) *+{x_{k_2}} ="1",
(70,65) *+{x(d-1)} ="2",
(75,-10) *+{y_{k_2}} ="3",
(120,30) *+{x_{k_3}} ="4",
(140,65) *+{y_{m_{d-1}}} ="5",
(105,100) *+{x({d-2})} ="6",
(110,120) *+{\vdots} ="7",
(125,150) *+{x({1})} ="8",
(195,150) *+{y_{k}} ="9",
(160,185) *+{x(0)} ="10",
(210,30) *+{y_{k_3}} ="11",
(165,70) *+{x_{k_4}} ="12",
(190,90) *+{\ddots} ="13",
(210,120) *+{x_{k_s}} ="14",
(300,120) *+{y_{k_s}} ="15",
(255,160) *+{z_{k_s}} ="16",
(255,170) *+{\equalto{}{}} ="21",
(255,180) *+{\text{tr}(y_k)} ="22",
"0", {\ar^{\alpha_{k_1}}"1"},
"2", {\ar^{\gamma_{m_d}}_{T_{m_d} = T_{k_1}}"0"},
"1", {\ar"2"},
"1", {\ar^{\alpha_{k_2}}"3"},
"4", {\ar_{T_{k_2}}"1"},
"2", {\ar_{T_{m_{d-1}}}"5"},
"6", {\ar^{\gamma_{m_{d-1}}}"2"},
"3", {\ar^{\beta_{k_2}}"4"},
"4", {\ar^{\alpha_{k_3}}_{T_{k_3}}"11"},
"12", {\ar"4"},
"5", {\ar"6"},
"8", {\ar_{T_{k}=T_{m_1}}"9"},
"10", {\ar^{\gamma_{m_1}}"8"},
"9", {\ar^{\beta_k}"10"},
"11", {\ar^{\beta_{k_3}}"12"},
"14", {\ar^{\alpha_{k_s}}_{T_{k_s}}"15"},
"16", {\ar"14"},
"15", {\ar^{\beta_{k_s}}"16"},
\end{xy}$$
\caption{The sequence of arrows one follows to compute the transport of $y_k$. Note that in this case, the sequence $A_1$ is empty.}
\label{st2}
\end{figure}
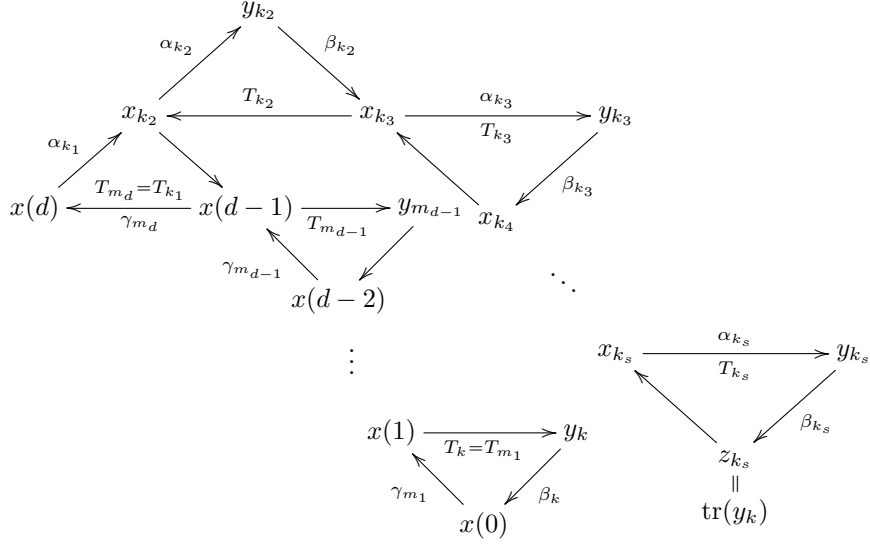

We now use the above notation to define the \textbf{associated mutation sequence} of $\mathcal{Q}$.

\begin{definition}
Let $\mathcal{Q} = (Q,S,\{T_i\}_{i \in [n]})$ be a signed irreducible type $\mathbb{A}$ quiver. Define $\underline{\mu}_0 :=\mu_{x_1}.$ For each $k \in [n]$ we define a sequence of mutations, denoted $\underline{\mu}_k$, as follows. Note that when we write $\emptyset$ below we mean the empty mutation sequence. We define

$$\underline{\mu}_k := \underline{\mu}_A\circ \underline{\mu}_B\circ \underline{\mu}_C\circ \underline{\mu}_D$$ where $\underline{\mu}_A, \underline{\mu}_B, \underline{\mu}_C,$ and $\underline{\mu}_D$ are mutation sequences defined in the following way
$$\begin{array}{rrlll}
\underline{\mu}_D & := & \mu_{y_{k}} \nonumber \\
\underline{\mu}_C & := & \mu_{x({{d-1}})}\circ\cdots\circ\mu_{x({1})}\circ\mu_{x({0})}
\\
\underline{\mu}_B & := & \left\{
            \begin{array}{lrl}
            \mu_{\text{tr}(x(d))}& : & \text{if $x(d) \neq x_1$}\\
            \emptyset  & : & \text{if $x(d) = x_1$}\\
             \end{array}\right. \nonumber \\
\underline{\mu}_A & := & \mu_{\text{tr}(y_k)}.
\end{array}$$
Note that $x(d)=x_1$ or $y_{m_d-1}$ so the transport $\text{tr}(x(d))$ in $\underline{\mu}_B$ is well-defined.
Now define the \textbf{associated mutation sequence} of $\mathcal{Q}$ to be $\underline{\mu}:= \underline{\mu}_n\circ \cdots \circ \underline{\mu}_1\circ\underline{\mu}_0.$ We will denote the associated mutation sequence of $\mathcal{Q}$ by $\underline{\mu}$ or by ${\underline{\mu}}^{\mathcal{Q}}$ if it is not clear from context which signed irreducible type $\mathbb{A}$ quiver defines $\underline{\mu}$. At times it will be useful to write $\underline{\mu}_k = \underline{\mu}_{A(k)}\circ \underline{\mu}_{B(k)}\circ \underline{\mu}_{C(k)}\circ \underline{\mu}_{D(k)}$.
\end{definition}

\begin{figure}
$$\begin{xy} 0;<1pt,0pt>:<0pt,-1pt>:: 
(50,50) *+{3} ="0",
(125,0) *+{6} ="1",
(200,50) *+{8} ="2",
(225,100) *+{14} ="3",
(150,50) *+{7} ="4",
(175,100) *+{9} ="5",
(200,150) *+{15} ="6",
(225,200) *+{17} ="7",
(250,150) *+{16} ="8",
(275,200) *+{18} ="9",
(25,0) *+{2} ="10",
(225,0) *+{10} ="11",
(250,50) *+{11} ="12",
(300,150) *+{20} ="13",
(325,200) *+{21} ="14",
(75,100) *+{5} ="15",
(0,50) *+{1} ="16",
(250,250) *+{19} ="17",
(100,50) *+{4} ="18",
(275,100) *+{13} ="19",
(300,50) *+{12} ="20",
(325,100) *+{22} ="21",
(350,150) *+{23} ="22",
(350,50) *+{24} ="23",
(375,100) *+{25} ="24",
(425,100) *+{26} ="25",
(400,150) *+{27} ="26",
(375,200) *+{30} = "27",
(350, 250) *+{31} = "28",
(425, 200) *+{29} = "29",
(450, 150) *+{28} = "30",
"10", {\ar"0"},
"15", {\ar"0"},
"0", {\ar_{T_1}"16"},
"0", {\ar_{T_2}"18"},
"1", {\ar"4"},
"18", {\ar"1"},
"4", {\ar_{T_4}"2"},
"2", {\ar"5"},
"2", {\ar"11"},
"12", {\ar_{T_5}"2"},
"5", {\ar_{T_7}"3"},
"3", {\ar"6"},
"5", {\ar"4"},
"4", {\ar_{T_3}"18"},
"6", {\ar"5"},
"7", {\ar"6"},
"6", {\ar_{T_8}"8"},
"8", {\ar"7"},
"7", {\ar_{T_9}"9"},
"17", {\ar"7"},
"9", {\ar"13"},
"14", {\ar_{T_{10}}"9"},
"9", {\ar"17"},
"16", {\ar"10"},
"11", {\ar"12"},
"19", {\ar"12"},
"12", {\ar_{T_6}"20"},
"13", {\ar"14"},
"13", {\ar"21"},
"22", {\ar_{T_{11}}"13"},
"18", {\ar"15"},
"20", {\ar"19"},
"21", {\ar"22"},
"21", {\ar"23"},
"24", {\ar_{T_{12}}"21"},
"23", {\ar"24"},
"24", {\ar_{T_{13}}"25"},
"26", {\ar"24"},
"25", {\ar"26"},
"14", {\ar_{T_{15}}"27"},
"27", {\ar"28"},
"28", {\ar"14"},
"29", {\ar"26"},
"26", {\ar_{T_{14}}"30"},
"30", {\ar"29"},
\end{xy}$$
\caption{The signed irreducible type $\mathbb{A}_{31}$ quiver described in Example~\ref{assocmutexamples}.}
\label{assocmutquiver}
\end{figure}
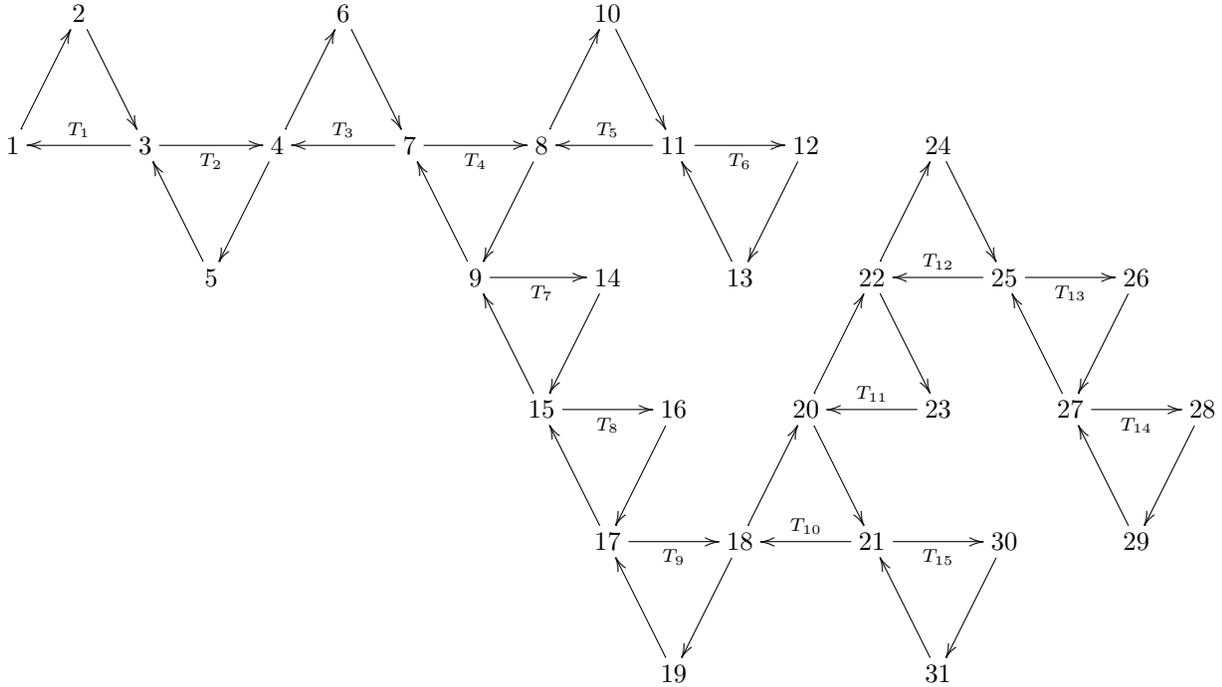

\begin{example}\label{assocmutexamples}
Let $\mathcal{Q}$ denote the signed irreducible type $\mathbb{A}_{31}$ quiver appearing in Figure~\ref{assocmutquiver}. In the table in Figure~\ref{assocmutatseq31}, we 
describe $\underline{\mu}_i$ for each $0 \leq i \leq 15$. Thus, the associated mutation sequence defined by $\mathcal{Q}$ is $\underline{\mu}_{15}\circ\underline{\mu}_{14}\circ \cdots \circ \underline{\mu}_{1}\circ \underline{\mu}_{0}$.

\begin{figure}
$$\begin{tabular}{llcll} 
$i$ & $\underline{\mu}_i$ & ~ & $i$ & $\underline{\mu}_i$\\
\hline
0  & $\mu_1$ & ~ & 8  & $\mu_{13}\circ\mu_{1}\circ\mu_{7}\circ\mu_9\circ\mu_{15}\circ\mu_{17}\circ\mu_{16}$ \\
1  & $\mu_1\circ \mu_3\circ \mu_2$ & ~ & 9 & $\mu_{13}\circ\mu_{1}\circ\mu_{7}\circ\mu_9\circ\mu_{15}\circ\mu_{17}\circ\mu_{19}\circ\mu_{18}$ \\
2  & $\mu_1\circ\mu_{3}\circ\mu_5\circ\mu_4$ & ~ & 10 & $\mu_{18}\circ\mu_{13}\circ\mu_{21}\circ\mu_{20}$ \\
3  & $\mu_4\circ\mu_{1}\circ\mu_7\circ\mu_6$ & ~ & 11 & $\mu_{20}\circ\mu_{18}\circ\mu_{23}\circ\mu_{22}$ \\
4  & $\mu_4\circ\mu_{1}\circ\mu_{7}\circ\mu_9\circ\mu_8$ & ~ & 12 & $\mu_{22}\circ\mu_{20}\circ\mu_{25}\circ\mu_{24}$ \\
5  & $\mu_8\circ\mu_{4}\circ\mu_{11}\circ\mu_{10}$ & ~ & 13 & $\mu_{22}\circ\mu_{20}\circ\mu_{25}\circ\mu_{27}\circ\mu_{26}$\\
6  & $\mu_8\circ\mu_{4}\circ\mu_{11}\circ\mu_{13}\circ\mu_{12}$ & ~ & 14 & $\mu_{22}\circ\mu_{20}\circ\mu_{25}\circ\mu_{27}\circ\mu_{29}\circ\mu_{28}$ \\
7  & $\mu_{13}\circ\mu_{1}\circ\mu_{7}\circ\mu_9\circ\mu_{15}\circ\mu_{14}$ & ~ & 15 & $\mu_{23} \circ \mu_{13} \circ \mu_{21}\circ\mu_{31}\circ\mu_{30}$
\end{tabular}$$\caption{The associated mutation of the signed irreducible type $\mathbb{A}_{31}$ quiver in Figure~\ref{assocmutquiver}.}\label{assocmutatseq31}\end{figure}

\end{example}

We now arrive at the main result of this paper.

\begin{theorem}\label{main2}
If $\mathcal{Q} = (Q,S,\{T_i\}_{i \in [n]})$ is a signed irreducible type $\mathbb{A}$ quiver with associated mutation sequence $\underline{\mu}$, then we have $\underline{\mu} \in \text{green}\left(Q\right)$.
\end{theorem}

We present the proof Theorem~\ref{main2} in the next section, as the argument requires some additional tools.

\begin{remark}\label{mubarlength}
For a given irreducible type $\mathbb{A}$ quiver with at least one 3-cycle, the length of $\underline{\mu}$ can vary depending on the choice of leaf 3-cycle. Let $Q$ denote the irreducible type $\mathbb{A}_7$ quiver shown in Figure~\ref{A7}. By choosing the 3-cycle 1,2,3 (resp. 5,6,7) to be the root 3-cycle, one obtains the signed irreducible type $\mathbb{A}$ quiver $\mathcal{Q}_1$ (resp. $\mathcal{Q}_2$) shown in Figure~\ref{A7signed}. Then the associated mutations of $\mathcal{Q}_1$ and $\mathcal{Q}_2$ are \begin{eqnarray}
\underline{\mu}^{\mathcal{Q}_1} & = & \mu_1\circ\mu_3\circ\mu_5\circ\mu_7\circ\mu_6\circ\mu_1\circ\mu_3\circ\mu_5\circ\mu_4\circ\mu_1\circ\mu_3\circ\mu_2\circ\mu_1 \nonumber \\
\underline{\mu}^{\mathcal{Q}_2} & = & \mu_3\circ\mu_6\circ\mu_2\circ\mu_1\circ\mu_6\circ\mu_5\circ\mu_4\circ\mu_3\circ\mu_6\circ\mu_5\circ\mu_7\circ\mu_6. \nonumber
\end{eqnarray}

\noindent Furthermore, the maximal green sequence produced by Theorem~\ref{main2}, i.e. the associated mutation sequence of a each signed irreducible type $\mathbb{A}$ quiver associated to $Q$, is not necessarily a minimal length maximal green sequence. For example, it is easy to check that $\underline{\nu} = \mu_3\circ\mu_1\circ\mu_4\circ\mu_3\circ\mu_7\circ\mu_6\circ\mu_2\circ\mu_5\circ\mu_1\circ\mu_4\circ\mu_7$ is a maximal green sequence of $Q$, which is of length less than that of $\underline{\mu}^{\mathcal{Q}_1}$ or $\underline{\mu}^{\mathcal{Q}_2}$. 
\end{remark}

\begin{remark}\label{cormierresult1}
While we were revising this paper, Cormier, Dillery, Resh, Serhiyenko, and Whelan \cite{CDRSW} found a construction of minimal length maximal green sequences for type $\mathbb{A}$ quivers.  Therein, they construct a maximal green sequence for any irreducible type $\mathbb{A}$ quiver $Q$ with at least one 3-cycle by mutating first at all leaf 3-cycles of $Q$, then mutating at the 3-cycles connected to the leaf 3-cycles of $Q$, continuing this process, and then mutating a subsequence of the vertices in reverse.  This contrasts with the maximal green sequences we construct in this paper, which involve some extraneous steps but whose process can be defined locally and inductively, akin to writing down the reduced word for a permutation using bubble sort.
\end{remark}

\begin{figure}[h]
$$\begin{array}{rcl} \raisebox{-.2in}{$Q$} & \raisebox{-.2in}{=} & \begin{xy} 0;<1pt,0pt>:<0pt,-1pt>:: 
(20,0) *+{2} ="0",
(60,0) *+{4} ="1",
(100,0) *+{6} ="2",
(0,20) *+{1} ="3",
(40,20) *+{3} ="4",
(80,20) *+{5} ="5",
(120,20) *+{7.} ="6",
"3", {\ar"0"},
"0", {\ar"4"},
"4", {\ar"1"},
"1", {\ar"5"},
"5", {\ar"2"},
"2", {\ar"6"},
"4", {\ar"3"},
"5", {\ar"4"},
"6", {\ar"5"},
\end{xy}\end{array}$$
\caption{}
\label{A7}
\end{figure}

\begin{figure}[h]
$$\begin{xy} 0;<1pt,0pt>:<0pt,-1pt>:: 
(20,0) *+{2} ="0",
(80,20) *+{4} ="1",
(100,40) *+{6} ="2",
(0,20) *+{1} ="3",
(-20,20) *+{=} ="14",
(-40, 20) *+{\mathcal{Q}_1} ="15",
(40,20) *+{3} ="4",
(60,40) *+{5} ="5",
(80,60) *+{7} ="6",
(120,20) *+{6} ="7",
(140,0) *+{7} ="8",
(160,20) *+{5} ="9",
(200,20) *+{3} ="10",
(180,40) *+{4} ="11",
(220,0) *+{1} ="12",
(240,20) *+{2} ="13",
(260,20) *+{=} ="16",
(280, 20) *+{\mathcal{Q}_2} ="17",
"3", {\ar"0"},
"0", {\ar"4"},
"4", {\ar"1"},
"1", {\ar"5"},
"5", {\ar"2"},
"2", {\ar"6"},
"4", {\ar"3"},
"5", {\ar"4"},
"6", {\ar"5"},
"7", {\ar"8"},
"9", {\ar"7"},
"8", {\ar"9"},
"9", {\ar"10"},
"11", {\ar"9"},
"10", {\ar"11"},
"10", {\ar"12"},
"13", {\ar"10"},
"12", {\ar"13"},
\end{xy}$$
\caption{The two signed irreducible type $\mathbb{A}$ quivers that can be obtained from $Q$.}
\label{A7signed}
\end{figure}
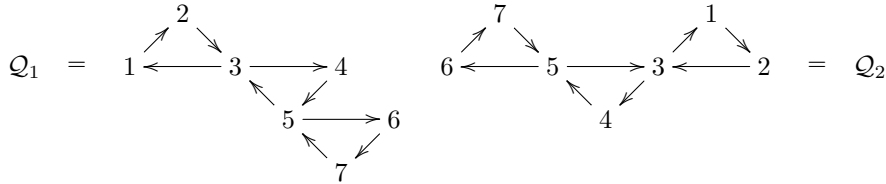

We conclude this section by using Theorem~\ref{main2} to show that any type $\mathbb{A}$ quiver has at least one maximal green sequence.

\begin{corollary}\label{anytypeA}
Let $Q \in \text{Mut}(1 \to 2 \to \cdots \to n).$ Then $Q$ has a maximal green sequence.
\end{corollary}
\begin{proof}
By Corollary~\ref{cor:tsurf}, $Q$ can be expressed as a direct sum of irreducible type $\mathbb{A}$ quivers $\{Q_1, Q_2, \ldots, Q_k\}.$ In other words, 
$$Q = Q_1 \oplus_{(a_{(1,1)},a_{(2,1)},\dots,a_{(d_1,1)})}^{(b_{(1,1)},b_{(2,1)},\dots, b_{(d_1,1)})}Q_2^\prime \mathrm{~where~} Q_j^\prime = Q_j \oplus_{(a_{(1,j)},a_{(2,j)},\dots,a_{(d_j,j)})}^{(b_{(1,j)},b_{(2,j)},\dots, b_{(d_j,j)})}Q_{j+1}^\prime \mathrm{~for~} 2 \leq j \leq k-1, \mathrm{~and~}Q_{k}^\prime = Q_k.$$

If $Q_i$ is of type $\mathbb{A}_1$ and $a_i$ denotes the unique vertex of $Q_i$, then $\underline{\mu}^{(i)} := \mu_{a_i}$ is a maximal green sequence of $Q_i$. {If $Q_i$ is not of type $\mathbb{A}_1$, then we form a signed irreducible type $\mathbb{A}$ quiver, $\mathcal{Q}^{(i)}$, associated to $Q_i$ by picking a leaf $3$-cycle.} Now by Theorem~\ref{main2}, the associated mutation sequence of $\mathcal{Q}^{(i)}$, denoted $\underline{\mu}^{(i)}$, is a maximal green sequence of $Q_i$. By applying Proposition \ref{tcolordirsum} iteratively, we obtain $\underline{\mu} = \underline{\mu}^{(k)}\circ \cdots \circ \underline{\mu}^{(2)}\circ \underline{\mu}^{(1)}$ is a maximal green sequence of $\widehat{Q}$.
\end{proof}

\section{Proof of Theorem~\ref{main2}} \label{Sec:Prep}

In this section, we work with a fixed signed irreducible type $\mathbb{A}$ quiver $\mathcal{Q} = (Q,S,\{T_i\}_{i \in [n]})$ with $N$ vertices. We write $\underline{\mu} = \underline{\mu}_n \circ \cdots \circ \underline{\mu}_1 \circ \underline{\mu}_0$ for the associated mutation sequence of $\mathcal{Q}$. 

\begin{definition} \label{def:associatedperm} For each $\underline{\mu}_i$ appearing in $\underline{\mu}$ we define a permutation $\tau_i\in \mathfrak{S}_{(\mathcal{Q})_0}\cong \mathfrak{S}_N$  where $\mathfrak{S}_{(\mathcal{Q})_0}$ denotes the symmetric group on the vertices of $\mathcal{Q}$. In the special case where $i = 0$, we define $\tau_0$ to be the identity permutation. Then for $i \in [n]$ where $\underline{\mu}_i = \mu_{i_d}\circ\cdots\circ\mu_{i_1}$ we define $\tau_i := (i_2, \ldots, i_d)$ in cycle notation (i.e. $i_j\cdot \tau_i = i_{j+1}$ for $j \in [d-1]$ and $i_d \cdot \tau_i = i_2$).  Note that $i_1 = y_i$. We also define
\begin{eqnarray}
\sigma_i & := & \tau_i\cdots\tau_1\tau_0 \nonumber \\
& = & \tau_i \cdots \tau_1 \nonumber
\end{eqnarray}
where the last equality holds since $\tau_0$ is the identity permutation.  We say that $\sigma_n$ is the \textbf{associated permutation} corresponding to $\mathcal{Q}$.
\end{definition}

Theorem~\ref{main2} will imply that the associated permutation $\sigma_n$ is exactly the permutation induced by $\underline{\mu}$ (see the last paragraph of Section~\ref{Sec:Prelim}).

Let $T_k$ and $T_t$ where $k \leq t$ be signed 3-cycles of $\mathcal{Q}$. Let $\mathcal{Q}_{k,t}$ denote the full subquiver of $\mathcal{Q}$ on the vertices of $T_1, \ldots, T_k$ and the vertices of $T_{m_1}, \ldots, T_{m_d}$ where the integers $m_1, \ldots, m_d \in [n]$ are those defined by $T_t$ as in Definition~\ref{specialT}. For example, $\mathcal{Q}_{k,k}$ is the full subquiver of $\mathcal{Q}$ on the vertices of the signed 3-cycles $T_1, \ldots, T_k$. By convention, we also define $\mathcal{Q}_{0,0}$ to be the full subquiver of $\mathcal{Q}$ consisting of only the vertex $x_1$. Now define $\text{tr}|_{{k,t}}$ to be the restriction of the transport to $\mathcal{Q}_{k,t}$.

\begin{lemma}\label{newmainlemma}
For each $k \in [n]$ there is an ice quiver $\overline{R}_k$ that is a full subquiver of $\underline{\mu}_{k-1}\circ \cdots \circ \underline{\mu}_1\circ \underline{\mu}_0(\widehat{Q})$ of the form shown in Figure~\ref{importantfullsub1} (resp. Figure~\ref{importantfullsub2}) where the vertices $z_k = x(0), x(1), \ldots, x(d-1), \text{tr}(x(d))$, and $\text{tr}(y_k)$ (resp. $z_k = x(0), x(1), \ldots, x(d-1),$ and $\text{tr}(y_k)$) are those appearing in the mutation sequence $\underline{\mu}_{A(k)}\circ \underline{\mu}_{B(k)}\circ \underline{\mu}_{C(k)}$  and the integers $m_1, m_2, \ldots, m_d$ are those defined by $T_k$ in Definition~\ref{specialT}. Recall that we only mutate at $\text{tr}(x(d))$ if $x(d) \neq x_1$. Furthermore, the ice quiver $\overline{R}_k$ has the following properties: 
\begin{itemize}\item $\overline{R}_k$ includes every frozen vertex that is connected to a mutable vertex appearing in Figure~\ref{importantfullsub1} (resp. Figure~\ref{importantfullsub2}) by at least one arrow in $\underline{\mu}_{k-1}\circ \cdots \circ \underline{\mu}_1 \circ \underline{\mu}_0(\widehat{Q})$ where $\widetilde{x}(1) := z^\prime_{m_d-1},\ \widetilde{\text{tr}}(x(d)) := x^\prime_{m_2}, \widetilde{\text{tr}}(y_k) := x^\prime_{m_1}$ and 
$\widetilde{x}(s) := x^\prime_{m_{d-s+2}}$ for\footnote{If $d = 1$, then $\widetilde{\text{tr}}(x(d)) = x^\prime_{m_1} = x_k$.} $s \in [2,d-1]$ (resp. $\widetilde{\text{tr}}(y_k) := x^\prime_{m_1}$ and $\widetilde{x}(s) := x^\prime_{m_{d-s+1}}$ for\footnote{If $d = 1$, then $\widetilde{\text{tr}}(x(d)) = x^\prime_{m_1} = x_k$. Furthermore, $d=1$, in this case, if and only if $k = 1$.} $s \in [1,d-1]$), \item vertices $y_m, y^\prime_m, z_m,$ and $z_m^\prime$ appear in $\overline{R}_k$ if and only if $\deg(y_k) =4$ in $Q$, 
\item vertices $y_\ell, y^\prime_\ell, z_\ell,$ and $z_\ell^\prime$ appear in $\overline{R}_k$ if and only if $\deg(z_k) =4$ in $Q$,
\item vertices $y_t$ and $y_t^\prime$ appear in $\overline{R}_k$ if and only if there exists a signed 3-cycle $T_t$ in $\mathcal{Q}$ with $k < t$ such that $\text{tr}|_{k,t}(y_t) = z_k$ and such that in $\underline{\mu}_{k-1}\circ \cdots \circ \underline{\mu}_1\circ \underline{\mu}_0(\widehat{Q})$ the vertex $x_t$ has been mutated exactly once\footnote{Note that this can only happen if there exists $j < k$ such that $z_j=x_t$ and $x(d,k)=y_j$ as in Definition \ref{specialT}.}, and
\item in Figure~\ref{importantfullsub1} (resp. Figure~\ref{importantfullsub2}) $C_{1,k} := x_{m_d -1}\cdot\sigma_{k-1}^{-1}$, $\widetilde{C_{1,k}} := x^\prime_{m_d -1},$ $C_{s,k} := y_{m_j}\cdot\sigma_{k-1}^{-1},$ and $\widetilde{C_{s,k}} := y_{m_j}^\prime$ for $s \in [2,d]$ and $j = d-s+2$ (resp. $C_{s,k} := y_{m_j}\cdot \sigma_{k-1}^{-1}$ and $\widetilde{C_{s,k}} := y_{m_j}^\prime$ for $s \in [1,d-1]$ and $j = d-s+2$).\end{itemize}

\noindent {Additionally, for each $k \in [n]$ we have $\underline{\mu}_k \circ \cdots \circ \underline{\mu}_1\circ \underline{\mu}_0(\widehat{\mathcal{Q}_{k,k}}) = \widecheck{\mathcal{Q}_{k,k}}\cdot \sigma_k.$}
\end{lemma}

We will prove Lemma~\ref{newmainlemma} in the case where the vertex $\text{tr}(x(d))$ appears in the mutation sequence $\underline{\mu}_k$ (i.e. when $x(d) \neq x_1$). Under this assumption, the following lemma will allow us to prove Lemma~\ref{newmainlemma} inductively. The proof of Lemma~\ref{newmainlemma} when $\text{tr}(x(d))$ does not appear in $\underline{\mu}_k$ is very similar so we omit it.

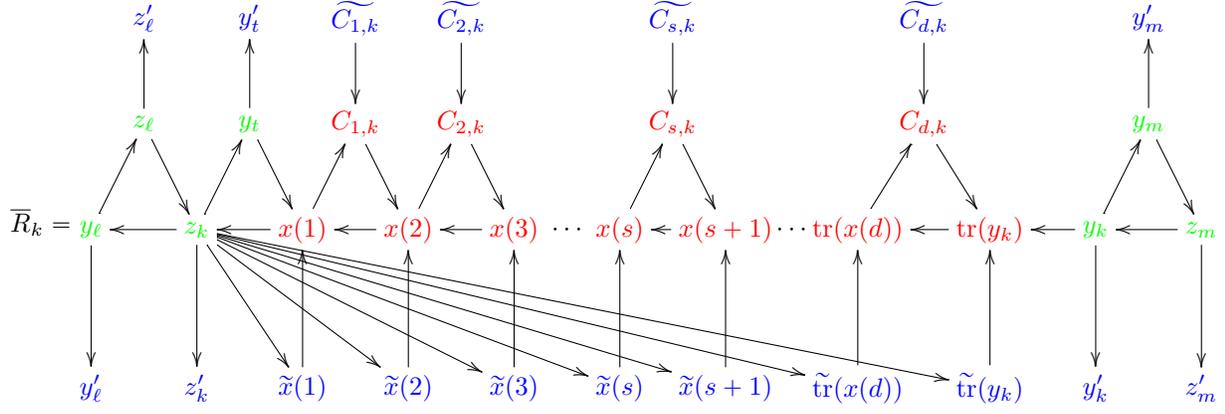
\begin{figure}
$$\raisebox{-.53in}{$\overline{R}_k =$} \begin{xy} 0;<1pt,0pt>:<0pt,-1pt>:: 
(0,40) *+{\textcolor{green}{y_{\ell}}} ="0",
(40,40) *+{\textcolor{green}{z_k}} ="1",
(20,0) *+{\textcolor{green}{z_{\ell}}} ="2",
(60,0) *+{\textcolor{green}{y_t}} ="3",
(80,40) *+{\textcolor{red}{x(1)}} ="4",
(100,0) *+{\textcolor{red}{C_{1,k}}} ="5",
(120,40) *+{\textcolor{red}{x(2)}} ="6",
(140,0) *+{\textcolor{red}{C_{2,k}}} ="7",
(160,40) *+{\textcolor{red}{x(3)}} ="8",
(200,40) *+{\textcolor{red}{x(s)}} ="9",
(220,0) *+{\textcolor{red}{C_{s,k}}} ="10",
(240,40) *+{\textcolor{red}{x(s+1)}} ="11",
(180,40) *+{\cdots} ="12",
(265,40) *+{\cdots} ="13",
(290,40) *+{\textcolor{red}{\text{tr}(x(d))}} ="14",
(315,0) *+{\textcolor{red}{C_{d,k}}} ="15",
(340,40) *+{\textcolor{red}{\text{tr}(y_k)}} ="16",
(380,40) *+{\textcolor{green}{y_k}} ="17",
(400,0) *+{\textcolor{green}{y_m}} ="18",
(420,40) *+{\textcolor{green}{z_m}} ="19",
(40,100) *+{\textcolor{blue}{z_k^\prime}} ="20",
(80,100) *+{\textcolor{blue}{\widetilde{x}(1)}} ="21",
(120,100) *+{\textcolor{blue}{\widetilde{x}(2)}} ="22",
(160,100) *+{\textcolor{blue}{\widetilde{x}(3)}} ="23",
(200,100) *+{\textcolor{blue}{\widetilde{x}(s)}} ="24",
(240,100) *+{\textcolor{blue}{\widetilde{x}(s+1)}} ="25",
(290,100) *+{\textcolor{blue}{\widetilde{\text{tr}}(x(d))}} ="26",
(340,100) *+{\textcolor{blue}{\widetilde{\text{tr}}(y_k)}} ="27",
(380,100) *+{\textcolor{blue}{y_k^\prime}} ="28",
(420,100) *+{\textcolor{blue}{z_m^\prime}} ="29",
(0,100) *+{\textcolor{blue}{y_\ell^\prime}} ="30",
(20,-40) *+{\textcolor{blue}{z_\ell^\prime}} ="31",
(400,-40) *+{\textcolor{blue}{y_m^\prime}} ="32",
(100,-40) *+{\textcolor{blue}{\widetilde{C_{1,k}}}} = "33",
(140,-40) *+{\textcolor{blue}{\widetilde{C_{2,k}}}} = "34",
(220,-40) *+{\textcolor{blue}{\widetilde{C_{s,k}}}} = "35",
(60,-40) *+{\textcolor{blue}{y_{t}^\prime}} = "36",
(315,-40) *+{\textcolor{blue}{\widetilde{C_{d,k}}}} = "37",
"1", {\ar"0"},
"0", {\ar"2"},
"0", {\ar"30"},
"2", {\ar"1"},
"1", {\ar"3"},
"4", {\ar"1"},
"1", {\ar"20"},
"1", {\ar"21"},
"1", {\ar"22"},
"1", {\ar"23"},
"1", {\ar"24"},
"1", {\ar"25"},
"1", {\ar"26"},
"1", {\ar"27"},
"2", {\ar"31"},
"3", {\ar"4"},
"4", {\ar"5"},
"6", {\ar"4"},
"21", {\ar"4"},
"5", {\ar"6"},
"6", {\ar"7"},
"8", {\ar"6"},
"22", {\ar"6"},
"7", {\ar"8"},
"23", {\ar"8"},
"9", {\ar"10"},
"11", {\ar"9"},
"24", {\ar"9"},
"10", {\ar"11"},
"25", {\ar"11"},
"14", {\ar"15"},
"16", {\ar"14"},
"26", {\ar"14"},
"15", {\ar"16"},
"17", {\ar"16"},
"27", {\ar"16"},
"17", {\ar"18"},
"19", {\ar"17"},
"17", {\ar"28"},
"18", {\ar"19"},
"18", {\ar"32"},
"19", {\ar"29"},
"33", {\ar"5"},
"34", {\ar"7"},
"35", {\ar"10"},
"3", {\ar"36"},
"37", {\ar"15"},
\end{xy}$$
\caption{The local configuration around $y_k$ and $z_k$ just before $\underline{\mu}_k$ is applied.}
\label{importantfullsub1}
\end{figure}

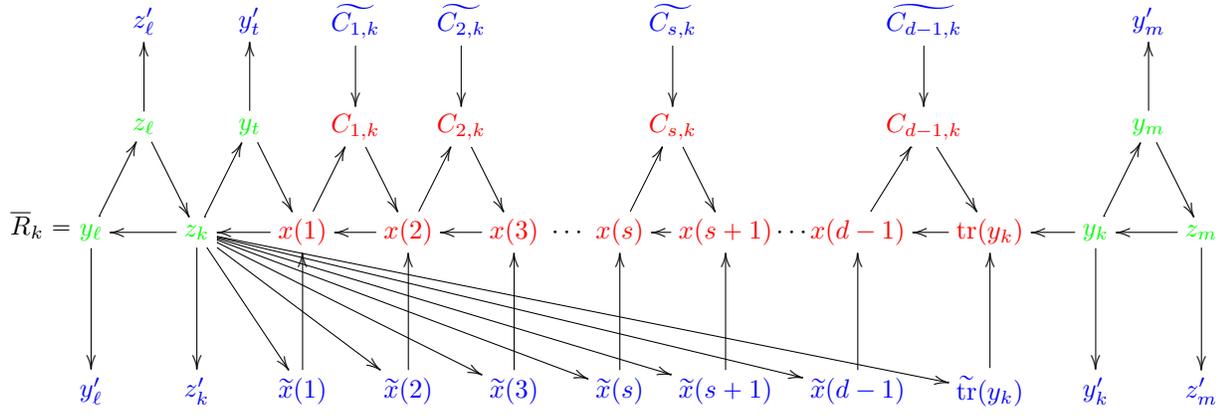
\begin{figure}
$$\raisebox{-.53in}{$\overline{R}_k =$} \begin{xy} 0;<1pt,0pt>:<0pt,-1pt>:: 
(0,40) *+{\textcolor{green}{y_{\ell}}} ="0",
(40,40) *+{\textcolor{green}{z_k}} ="1",
(20,0) *+{\textcolor{green}{z_{\ell}}} ="2",
(60,0) *+{\textcolor{green}{y_t}} ="3",
(80,40) *+{\textcolor{red}{x(1)}} ="4",
(100,0) *+{\textcolor{red}{C_{1,k}}} ="5",
(120,40) *+{\textcolor{red}{x(2)}} ="6",
(140,0) *+{\textcolor{red}{C_{2,k}}} ="7",
(160,40) *+{\textcolor{red}{x(3)}} ="8",
(200,40) *+{\textcolor{red}{x(s)}} ="9",
(220,0) *+{\textcolor{red}{C_{s,k}}} ="10",
(240,40) *+{\textcolor{red}{x(s+1)}} ="11",
(180,40) *+{\cdots} ="12",
(265,40) *+{\cdots} ="13",
(290,40) *+{\textcolor{red}{x(d-1)}} ="14",
(315,0) *+{\textcolor{red}{C_{d-1,k}}} ="15",
(340,40) *+{\textcolor{red}{\text{tr}(y_k)}} ="16",
(380,40) *+{\textcolor{green}{y_k}} ="17",
(400,0) *+{\textcolor{green}{y_m}} ="18",
(420,40) *+{\textcolor{green}{z_m}} ="19",
(40,100) *+{\textcolor{blue}{z_k^\prime}} ="20",
(80,100) *+{\textcolor{blue}{\widetilde{x}(1)}} ="21",
(120,100) *+{\textcolor{blue}{\widetilde{x}(2)}} ="22",
(160,100) *+{\textcolor{blue}{\widetilde{x}(3)}} ="23",
(200,100) *+{\textcolor{blue}{\widetilde{x}(s)}} ="24",
(240,100) *+{\textcolor{blue}{\widetilde{x}(s+1)}} ="25",
(290,100) *+{\textcolor{blue}{\widetilde{x}(d-1)}} ="26",
(340,100) *+{\textcolor{blue}{\widetilde{\text{tr}}(y_k)}} ="27",
(380,100) *+{\textcolor{blue}{y_k^\prime}} ="28",
(420,100) *+{\textcolor{blue}{z_m^\prime}} ="29",
(0,100) *+{\textcolor{blue}{y_\ell^\prime}} ="30",
(20,-40) *+{\textcolor{blue}{z_\ell^\prime}} ="31",
(400,-40) *+{\textcolor{blue}{y_m^\prime}} ="32",
(100,-40) *+{\textcolor{blue}{\widetilde{C_{1,k}}}} = "33",
(140,-40) *+{\textcolor{blue}{\widetilde{C_{2,k}}}} = "34",
(220,-40) *+{\textcolor{blue}{\widetilde{C_{s,k}}}} = "35",
(60,-40) *+{\textcolor{blue}{y_{t}^\prime}} = "36",
(315,-40) *+{\textcolor{blue}{\widetilde{C_{d-1,k}}}} = "37",
"1", {\ar"0"},
"0", {\ar"2"},
"0", {\ar"30"},
"2", {\ar"1"},
"1", {\ar"3"},
"4", {\ar"1"},
"1", {\ar"20"},
"1", {\ar"21"},
"1", {\ar"22"},
"1", {\ar"23"},
"1", {\ar"24"},
"1", {\ar"25"},
"1", {\ar"26"},
"1", {\ar"27"},
"2", {\ar"31"},
"3", {\ar"4"},
"4", {\ar"5"},
"6", {\ar"4"},
"21", {\ar"4"},
"5", {\ar"6"},
"6", {\ar"7"},
"8", {\ar"6"},
"22", {\ar"6"},
"7", {\ar"8"},
"23", {\ar"8"},
"9", {\ar"10"},
"11", {\ar"9"},
"24", {\ar"9"},
"10", {\ar"11"},
"25", {\ar"11"},
"14", {\ar"15"},
"16", {\ar"14"},
"26", {\ar"14"},
"15", {\ar"16"},
"17", {\ar"16"},
"27", {\ar"16"},
"17", {\ar"18"},
"19", {\ar"17"},
"17", {\ar"28"},
"18", {\ar"19"},
"18", {\ar"32"},
"19", {\ar"29"},
"33", {\ar"5"},
"34", {\ar"7"},
"35", {\ar"10"},
"3", {\ar"36"},
"37", {\ar"15"},
\end{xy}$$
\caption{The local configuration in the special case when $x(d)=x_1$ since $\text{tr}(x_1)$ is not defined.}
\label{importantfullsub2}
\end{figure}

\begin{lemma}\label{Lemma:mukR}
Let $k \in [n]$ be given and let $\overline{R}_k$ be the ice quiver described in Lemma~\ref{newmainlemma}.  (See Figure \ref{importantfullsub1}.) Then \begin{itemize} \item$\underline{\mu}_k(\overline{R}_k)$ has the form shown in Figure~\ref{mukR-reorganized}  (here, the vertices $y_m, y_m^\prime, z_m, z_m^\prime, y_\ell, y_\ell^\prime, z_\ell, z_\ell^\prime, y_t,$ and $y_t^\prime$ appear in $\underline{\mu}_k(\overline{R}_k)$ if and only if they appear in $\overline{R}_k$), 
\item $\underline{\mu}_k(\overline{R}_k)$ is a full subquiver of $\underline{\mu}_k\circ \cdots \circ \underline{\mu}_1 \circ \underline{\mu}_0(\widehat{Q}),$
\item as one mutates $\overline{R}_k$ along $\underline{\mu}_k$, one does so only at green vertices,
\item $\underline{\mu}_k(\overline{R}_k)$ includes every frozen vertex that is connected to a mutable vertex appearing in Figure~\ref{mukR-reorganized} by at least one arrow in $\underline{\mu}_k\circ \cdots \circ \underline{\mu}_1 \circ \underline{\mu}_0(\widehat{Q})$,
\item the full subquiver of $\underline{\mu}_{k-1}\circ \cdots \circ \underline{\mu}_1\circ \underline{\mu}_0(\widehat{Q})$ on the vertices $(\widehat{Q})_0\backslash(\overline{R}_k)_0$ is unchanged by the mutation sequence $\underline{\mu}_k$.
\item the vertices $z_\ell$ and $z_m$ (rather than $z_k$) are the only mutable vertices in $\underline\mu_k(\overline{R}_k)$ that are incident to multiple frozen vertices.  
\end{itemize}

{Additionally, for each $k \in [n]$, the full subquiver of $\underline{\mu}_k \circ \cdots \circ \underline{\mu}_1\circ \underline{\mu}_0(\widehat{Q})$ restricted to the green mutable vertices outside of $(\overline{R}_k)_0$, as well as the incident frozen vertices, equals the original framed quiver $\widehat{Q}$ restricted to those vertices.}

\end{lemma}

\begin{figure}
$$\raisebox{-.53in}{$\underline{\mu}_k(\overline{R}_k) =$} \begin{xy} 0;<1pt,0pt>:<0pt,-1pt>:: 
(360,0) *+{\textcolor{green}{y_{\ell}}} ="0",
(40,40) *+{\textcolor{red}{z_k}} ="1",
(0,40) *+{\textcolor{green}{z_{\ell}}} ="2",
(60,0) *+{\textcolor{green}{y_t}} ="3",
(80,40) *+{\textcolor{red}{x(1)}} ="4",
(100,0) *+{\textcolor{red}{C_{1,k}}} ="5",
(120,40) *+{\textcolor{red}{x(2)}} ="6",
(140,0) *+{\textcolor{red}{C_{2,k}}} ="7",
(160,40) *+{\textcolor{red}{x(s-1)}} ="8",
(200,40) *+{\textcolor{red}{x(s)}} ="9",
(220,0) *+{\textcolor{red}{C_{s,k}}} ="10",
(240,40) *+{\textcolor{red}{x(d-1)}} ="11",
(160,20) *+{\cdots} ="12",
(240,20) *+{\cdots} ="13",
(290,40) *+{\textcolor{red}{\text{tr}(x(d))}} ="14",
(315,0) *+{\textcolor{red}{C_{d,k}}} ="15",
(340,70) *+{\textcolor{red}{\text{tr}(y_k)}} ="16",
(380,40) *+{\textcolor{red}{y_k}} ="17",
(400,0) *+{\textcolor{green}{y_m}} ="18",
(420,40) *+{\textcolor{green}{z_m}} ="19",
(40,100) *+{\textcolor{blue}{z_k^\prime}} ="20",
(80,100) *+{\textcolor{blue}{\widetilde{x}(1)}} ="21",
(120,100) *+{\textcolor{blue}{\widetilde{x}(2)}} ="22",
(160,100) *+{\textcolor{blue}{\widetilde{x}(3)}} ="23",
(200,100) *+{\textcolor{blue}{\widetilde{x}(s)}} ="24",
(240,100) *+{\textcolor{blue}{\widetilde{x}(s+1)}} ="25",
(290,100) *+{\textcolor{blue}{\widetilde{\text{tr}}(x(d))}} ="26",
(340,100) *+{\textcolor{blue}{\widetilde{\text{tr}}(y_k)}} ="27",
(380,100) *+{\textcolor{blue}{y_k^\prime}} ="28",
(420,100) *+{\textcolor{blue}{z_m^\prime}} ="29",
(0,100) *+{\textcolor{blue}{y_\ell^\prime}} ="30",
(0,-40) *+{\textcolor{blue}{z_\ell^\prime}} ="31",
(400,-40) *+{\textcolor{blue}{y_m^\prime}} ="32",
(100,-40) *+{\textcolor{blue}{\widetilde{C_{1,k}}}} = "33",
(140,-40) *+{\textcolor{blue}{\widetilde{C_{2,k}}}} = "34",
(220,-40) *+{\textcolor{blue}{\widetilde{C_{s,k}}}} = "35",
(60,-40) *+{\textcolor{blue}{y_{t}^\prime}} = "36",
(315,-40) *+{\textcolor{blue}{\widetilde{C_{d,k}}}} = "37",
"0", {\ar"30"},
"0", {\ar"16"},
"1", {\ar"2"},
"1", {\ar"5"},
"3", {\ar"1"},
"4", {\ar"1"},
"2", {\ar"21"},
"2", {\ar"22"},
"2", {\ar"23"},
"2", {\ar"24"},
"2", {\ar"25"},
"2", {\ar"26"},
"2", {\ar"27"},
"2", {\ar"3"},
"2", {\ar"20"},
"2", {\ar"31"},
"4", {\ar"7"},
"5", {\ar"4"},
"6", {\ar"4"},
"21", {\ar"1"},
"7", {\ar"6"},
"8", {\ar"10"},
"14", {\ar"11"},
"14", {\ar"17"},
"16", {\ar"19"},
"19", {\ar"0"},
"20", {\ar"16"},
"22", {\ar"4"},
"23", {\ar"6"},
"10", {\ar"9"},
"9", {\ar"8"},
"11", {\ar"15"},
"24", {\ar"8"},
"25", {\ar"9"},
"15", {\ar"14"},
"16", {\ar"14"},
"26", {\ar"11"},
"17", {\ar"16"},
"27", {\ar"14"},
"18", {\ar"17"},
"19", {\ar"20"},
"28", {\ar"17"},
"18", {\ar"32"},
"19", {\ar"28"},
"19", {\ar"29"},
"33", {\ar"5"},
"34", {\ar"7"},
"35", {\ar"10"},
"3", {\ar"36"},
"37", {\ar"15"},
\end{xy}$$
\caption{The quiver $\underline{\mu}_k(\overline{R}_k)$ before rearrangement.}
\label{mukR}
\end{figure}
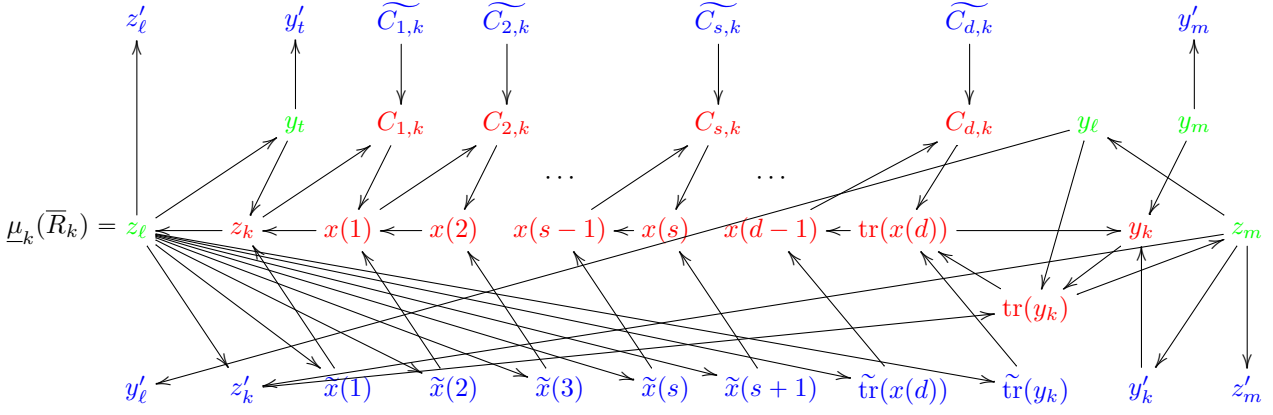

\begin{figure}
$$\raisebox{-.53in}{$\underline{\mu}_k(\overline{R}_k) =$} \begin{xy} 0;<1pt,0pt>:<0pt,-1pt>:: 
(400,40) *+{\textcolor{green}{y_{\ell}}} ="0",
(40,40) *+{\textcolor{red}{z_k}} ="1",
(0,40) *+{\textcolor{green}{z_{\ell}}} ="2",
(20,0) *+{\textcolor{green}{y_t}} ="3",
(80,40) *+{\textcolor{red}{x(1,k)}} ="4",
(60,0) *+{\textcolor{red}{C_{1,k}}} ="5",
(140,40) *+{\textcolor{red}{x(s-1,k)}} ="8",
(190,40) *+{\textcolor{red}{x(s,k)}} ="9",
(165,0) *+{\textcolor{red}{C_{s,k}}} ="10",
(250,40) *+{\textcolor{red}{x(d-1,k)}} ="11",
(105,40) *+{\cdots} ="12",
(215,40) *+{\cdots} ="13",
(310,40) *+{\textcolor{red}{\text{tr}(x(d,k))}} ="14",
(280,0) *+{\textcolor{red}{C_{d,k}}} ="15",
(360,40) *+{\textcolor{red}{\text{tr}(y_k)}} ="16",
(335,0) *+{\textcolor{red}{y_k}} ="17",
(360,0) *+{\textcolor{green}{y_m}} ="18",
(380,0) *+{\textcolor{green}{z_m}} ="19",
(360,100) *+{\textcolor{blue}{z_k^\prime}} ="20",
(40,100) *+{\textcolor{blue}{\widetilde{x}(1,k)}} ="21",
(80,100) *+{\textcolor{blue}{\widetilde{x}(2,k)}} ="22",
(140,100) *+{\textcolor{blue}{\widetilde{x}(s,k)}} ="24",
(190,100) *+{\textcolor{blue}{\widetilde{x}(s+1,k)}} ="25",
(250,100) *+{\textcolor{blue}{\widetilde{\text{tr}}(x(d-1,k))}} ="26",
(310,100) *+{\textcolor{blue}{\widetilde{\text{tr}}(y_k)}} ="27",
(335,-40) *+{\textcolor{blue}{y_k^\prime}} ="28",
(380,-40) *+{\textcolor{blue}{z_m^\prime}} ="29",
(400,100) *+{\textcolor{blue}{y_\ell^\prime}} ="30",
(0,100) *+{\textcolor{blue}{z_\ell^\prime}} ="31",
(360,-40) *+{\textcolor{blue}{y_m^\prime}} ="32",
(60,-40) *+{\textcolor{blue}{\widetilde{C_{1,k}}}} = "33",
(165,-40) *+{\textcolor{blue}{\widetilde{C_{s,k}}}} = "35",
(20,-40) *+{\textcolor{blue}{y_{t}^\prime}} = "36",
(280,-40) *+{\textcolor{blue}{\widetilde{C_{d,k}}}} = "37",
"0", {\ar"30"},
"0", {\ar"16"},
"1", {\ar"2"},
"1", {\ar"5"},
"3", {\ar"1"},
"4", {\ar"1"},
"2", {\ar"21"},
"2", {\ar"22"},
"2", {\ar"24"},
"2", {\ar"25"},
"2", {\ar"26"},
"2", {\ar"27"},
"2", {\ar"3"},
"2", {\ar"20"},
"2", {\ar"31"},
"5", {\ar"4"},
"21", {\ar"1"},
"8", {\ar"10"},
"14", {\ar"11"},
"14", {\ar"17"},
"16", {\ar"19"},
"19", {\ar"0"},
"20", {\ar"16"},
"22", {\ar"4"},
"10", {\ar"9"},
"9", {\ar"8"},
"11", {\ar"15"},
"24", {\ar"8"},
"25", {\ar"9"},
"15", {\ar"14"},
"16", {\ar"14"},
"26", {\ar"11"},
"17", {\ar"16"},
"27", {\ar"14"},
"18", {\ar"17"},
"19", {\ar"20"},
"28", {\ar"17"},
"18", {\ar"32"},
"19", {\ar"28"},
"19", {\ar"29"},
"33", {\ar"5"},
"35", {\ar"10"},
"3", {\ar"36"},
"37", {\ar"15"},
\end{xy}$$
\caption{The quiver $\underline{\mu}_k(\overline{R}_k)$ rearranged to look more like $\overline{R}_{k+1}$.}
\label{mukR-reorganized}
\end{figure}
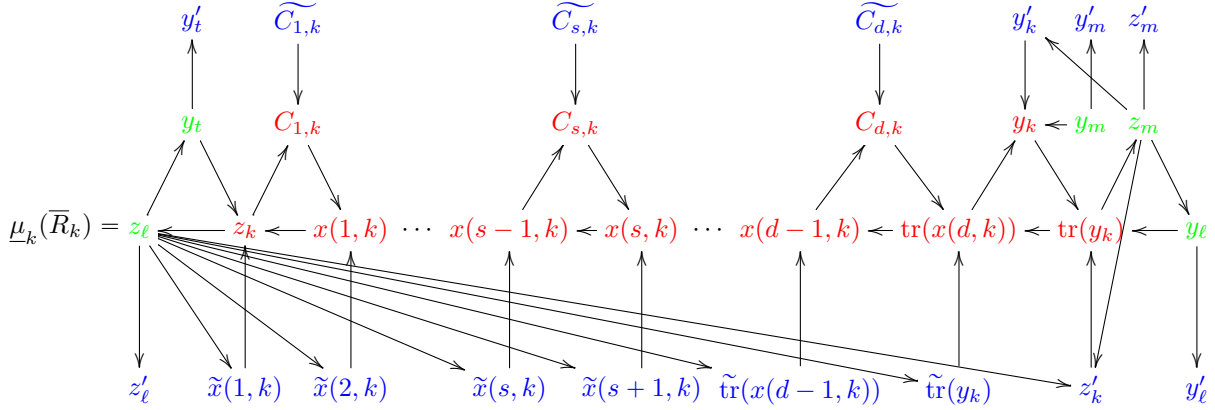

\begin{proof}[Proof of Theorem~\ref{main2}]
By the third assertion in Lemma~\ref{Lemma:mukR}, the associated mutation sequence $\underline{\mu} = \underline{\mu}_n\circ \cdots \underline{\mu}_1\circ \underline{\mu}_0$ of $\mathcal{Q}$ is a green {mutation} sequence of $Q$. By Lemma~\ref{newmainlemma}, $$\underline{\mu}_n \circ \cdots \circ \underline{\mu}_1\circ \underline{\mu}_0(\widehat{Q}) = \underline{\mu}_n \circ \cdots \circ \underline{\mu}_1\circ \underline{\mu}_0(\widehat{\mathcal{Q}}_{n,n}) = \widecheck{\mathcal{Q}}_{n,n}\cdot \sigma_n = \widecheck{Q}\cdot\sigma_n$$ and so every mutable vertex of $\underline{\mu}_n \circ \cdots \circ \underline{\mu}_1\circ \underline{\mu}_0(\widehat{{Q}})$ is red. Thus $\underline{\mu} = \underline{\mu}_n\circ \cdots \underline{\mu}_1\circ \underline{\mu}_0 \in \text{green}(Q).$
\end{proof}

\begin{remark}\label{Lemma:Only4}
It follows from Lemma~\ref{BV} that as one mutates $\overline{R}_k$ along $\underline{\mu}_k = \mu_{i_r}\circ \cdots \circ \mu_{i_1}$, we have that $i_j$ in $\mu_{i_{j-1}}\circ \cdots \circ \mu_{i_1}(\overline{R}_k)$ is incident to at most four other mutable vertices. \end{remark}

\begin{proof}[Proof of Lemma~\ref{Lemma:mukR}]
The first assertion follows inductively by mutating the vertices of $\overline{R}_k$ in the specified order $\underline{\mu}_k = \mu_{\text{tr}(y_k)} \circ \mu_{\text{tr}(x(d))}\circ \mu_{x(d-1)}\circ \mu_{x(d-2)}\circ \cdots \mu_{x(1)}\circ\mu_{x(0)}\circ \mu_{y_k}$, reading right-to-left.  In particular, as this mutation sequence is applied to $\overline{R}_k$, Remark~\ref{Lemma:Only4} shows that the mutable vertices incident to $y_\ell$ are located further and further to the right in Figure 19 until we see that they are $\text{tr}(y_k)$ and $z_m$ at the end of the sequence. In fact, we observe after mutating $\overline{R}_k$ at $y_k$ that $z_k$ is the unique green vertex of $\overline{R}_k$ (with the exception of the vertices $y_\ell$, $z_\ell$, $y_t$, $y_m$ and $z_m$, if they appear in $\overline{R}_k$). As we continue to mutate $\mu_{y_k}(\overline{R}_k)$ along the remaining mutations in $\underline{\mu}_k$, the unique green vertex is $x(s)$ for some $s \in [0,d-1]$ or as $\text{tr}(x(d))$ or $\text{tr}(y_k)$ (as before, with the exception of the vertices $y_\ell$, $z_\ell$, $y_t$, $y_m$ and $z_m$).  Iteratively mutating at this unique green vertex exactly corresponds to performing the mutation sequence $\underline{\mu}_{A(k)}\circ\underline{\mu}_{B(k)}\circ\underline{\mu}_{C(k)}$ on $\underline{\mu}_{D(k)}(\overline{R}_k)$.

The second assertion, $\underline{\mu}_k(\overline{R}_k)$ is a full subquiver of $\underline{\mu}_k\circ \cdots \circ \underline{\mu}_1 \circ \underline{\mu}_0(\widehat{Q})$, follows since the vertices of $Q$ at which one mutates when applying $\underline{\mu}_k$ are all vertices of $\overline{R}_k$. One can see that the third assertion follows from the above observation that a unique vertex becomes green as we iteratively mutate. The fourth assertion holds for $\underline{\mu}_k(\overline{R}_k)$ since it holds for $\overline{R}_k$. The fifth assertion follows since the vertices in the support of $\underline{\mu}_k$ are all disconnected from the vertices in $(\underline{\mu}_{k-1}\circ \cdots \circ \underline{\mu}_1 \circ \underline{\mu}_0(\widehat{Q}))_0\setminus (\overline{R}_k)_0$. Further, the sixth assertion is demonstrated inductively as we mutate $x(s)$ for $s \in [0,d-1]$.  Lastly, by restricting to the green mutable vertices outside of $(\overline{R}_k)_0$ and the incident frozen vertices, it is clear that the mutation sequence $\underline{\mu}_k$ leaves this full subquiver unaffected.\end{proof}

\begin{proof}[Proof of Lemma~\ref{newmainlemma}]
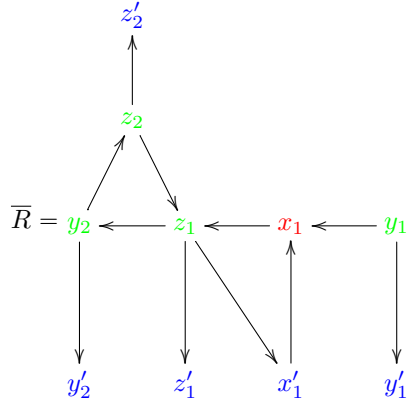
\begin{figure}
$$\raisebox{-.53in}{${\overline{R}} =$} \begin{xy} 0;<1pt,0pt>:<0pt,-1pt>:: 
(0,40) *+{\textcolor{green}{y_{2}}} ="0",
(40,40) *+{\textcolor{green}{z_1}} ="1",
(20,0) *+{\textcolor{green}{z_{2}}} ="2",
(80,40) *+{\textcolor{red}{x_1}} ="4",
(120,40) *+{\textcolor{green}{y_1}} ="6",
(40,100) *+{\textcolor{blue}{z_1^\prime}} ="20",
(80,100) *+{\textcolor{blue}{x_1^\prime}} ="21",
(120,100) *+{\textcolor{blue}{y_{1}^\prime}} ="22",
(0,100) *+{\textcolor{blue}{y_2^\prime}} ="30",
(20,-40) *+{\textcolor{blue}{z_2^\prime}} ="31",
"1", {\ar"0"},
"0", {\ar"2"},
"0", {\ar"30"},
"2", {\ar"1"},
"4", {\ar"1"},
"1", {\ar"20"},
"1", {\ar"21"},
"2", {\ar"31"},
"6", {\ar"4"},
"21", {\ar"4"},
"6", {\ar"22"},
\end{xy}$$
\caption{The subquiver $\overline{R} = \overline{R}_1$ of $\underline{\mu}_0(\widehat{Q}).$}
\label{mu0sub}
\end{figure}

We prove the lemma by induction. For $k=1$, observe that $\underline{\mu}_0(\widehat{Q})$ has the full subquiver $\overline{R}$ shown in Figure~\ref{mu0sub} where we assume that $n > 1$. We show that $\overline{R}$ has all of the properties that $\overline{R}_1$ must satisfy. Note that $\text{tr}(y_1) = x_1$ and for $k = 1$ one has that $x(1)= x_{m_1} = x_1$.  Since $\text{deg}(y_1) = 2$, no vertices $y_m, y_m^\prime, z_m,$ and $z_m^\prime$ appear in $\overline{R}$, as desired. Since only vertex $x_1$ has been mutated to obtain $\underline{\mu}_0(\widehat{Q})$, no arrows between vertices of a signed 3-cycle $T_t$ with $1 < t$ and vertices of signed 3-cycle $T_i$ with $i \le 1$ have been created. Furthermore, there is no signed 3-cycle $T_t$ of $\mathcal{Q}$ with $1 < t$ where $\text{tr}|_{1,t}(y_t) = z_1$. Note that in this degenerate case, $\text{tr}(y_1)=x(1)$ and so no $C_{i,1}$'s or $\widetilde{C_{i,1}}$'s appear in $\overline{R}_1$, and $\widetilde{x}(1)=\widetilde{\text{tr}}(y_1)=x_1'$. 
Thus the quiver $\overline{R}$ satisfies all of the properties that $\overline{R}_1$ must satisfy.  
Further, in this special case $\widehat{\mathcal{Q}_{0,0}}$ contains only the vertex $x_1$ and $x_1^\prime$ and $\sigma_0$ is the identity permutation.  Thus $\underline{\mu}_0\widehat{Q_{0,0}}$ indeed equals 
$\widecheck{Q_{0,0}}\cdot \sigma_0$.

Now assume that $k>1$ and that $\underline{\mu}_{k-1}\circ \cdots \circ \underline{\mu}_1\circ \underline{\mu}_0(\widehat{Q})$ has a full subquiver $\overline{R}_k$ with the properties in the statement of the lemma. To show that $\underline{\mu}_{k}\circ \cdots \circ \underline{\mu}_1\circ \underline{\mu}_0(\widehat{Q})$ has the desired full subquiver $\overline{R}_{k+1}$, we consider four cases: \begin{itemize}\item[i)]$\text{deg}(y_k) = 2$ and $\text{deg}(z_k) = 4,$
\item[ii)] $\text{deg}(y_k) = 4$ and $\text{deg}(z_k) = 2,$
\item[iii)] $\text{deg}(y_k) = 4$ and $\text{deg}(z_k) = 4,$ and
\item[iv)] $\text{deg}(y_k) = 2$ and $\text{deg}(z_k) = 2.$

\end{itemize}

Suppose that we are in Case i). By the properties of the ice quiver $\overline{R}_k$, this means that vertices $y_m, y_m^\prime, z_m,$ and $z_m^\prime$ do not appear in $\overline{R}_k$. This also implies that $\ell = k+1$. Now Lemma~\ref{Lemma:mukR} implies that $\underline{\mu}_k(\overline{R}_k)$ has the form shown in Figure~\ref{mukRcaseI} where the vertices $y_t$ and $y_t^\prime$ appear $\underline{\mu}_k(\overline{R}_k)$ if and only if they appear in $\overline{R}_k$. Note that the quiver in Figure~\ref{mukRcaseI} is the same as the quiver in Figure \ref{mukR-reorganized} with the notation updated accordingly.
In particular, the integers $m_1^{(k+1)}, m_2^{(k+1)}, \ldots, m_{d+1}^{(k+1)} \in [n]$ and the vertices $x(1,k+1), \ldots, x(d+1, k+1) \in (Q)_0$ are those defined by $T_{k+1}$ following Definition \ref{specialT}. 
Since the signed 3-cycles $T_k$ and $T_{k+1}$ share the vertex $z_k$ (i.e. $z_k=x_{k+1}$), we have that \begin{equation} \label{label-update} m_1^{(k+1)} = k+1,\ m_2^{(k+1)} = m_1,\ \ldots,\ m_j^{(k+1)} = m_{j-1},\ \ldots,\ m_{d+1}^{(k+1)} = m_d\end{equation} and $$x(1,k+1) = z_k,\ x(2,k+1) = x(1),\ \ldots,\ x(s,k+1) = x(s-1),\ \ldots, x(d+1, k+1) = x(d).$$ This implies that $\text{tr}(x(d+1,k+1)) = \text{tr}(x(d,k))$ and $\text{tr}(y_{k+1}) = \text{tr}(y_k).$ Now we also obtain that $$\widetilde{x}(1,k+1) = z^\prime_{m_{d+1}^{(k+1)}-1} = z^\prime_{m_d-1} = \widetilde{x}(1,k) \ \text{and} \ \widetilde{x}(s,k+1) = x^\prime_{m^{(k+1)}_{d+1-s+2}}= x^\prime_{m_j^{(k+1)}} = x^\prime_{m_{j-1}} = \widetilde{x}(s,k)$$ for $s \in [2,d]$ where $j = (d+1)-s+2$ and that $$\widetilde{\text{tr}}(x(d+1,k+1)) = x^\prime_{m_2^{(k+1)}} = x_{m_1}^\prime = \widetilde{\text{tr}}(y_k) \ \text{and} \ \widetilde{\text{tr}}(y_{k+1}) = x^\prime_{m_1^{(k+1)}} = x^\prime_{k+1} = z_k^\prime$$ where the last equality follows from the fact that $T_k$ and $T_{k+1}$ share the vertex $z_k$. Thus we have labeled the vertices of $\underline{\mu}_k(\overline{R}_k)$ accordingly in Figure~\ref{mukRcaseI}. Furthermore, that the signed 3-cycles $T_k$ and $T_{k+1}$ share the vertex $z_k$ implies that $z_{k+1} = \text{tr}|_{k+1,t}(y_t)$ if and only if $z_k = \text{tr}|_{k,t}(y_t)$.

Next, observe that for any $s \in [d]$ we have $C_{s,k}  \cdot \tau_{k}^{-1} = C_{s,k}$ since we do not mutate $C_{s,k}$ when applying $\underline{\mu}_k$. Additionally 
$x_{m^{(k+1)}_{d+1}-1} = x_{m_d-1}$ and $y_{m^{(k+1)}_{j+1}} = y_{m_j}$ (for $j = (d+1)-s+2$ where $s \in [2,d]$) follows from (\ref{label-update}).  Comparing with the fifth bullet point of Lemma \ref{newmainlemma}, we obtain 
$C_{s,k} = C_{s,k}\cdot \tau_k^{-1} = C_{s,k+1}$
and $\widetilde{C_{s,k}} = \widetilde{C_{s,k}}\cdot \tau_k^{-1} = \widetilde{C_{s,k+1}}$ for any $s \in [d]$.

Now let $C_{d+1,k+1} = y_k$ and $\widetilde{C_{d+1,k+1}}= y_k^\prime$. Note that $y_k \cdot \sigma_{k-1}^{-1} = y_k$ since $y_k$ has not been mutated in $\underline{\mu}_{k-1}\circ \cdots \circ \underline{\mu}_0(\widehat{Q})$. Furthermore, $y_k \cdot \sigma_{k-1}^{-1}\tau_k^{-1} = y_k \cdot \tau_{k}^{-1} = y_k$ by the definition of $\tau_k$ so $C_{d+1,k+1} = y_k\cdot \sigma_{k}^{-1},$ as desired.

We now construct an ice quiver $\overline{R}$ that is the full subquiver of $\underline{\mu}_k\circ \cdots \circ \underline{\mu}_1\circ\underline{\mu}_0(\widehat{Q})$ on the vertices of $\underline{\mu}_k(\overline{R}_k)$, as well as the vertices $x_r, y_r, z _r$ and corresponding frozen vertices $x^\prime_r, y^\prime_r, z^\prime_r$ of any signed 3-cycles $T_r$ of $\mathcal{Q}$ where $k<r$ and $z_{k+1}$ or $y_{k+1}$ is a vertex if $T_r$.  Comparing this construction of $\overline{R}$ to the quiver $\overline{R}_{k+1}$ appearing in Figure~\ref{importantfullsub1}, we verify that $\overline{R}$ indeed equals $\overline{R}_{k+1}$ and satisfies the five properties listed as bullet points in Lemma \ref{newmainlemma}.

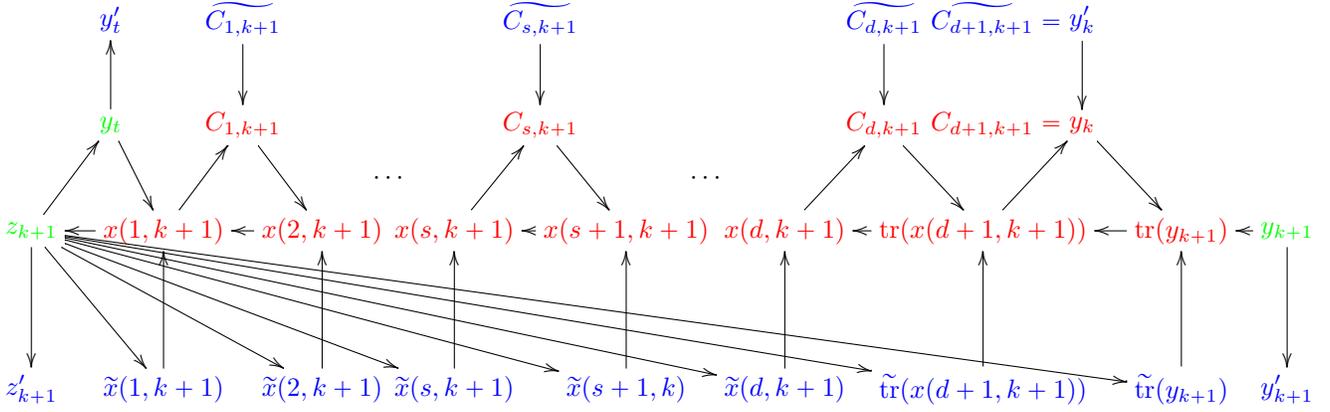
\begin{figure}
$\begin{xy} 0;<1pt,0pt>:<0pt,-1pt>:: 
(465,40) *+{\textcolor{green}{y_{k+1}}} ="0",
(40,40) *+{\textcolor{red}{x(1,k+1)}} ="1",
(-10,40) *+{\textcolor{green}{z_{k+1}}} ="2",
(20,0) *+{\textcolor{green}{y_t}} ="3",
(100,40) *+{\textcolor{red}{x(2,k+1)}} ="4",
(70,0) *+{\textcolor{red}{C_{1,k+1}}} ="5",
(150,40) *+{\textcolor{red}{x(s,k+1)}} ="8",
(215,40) *+{\textcolor{red}{x(s+1,k+1)}} ="9",
(182.5,0) *+{\textcolor{red}{C_{s,k+1}}} ="10",
(275,40) *+{\textcolor{red}{x(d,k+1)}} ="11",
(125,20) *+{\cdots} ="12",
(245,20) *+{\cdots} ="13",
(350,40) *+{\textcolor{red}{\text{tr}(x(d+1,k+1))}} ="14",
(312.5,0) *+{\textcolor{red}{C_{d,k+1}}} ="15",
(425,40) *+{\textcolor{red}{\text{tr}(y_{k+1})}} ="16",
(387.5,0) *+{\textcolor{red}{y_k}} ="17",
(355,0) *+{\textcolor{red}{{C_{d+1,k+1}}=}} ="18",
(425,100) *+{\textcolor{blue}{\widetilde{\text{tr}}(y_{k+1})}} ="20",
(40,100) *+{\textcolor{blue}{\widetilde{x}(1,k+1)}} ="21",
(100,100) *+{\textcolor{blue}{\widetilde{x}(2,k+1)}} ="22",
(150,100) *+{\textcolor{blue}{\widetilde{x}(s,k+1)}} ="24",
(215,100) *+{\textcolor{blue}{\widetilde{x}(s+1,k)}} ="25",
(275,100) *+{\textcolor{blue}{\widetilde{x}(d,k+1)}} ="26",
(350,100) *+{\textcolor{blue}{\widetilde{\text{tr}}(x(d+1,k+1))}} ="27",
(387.5,-40) *+{\textcolor{blue}{y_k^\prime}} ="28",
(465,100) *+{\textcolor{blue}{y_{k+1}^\prime}} ="30",
(-10,100) *+{\textcolor{blue}{z_{k+1}^\prime}} ="31",
(355,-40) *+{\textcolor{blue}{{\widetilde{C_{d+1,k+1}}=}}} ="32",
(70,-40) *+{\textcolor{blue}{\widetilde{C_{1,k+1}}}} = "33",
(182.5,-40) *+{\textcolor{blue}{\widetilde{C_{s,k+1}}}} = "35",
(20,-40) *+{\textcolor{blue}{y_{t}^\prime}} = "36",
(312.5,-40) *+{\textcolor{blue}{\widetilde{C_{d,k+1}}}} = "37",
"0", {\ar"30"},
"0", {\ar"16"},
"1", {\ar"2"},
"1", {\ar"5"},
"3", {\ar"1"},
"4", {\ar"1"},
"2", {\ar"21"},
"2", {\ar"22"},
"2", {\ar"24"},
"2", {\ar"25"},
"2", {\ar"26"},
"2", {\ar"27"},
"2", {\ar"3"},
"2", {\ar"20"},
"2", {\ar"31"},
"5", {\ar"4"},
"21", {\ar"1"},
"8", {\ar"10"},
"14", {\ar"11"},
"14", {\ar"17"},
"20", {\ar"16"},
"22", {\ar"4"},
"10", {\ar"9"},
"9", {\ar"8"},
"11", {\ar"15"},
"24", {\ar"8"},
"25", {\ar"9"},
"15", {\ar"14"},
"16", {\ar"14"},
"26", {\ar"11"},
"17", {\ar"16"},
"27", {\ar"14"},
"28", {\ar"17"},
"33", {\ar"5"},
"35", {\ar"10"},
"3", {\ar"36"},
"37", {\ar"15"},
\end{xy}$
\caption{The quiver $\underline{\mu}_k(\overline{R}_k)$ obtained by mutating $\overline{R}_k$ in Case i).}
\label{mukRcaseI}
\end{figure}

Next, suppose that we are in Case ii). %
%
In this situation, we have that $m = k+1$ and the vertices $y_\ell, y_\ell^\prime, z_\ell,$ and $z_\ell^\prime$ do not belong to $\overline{R}_k$. Now Lemma~\ref{Lemma:mukR} implies that $\underline{\mu}_k(\overline{R}_k)$ has the form shown in Figure~\ref{mukRii}.  We let $T_p$ (resp. $T_q$) be the signed $3$-cycle not equal to $T_{k+1}$ that contains $z_{k+1}$ (resp. $y_{k+1}$), if they exist.  Define $\overline{R}$ to be the ice quiver that is a full subquiver of $\underline{\mu}_k \circ \cdots \circ \underline{\mu}_1\circ \underline{\mu}_0(\widehat{Q})$ on the vertices $$y_{k+1}, y_{k+1}^\prime, z_{k+1}, z_{k+1}^\prime, \text{tr}(y_k), z_k^\prime, y_k, y_k^\prime, \text{tr}(x(d)), x_k^\prime, y_{p}, y_{p}^\prime, z_{p}, z_{p}^\prime, y_q, y_q^\prime, z_q, z_q^\prime$$ where we include $y_p$ and $y_p^\prime$ (resp. $z_p$, $z_p^\prime$, $y_q, y_q^\prime, z_q,$ and $z_q^\prime$) in $\overline{R}$ if and only if $y_p$ and $y_p^\prime$ (resp. $z_q$, $z_q^\prime$, $y_q, y_q^\prime, z_q,$ and $z_q^\prime$) appear in $\underline{\mu}_k(\overline{R}_k)$, i.e. depending on if $\text{deg}(y_{k+1})=4$ and if $\text{deg}(z_{k+1}) = 4$.  See Figure \ref{mu0subII}.

\begin{figure}
$$\raisebox{-.53in}{$\underline{\mu}_k(\overline{R}_k) =$} \begin{xy} 0;<1pt,0pt>:<0pt,-1pt>:: 
(40,40) *+{\textcolor{red}{z_k}} ="1",
(80,40) *+{\textcolor{red}{x(1,k)}} ="4",
(60,0) *+{\textcolor{red}{C_{1,k}}} ="5",
(140,40) *+{\textcolor{red}{x(s-1,k)}} ="8",
(190,40) *+{\textcolor{red}{x(s,k)}} ="9",
(165,0) *+{\textcolor{red}{C_{s,k}}} ="10",
(250,40) *+{\textcolor{red}{x(d-1,k)}} ="11",
(105,40) *+{\cdots} ="12",
(215,40) *+{\cdots} ="13",
(310,40) *+{\textcolor{red}{\text{tr}(x(d,k))}} ="14",
(280,0) *+{\textcolor{red}{C_{d,k}}} ="15",
(360,40) *+{\textcolor{red}{\text{tr}(y_k)}} ="16",
(335,0) *+{\textcolor{red}{y_k}} ="17",
(360,0) *+{\textcolor{green}{y_{k+1}}} ="18",
(410,40) *+{\textcolor{green}{z_{k+1}}} ="19",
(360,100) *+{\textcolor{blue}{z_k^\prime}} ="20",
(40,100) *+{\textcolor{blue}{\widetilde{x}(1,k)}} ="21",
(80,100) *+{\textcolor{blue}{\widetilde{x}(2,k)}} ="22",
(140,100) *+{\textcolor{blue}{\widetilde{x}(s,k)}} ="24",
(190,100) *+{\textcolor{blue}{\widetilde{x}(s+1,k)}} ="25",
(250,100) *+{\textcolor{blue}{\widetilde{\text{tr}}(x(d,k))}} ="26",
(310,100) *+{\textcolor{blue}{\widetilde{\text{tr}}(y_k)}} ="27",
(335,-40) *+{\textcolor{blue}{y_k^\prime}} ="28",
(410,-40) *+{\textcolor{blue}{z_{k+1}^\prime}} ="29",
(360,-40) *+{\textcolor{blue}{y_{k+1}^\prime}} ="32",
(60,-40) *+{\textcolor{blue}{\widetilde{C_{1,k}}}} = "33",
(165,-40) *+{\textcolor{blue}{\widetilde{C_{s,k}}}} = "35",
(280,-40) *+{\textcolor{blue}{\widetilde{C_{d,k}}}} = "37",
"1", {\ar"5"},
"4", {\ar"1"},
"5", {\ar"4"},
"21", {\ar"1"},
"8", {\ar"10"},
"14", {\ar"11"},
"14", {\ar"17"},
"16", {\ar"19"},
"20", {\ar"16"},
"22", {\ar"4"},
"10", {\ar"9"},
"9", {\ar"8"},
"11", {\ar"15"},
"24", {\ar"8"},
"25", {\ar"9"},
"15", {\ar"14"},
"16", {\ar"14"},
"26", {\ar"11"},
"17", {\ar"16"},
"27", {\ar"14"},
"18", {\ar"17"},
"19", {\ar"20"},
"28", {\ar"17"},
"18", {\ar"32"},
"19", {\ar"28"},
"19", {\ar"29"},
"33", {\ar"5"},
"35", {\ar"10"},
"37", {\ar"15"},
\end{xy}$$
\caption{The quiver obtained by mutating $\overline{R}_k$ in Case ii). }
\label{mukRii}
\end{figure}
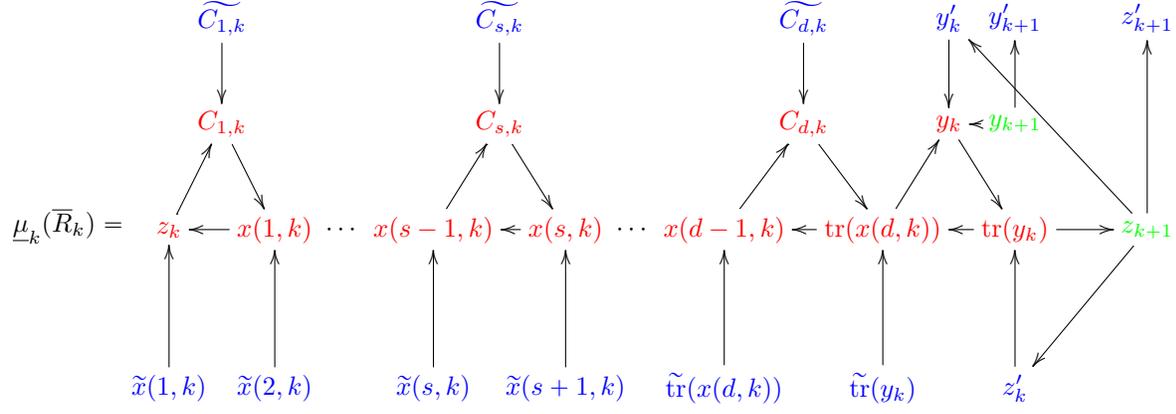

\begin{figure}
$$\raisebox{-.53in}{$\overline{R} =$} \begin{xy} 0;<1pt,0pt>:<0pt,-1pt>:: 
(-20,40) *+{\textcolor{green}{z_{k+1}}} ="0",
(-60,40) *+{\textcolor{green}{y_{p}}} ="3",
(-40,0) *+{\textcolor{green}{z_{p}}} ="7",
(40,40) *+{\textcolor{red}{\text{tr}(x(1,k+1))}} ="1",
(120,40) *+{\textcolor{red}{\text{tr}{(y_{k+1})}}} ="4",
(80,0) *+{\textcolor{red}{x_k\cdot\sigma_{k}^{-1}}} ="5",
(40,0) *+{\textcolor{red}{C_{1,k+1} =}} ="39",
(50,-40) *+{\textcolor{blue}{\widetilde{C_{1,k+1}} =}} ="40",
(170,40) *+{\textcolor{green}{y_{k+1}}} ="6",
(40,100) *+{\textcolor{blue}{z_k^\prime}} ="20",
(120,100) *+{\textcolor{blue}{y^\prime_{k}}} ="21",
(170,100) *+{\textcolor{blue}{y_{k+1}^\prime}} ="22",
(-20,100) *+{\textcolor{blue}{z_{k+1}^\prime}} ="30",
(-60,100) *+{\textcolor{blue}{y_{p}^\prime}} ="33",
(-40,-40) *+{\textcolor{blue}{z_p^\prime}} ="34",
(80,-40) *+{\textcolor{blue}{x_k^\prime}} ="32",
(190,-40) *+{\textcolor{blue}{y_{q}^\prime}} ="35",
(190,0) *+{\textcolor{green}{y_{q}}} ="36",
(210,40) *+{\textcolor{green}{z_{q}}} ="37",
(210,100) *+{\textcolor{blue}{z_{q}^\prime}} ="38",
"1", {\ar"0"},
"0", {\ar"30"},
"4", {\ar"1"},
"20", {\ar"1"},
"0", {\ar"21"},
"0", {\ar"20"},
"6", {\ar"4"},
"6", {\ar"36"},
"21", {\ar"4"},
"6", {\ar"22"},
"1", {\ar"5"},
"7", {\ar"34"},
"7", {\ar"0"},
"0", {\ar"3"},
"3", {\ar"33"},
"3", {\ar"7"},
"5", {\ar"4"},
"32", {\ar"5"},
"36", {\ar"37"},
"36", {\ar"35"},
"37", {\ar"6"},
"37", {\ar"38"},
\end{xy}$$
\caption{The quiver $\overline{R} = \overline{R}_{k+1}$ that we obtain in Case ii).}
\label{mu0subII}
\end{figure}
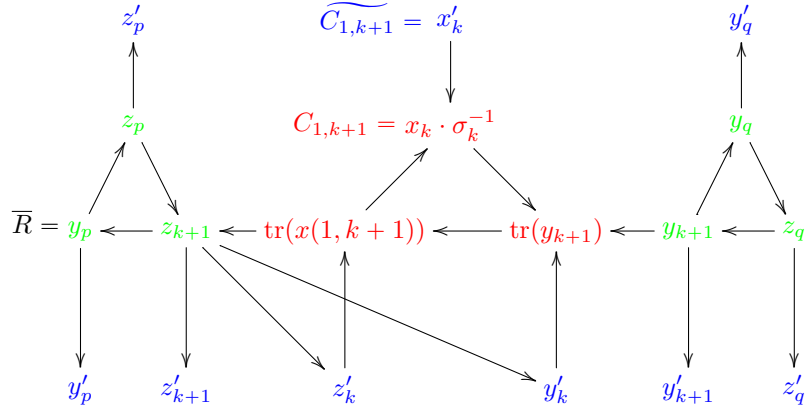

Just as above, we claim that the ice quiver $\overline{R}$ equals $\overline{R}_{k+1}$ and satifies the five bullet points in the statement of  Lemma \ref{newmainlemma}. It is easy to see that $\overline{R}$ is a full subquiver of $\underline{\mu}_k \circ \cdots \circ \underline{\mu}_1\circ \underline{\mu}_0(\widehat{Q})$ that includes every frozen vertex that is connected to a mutable vertex appearing in Figure~\ref{mu0subII} by at least one arrow in $\underline{\mu}_{k}\circ \cdots \circ \underline{\mu}_1 \circ \underline{\mu}_0(\widehat{Q})$. In particular, $\underline{\mu}_k(\overline{R}_k)$ has this property and no vertices of $T_p$ or $T_q$ and neither $y_{k+1}$ nor $z_{k+1}$ have been mutated in $\underline{\mu}_{k}\circ \cdots \circ \underline{\mu}_1 \circ \underline{\mu}_0(\widehat{Q})$. Furthermore, defining $m_1^{(k+1)} \in [n]$, $x(0,k+1)$, and  $x(1,k+1) \in (\widehat{Q})_0$ just as we did in Case i), following Definition~\ref{specialT}, and using the fact that $\text{sgn}(T_{k+1}) = +$, we have $m_1^{(k+1)} = k+1$, $x(0,k+1) = z_{k+1}$, and $x(1,k+1) = x_{k+1}$. Hence we obtain that $$\text{tr}(x(1,k+1)) = \text{tr}(x_{k+1})=\text{tr}(y_k) \ \text{and} \ z^\prime_{m_1^{(k+1)}-1} = z_k^\prime,$$  as desired. Additionally, the fact that $\text{sgn}(T_{k+1}) = +$ also implies that $$\text{tr}(y_{k+1}) = x_{k+1} = y_k \ \text{and} \ x_{m_1^{(k+1)}}^\prime = x_k^\prime = y_k^\prime,$$ as desired. These calculations are reflected in the quiver $\overline{R}$ shown in Figure~\ref{mu0subII}, thus verifying the first three bullet points of Lemma~\ref{newmainlemma}.

Furthermore, since $\text{deg}(z_{k}) = 2,$ there is no signed 3-cycle $T_t$ with $k+1 < t$ such that $\text{tr}|_{k+1,t}(y_t) = z_{k+1}$ in $\underline{\mu}_k\circ \cdots \circ \underline{\mu}_1\circ \underline{\mu}_0(\widehat{Q})$ vertex $x_t$ has been mutated exactly once. The fourth bullet point follows. Now observe that $\widetilde{\text{tr}}(y_k) = x_{m_1}^\prime = x_{k}^\prime$. Since we have applied a maximal green sequence to $\mathcal{Q}_k$ and since $\text{tr}(x(d,k))$ is only connected to the frozen vertex $x_k^\prime$, Proposition 2.10 of \cite{BDP} implies that $\text{tr}(x(d,k)) = x_k\cdot \sigma_{k}^{-1}$. We thus have the fifth bullet point.

Case iii) is similar to Case ii), but with some key differences. In this situation, we again have that $m = k+1$ but this time both $y_\ell$ and $z_\ell$ are relevant. Now Lemma~\ref{Lemma:mukR} implies that $\underline{\mu}_k(\overline{R}_k)$ has the form shown in Figure~\ref{mukRiii}.  We let $T_p$ (resp. $T_q$) be the signed $3$-cycles incident to $z_{k+1}$ (resp. $y_{k+1}$) if they exist.  Define $\overline{R}$ to be the ice quiver that is a full subquiver of $\underline{\mu}_k \circ \cdots \circ \underline{\mu}_1\circ \underline{\mu}_0(\widehat{Q})$ on the vertices $$y_{k+1}, y_{k+1}^\prime, z_{k+1}, z_{k+1}^\prime, \text{tr}(y_k), z_k^\prime, y_k, y_k^\prime, \text{tr}(x(d)), x_k^\prime, y_\ell, y_\ell^\prime, y_{p}, y_{p}^\prime, z_{p}, z_{p}^\prime, y_q, y_q^\prime, z_q, z_q^\prime$$ where we include $y_p$ and $y_p^\prime$ (resp. $z_p$, $z_p^\prime$, $y_q, y_q^\prime, z_q,$ and $z_q^\prime$) in $\overline{R}$ if and only if $y_p$ and $y_p^\prime$ (resp. $z_q$, $z_q^\prime$, $y_q, y_q^\prime, z_q,$ and $z_q^\prime$) appear in $\underline{\mu}_k(\overline{R}_k)$, i.e. depending on if $\text{deg}(y_{k+1})=4$ and if $\text{deg}(z_{k+1}) = 4$.  See Figure \ref{mu0subIII}. 

We claim that the ice quiver $\overline{R}$ has the properties in the statement of Lemma~\ref{newmainlemma}. It is easy to see that $\overline{R}$ is a full subquiver of $\underline{\mu}_k \circ \cdots \circ \underline{\mu}_1\circ \underline{\mu}_0(\widehat{Q}).$ That $\overline{R}$ includes every frozen vertex that is connected to a mutable vertex appearing in Figure~\ref{mu0subIII} by at least one arrow in $\underline{\mu}_{k}\circ \cdots \circ \underline{\mu}_1 \circ \underline{\mu}_0(\widehat{Q})$ follows from the fact that $\underline{\mu}_k(\overline{R}_k)$ has this property and from the fact that no vertices of $T_p$ or $T_q$  and neither $y_{k+1}$ nor $z_{k+1}$ have been mutated in  $\underline{\mu}_{k}\circ \cdots \circ \underline{\mu}_1 \circ \underline{\mu}_0(\widehat{Q})$. Now observe that $\widetilde{\text{tr}}(y_k) = x_{m_1}^\prime = x_{k}^\prime$. As in Case ii), Proposition 2.10 of \cite{BDP} implies that $\text{tr}(x(d,k)) = x_k\cdot \sigma_{k}^{-1}$.



Let $m_1^{(k+1)} \in [n]$ be the integer from the definition of $\underline{\mu}_{k+1}$, and let $x(0,k+1), x(1,k+1) \in (Q)_0$ be the vertices from the definition of $\underline{\mu}_{k+1}.$ As in Case ii), we are using the fact that $\text{sgn}(T_{k+1}) = +$. Now notice that $m_1^{(k+1)} = m_d^{(k+1)} = k+1$, $x(0,k+1) = z_{k+1}$, and $x(1,k+1) = x_{k+1}$. We now obtain that $$\text{tr}(x(d,k+1)) = \text{tr}(x_{k+1})=\text{tr}(y_k),$$ as desired. The fact that $\text{sgn}(T_{k+1}) = +$ implies that $$\text{tr}(y_{k+1}) = x_{k+1} = y_k.$$

Since $\text{deg}(z_{k+1})=4$, the vertices $y_\ell$ and $y_\ell^\prime$ both appear in $\overline{R}$. Now it is clear that $\text{deg}(z_{k+1}) = 4$ if and only if $\text{tr}|_{k+1,\ell}(y_\ell) = z_{k+1}$ and the signed 3-cycle $T_\ell$ has the property that vertex $x_\ell$ has been mutated exactly once in $\underline{\mu}_{k}\circ \cdots \circ \underline{\mu}_1\circ \underline{\mu}_0(\widehat{Q}).$  Hence we see that the vertex $y_\ell$ is positioned in $\overline{R}$ exactly where $y_t$ is positioned in $\overline{R}_{k+1}$, see Figure \ref{importantfullsub1}. These calculations are reflected in the quiver $\overline{R}$ shown in Figure~\ref{mu0subIII}, and we see that this quiver has the properties that the desired quiver $\overline{R}_{k+1}$ should have. The proof of the five bullet points of Lemma \ref{newmainlemma} in Case iii) concludes in the same way as the proof for Case ii).


In addition, we illustrate how in Case iii), for each $c \in [n]$ satisfying $k<c \le \ell$ there is an ice quiver $\overline{R}_{c,\ell}$ that is isomorphic to $\overline{R}_{\ell}$ and that appears as a full subquiver of $\underline{\mu}_{c-1}\circ \cdots \circ \underline{\mu}_1\circ \underline{\mu}_0(\widehat{Q})$. Furthermore, we show that $\overline{R}_{\ell,\ell} = \overline{R}_\ell.$ This analysis will be used in the argument for Case iv), which is given below.

As we are in Case iii), we know that both vertices $y_k$ and $z_k$ are of degree $4$ and the signed 3-cycles $T_k$, $T_{k+1}$, and $T_\ell$ appear in a full subquiver of $\mathcal{Q}$ of the form shown in Figure~\ref{st1-relabel} or \ref{st2-relabel}. It follows that $y_\ell$ and $z_\ell$, which are incident to $z_k$ in $\underline{\mu}_{k-1}\circ \cdots \circ \underline{\mu}_1\circ \underline{\mu}_0(\widehat{{Q}})$, will not be mutated until after applying the mutation sequences $\underline{\mu}_{k}, \underline{\mu}_{k+1}, \dots, \underline{\mu}_{r}$ where $k \le r < \ell$ (see Figure~\ref{scrshot0}). To be precise, the quiver in Figure~\ref{scrshot0} is a full subquiver of $\underline{\mu}_{k-1}\circ \cdots \circ \underline{\mu}_1\circ \underline{\mu}_0(\widehat{Q})$, which we define as follows.  Letting $k < r < \ell$ be the integer such that $z_r = \text{tr}(y_\ell)$ and $e$ such that $x(e,r) = x_{k+1}=y_k$, this full subquiver includes the vertices of $\overline{R}_{k-1}$ as well as the mutable vertices of the signed 3-cycles $T_{m^{(r)}_e} = T_{k+1}, T_{m^{(r)}_{e-1}}, \ldots, T_{m^{(r)}_2}, T_{m^{(r)}_1} = T_r$, as in Definition~\ref{specialT}, and their corresponding frozen vertices.

We now mutate the quiver shown in Figure~\ref{scrshot0} along $\underline{\mu}_k$. By Lemma~\ref{Lemma:mukR}, this does not affect the full subquiver of $\underline{\mu}_{k-1}\circ \cdots \circ \underline{\mu}_1\circ \underline{\mu}_0(\widehat{Q})$ on the vertices $(\widehat{Q})_0\backslash(\overline{R}_k)_0$. Thus we conclude that $\underline{\mu}_{k}\circ \cdots \circ \underline{\mu}_1\circ \underline{\mu}_0(\widehat{Q})$ has the quiver shown in Figure~\ref{scrshot-aftermu} as a full subquiver. We observe that the permutation $\sigma_{k-1}^{-1}$ has the vertices $y_k$ and $z_k$ as fixed points. However, $\tau_k^{-1}$ maps $z_k \mapsto \text{tr}(y_k)$ and fixes $y_k$.  These equalities are illustrated in Figure~\ref{scrshot-aftermu}.

Next, we relabel the vertices of the quiver in Figure~\ref{scrshot-aftermu} to obtain the quiver shown in Figure~\ref{scrshot2}. In particular, since  $\text{sgn}(T_\ell) = -$ with $x_\ell = z_k$, note that $z_k = x(1,\ell)$, $x(s,k) = x(s+1,\ell)$, and $\text{tr}(y_k) = \text{tr}|_{k,\ell}(y_\ell)$. Define $\overline{R}_{k+1,\ell}$ to be the full subquiver of $\underline{\mu}_{k}\circ \cdots \circ \underline{\mu}_1\circ \underline{\mu}_0(\widehat{Q})$ on the red vertices appearing in Figure~\ref{scrshot2}, the neighbors of $y_\ell$ and $z_\ell$, as well as the frozen vertices to which these all are connected. One observes that $\overline{R}_{k+1,\ell}$ and $\overline{R}_\ell$ are isomorphic as ice quivers. Furthermore, we will see that $\overline{R}_{k+1,\ell}$ has the same vertices as $\overline{R}_\ell$ with the exceptions of $y_k$ and $\text{tr}|_{k,\ell}(y_\ell).$

For $k < c \leq \ell$, we define $\overline{R}_{c,\ell}$ analogously as the full subquiver of $\underline{\mu}_{c-1}\circ \cdots \circ \underline{\mu}_1\circ \underline{\mu}_0(\widehat{Q})$ on the set of vertices $(\overline{R}_{k+1,\ell})_0\cdot \tau_{k+1}^{-1} \tau_{k+2}^{-1}\cdots  \tau_{c-1}^{-1}$. With this definition, we observe that $\overline{R}_{c+1,\ell}$ is identical to $\overline{R}_{c,\ell}$ except possibly at two vertices.  In particular, for $k < c \leq \ell$, if $T_c$ does not appear in Figure \ref{st1-relabel} (resp. Figure \ref{st2-relabel}), then the mutation sequence $\underline{\mu}_c$ does not involve any vertices that appear in $\overline{R}_{c,\ell}$.  Consequently, after mutation by $\underline{\mu}_c$, we obtain $\overline{R}_{c+1,\ell} = \overline{R}_{c,\ell}$

On the other hand, when $T_c$, for $k < c \leq \ell$, i.e. $c=m_{s}^{(r)}$ for some $s$, does appear in Figure \ref{st1-relabel} (resp. Figure \ref{st2-relabel}), then the mutation sequence $\underline{\mu}_c$, as indicated by bold arrows in Figures \ref{scrshot2} and \ref{scrshot6}, involves vertices $y_{k}\cdot \sigma^{-1}_{c-1}$ and $z_{k}\cdot \sigma^{-1}_{c-1}$.  In this case, $\overline{R}_{c+1,\ell} \cong \overline{R}_{c,\ell}$ with the relabeling $y_{k}\cdot \sigma^{-1}_{c-1} \mapsto y_k\cdot \sigma^{-1}_{c}$ and $z_k\cdot \sigma^{-1}_{c-1} \mapsto z_k\cdot \sigma^{-1}_{c}$ since each application of $\underline{\mu}_{c}$ permutes these two vertices by $\tau_{c}^{-1}$.  This isomorphism of full subquivers follows from Lemma \ref{Lemma:mukR}.

We obtain the identity $z_k\cdot \sigma^{-1}_{c} = \text{tr}|_{m^{(r)}_{s},\ell}(y_\ell)$ for $m^{(r)}_{s}\le c < m^{(r)}_{s-1}$ when $s \in [2,e]$ or $r = m^{(r)}_{1}\le c < \ell$ when $s=1$, which is implicit in Figure~\ref{scrshot6}, by Lemma~\ref{TRsigma_c}. We leave this argument until after completing the proof of Lemma~\ref{newmainlemma} (see below). We also observe, by the specialization $c=\ell-1$, that $y_k \cdot \sigma^{-1}_{\ell-1} = C_{d+1,\ell}$ and $z_k\cdot \sigma^{-1}_{\ell-1} = \text{tr}(y_\ell)$.  Consequently, we eventually arrive at the configuration in Figure~\ref{scrshot8} with configurations of the form as in Figure  \ref{scrshot6} as intermediate steps.  In summary, we conclude that $\overline{R}_{\ell,\ell} = \overline{R}_\ell$ as desired.

\begin{figure}
$$\raisebox{-.53in}{$\underline{\mu}_k(\overline{R}_k) =$} \begin{xy} 0;<1pt,0pt>:<0pt,-1pt>:: 
(400,40) *+{\textcolor{green}{y_{\ell}}} ="0",
(40,40) *+{\textcolor{red}{z_k}} ="1",
(0,40) *+{\textcolor{green}{z_{\ell}}} ="2",
(20,0) *+{\textcolor{green}{y_t}} ="3",
(80,40) *+{\textcolor{red}{x(1,k)}} ="4",
(60,0) *+{\textcolor{red}{C_{1,k}}} ="5",
(140,40) *+{\textcolor{red}{x(s-1,k)}} ="8",
(190,40) *+{\textcolor{red}{x(s,k)}} ="9",
(165,0) *+{\textcolor{red}{C_{s,k}}} ="10",
(250,40) *+{\textcolor{red}{x(d-1,k)}} ="11",
(105,40) *+{\cdots} ="12",
(215,40) *+{\cdots} ="13",
(310,40) *+{\textcolor{red}{\text{tr}(x(d,k))}} ="14",
(280,0) *+{\textcolor{red}{C_{d,k}}} ="15",
(360,40) *+{\textcolor{red}{\text{tr}(y_k)}} ="16",
(335,0) *+{\textcolor{red}{y_k}} ="17",
(360,0) *+{\textcolor{green}{y_m}} ="18",
(380,0) *+{\textcolor{green}{z_m}} ="19",
(360,100) *+{\textcolor{blue}{z_k^\prime}} ="20",
(40,100) *+{\textcolor{blue}{\widetilde{x}(1,k)}} ="21",
(80,100) *+{\textcolor{blue}{\widetilde{x}(2,k)}} ="22",
(140,100) *+{\textcolor{blue}{\widetilde{x}(s,k)}} ="24",
(190,100) *+{\textcolor{blue}{\widetilde{x}(s+1,k)}} ="25",
(250,100) *+{\textcolor{blue}{\widetilde{\text{tr}}(x(d,k))}} ="26",
(310,100) *+{\textcolor{blue}{\widetilde{\text{tr}}(y_k)}} ="27",
(335,-40) *+{\textcolor{blue}{y_k^\prime}} ="28",
(380,-40) *+{\textcolor{blue}{z_m^\prime}} ="29",
(400,100) *+{\textcolor{blue}{y_\ell^\prime}} ="30",
(0,100) *+{\textcolor{blue}{z_\ell^\prime}} ="31",
(360,-40) *+{\textcolor{blue}{y_m^\prime}} ="32",
(60,-40) *+{\textcolor{blue}{\widetilde{C_{1,k}}}} = "33",
(165,-40) *+{\textcolor{blue}{\widetilde{C_{s,k}}}} = "35",
(20,-40) *+{\textcolor{blue}{y_{t}^\prime}} = "36",
(280,-40) *+{\textcolor{blue}{\widetilde{C_{d,k}}}} = "37",
"0", {\ar"30"},
"0", {\ar"16"},
"1", {\ar"2"},
"1", {\ar"5"},
"3", {\ar"1"},
"4", {\ar"1"},
"2", {\ar"21"},
"2", {\ar"22"},
"2", {\ar"24"},
"2", {\ar"25"},
"2", {\ar"26"},
"2", {\ar"27"},
"2", {\ar"3"},
"2", {\ar"20"},
"2", {\ar"31"},
"5", {\ar"4"},
"21", {\ar"1"},
"8", {\ar"10"},
"14", {\ar"11"},
"14", {\ar"17"},
"16", {\ar"19"},
"19", {\ar"0"},
"20", {\ar"16"},
"22", {\ar"4"},
"10", {\ar"9"},
"9", {\ar"8"},
"11", {\ar"15"},
"24", {\ar"8"},
"25", {\ar"9"},
"15", {\ar"14"},
"16", {\ar"14"},
"26", {\ar"11"},
"17", {\ar"16"},
"27", {\ar"14"},
"18", {\ar"17"},
"19", {\ar"20"},
"28", {\ar"17"},
"18", {\ar"32"},
"19", {\ar"28"},
"19", {\ar"29"},
"33", {\ar"5"},
"35", {\ar"10"},
"3", {\ar"36"},
"37", {\ar"15"},
\end{xy}$$
\caption{The quiver obtained by mutating $\overline{R}_k$ in Case iii).}
\label{mukRiii}
\end{figure}
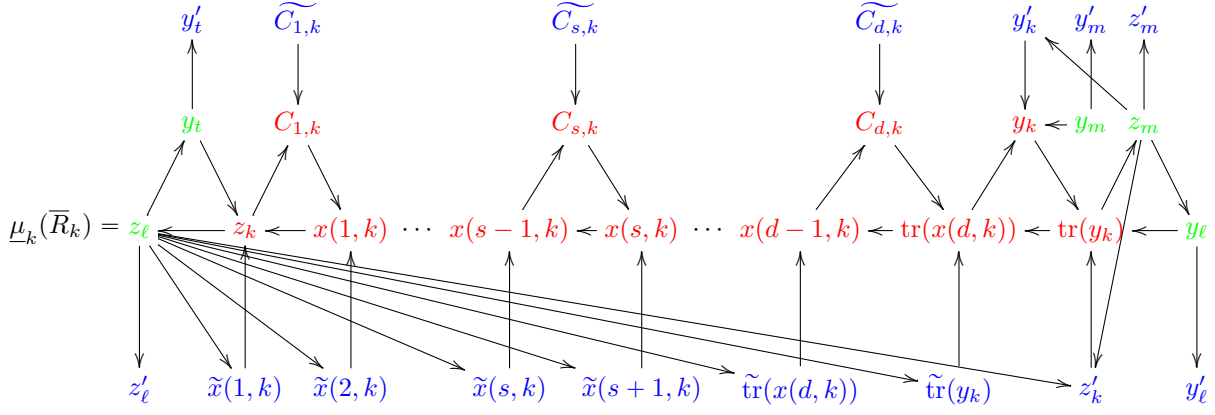

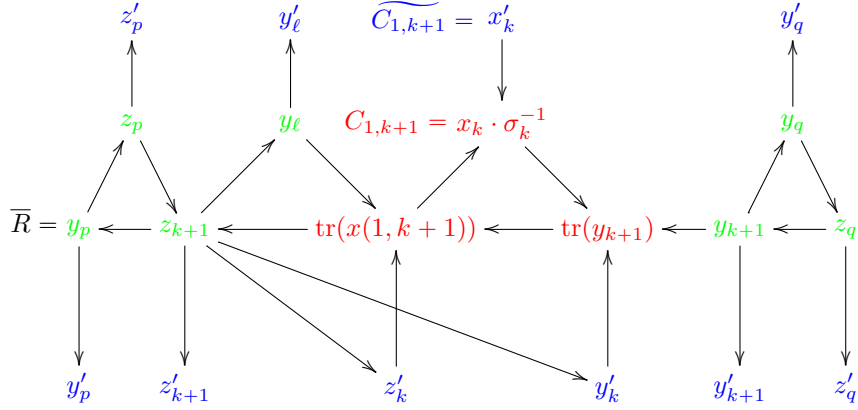
\begin{figure}
$$\raisebox{-.53in}{$\overline{R} =$} \begin{xy} 0;<1pt,0pt>:<0pt,-1pt>:: 
(-40,40) *+{\textcolor{green}{z_{k+1}}} ="0",
(0,0) *+{\textcolor{green}{y_{\ell}}} ="42",
(0,-40) *+{\textcolor{blue}{y_{\ell}^\prime}} ="41",
(-80,40) *+{\textcolor{green}{y_{p}}} ="3",
(-60,0) *+{\textcolor{green}{z_{p}}} ="7",
(40,40) *+{\textcolor{red}{\text{tr}(x(1,k+1))}} ="1",
(120,40) *+{\textcolor{red}{\text{tr}{(y_{k+1})}}} ="4",
(80,0) *+{\textcolor{red}{x_k\cdot\sigma_{k}^{-1}}} ="5",
(40,0) *+{\textcolor{red}{C_{1,k+1} =}} ="39",
(50,-40) *+{\textcolor{blue}{\widetilde{C_{1,k+1}} =}} ="40",
(170,40) *+{\textcolor{green}{y_{k+1}}} ="6",
(40,100) *+{\textcolor{blue}{z_k^\prime}} ="20",
(120,100) *+{\textcolor{blue}{y^\prime_{k}}} ="21",
(170,100) *+{\textcolor{blue}{y_{k+1}^\prime}} ="22",
(-40,100) *+{\textcolor{blue}{z_{k+1}^\prime}} ="30",
(-80,100) *+{\textcolor{blue}{y_{p}^\prime}} ="33",
(-60,-40) *+{\textcolor{blue}{z_p^\prime}} ="34",
(80,-40) *+{\textcolor{blue}{x_k^\prime}} ="32",
(190,-40) *+{\textcolor{blue}{y_{q}^\prime}} ="35",
(190,0) *+{\textcolor{green}{y_{q}}} ="36",
(210,40) *+{\textcolor{green}{z_{q}}} ="37",
(210,100) *+{\textcolor{blue}{z_{q}^\prime}} ="38",
"1", {\ar"0"},
"0", {\ar"30"},
"0", {\ar"42"},
"4", {\ar"1"},
"20", {\ar"1"},
"0", {\ar"21"},
"0", {\ar"20"},
"6", {\ar"4"},
"6", {\ar"36"},
"21", {\ar"4"},
"6", {\ar"22"},
"1", {\ar"5"},
"7", {\ar"34"},
"7", {\ar"0"},
"0", {\ar"3"},
"3", {\ar"33"},
"3", {\ar"7"},
"5", {\ar"4"},
"32", {\ar"5"},
"36", {\ar"37"},
"36", {\ar"35"},
"37", {\ar"6"},
"37", {\ar"38"},
"42", {\ar"41"},
"42", {\ar"1"},
\end{xy}$$
\caption{The quiver $\overline{R} = \overline{R}_{k+1}$ that we obtain in Case iii).}
\label{mu0subIII}
\end{figure}

Next, suppose we are in Case iv).  Since deg$(z_k)=2$, this case is similar to Case i).  However, here we have deg$(y_k)=2$ as well, and so the quiver $\underline{\mu}_k(\overline{R}_k)$ looks like Figure~\ref{mukR-reorganized}, but without $y_{\ell}, y_{\ell}^\prime$, $z_{\ell}$, $z_{\ell}^\prime$, $y_m, y_m^\prime, z_m,$ nor $z_m^\prime$.  The green vertex $y_t$ and $y_t^\prime$ may or may not appear in the quiver $\underline{\mu}_k(\overline{R}_k)$. In the latter case, $k = n$ and we have applied the entire mutation sequence $\underline{\mu}$ to $\widehat{Q}$. In the former case, we see that $t = k+1$, and $T_t$ can be realized as a signed 3-cycle $T_\ell$ appearing in one of Figures~\ref{st1-relabel} or \ref{st2-relabel}. Now by the argument at the end of Case iii), $\overline{R}_{t,t} = \overline{R}_{t}$ is indeed a full subquiver of Figure~\ref{scrshot8} with the desired properties. The five bullet points of Lemma \ref{newmainlemma} follow immediately. 

Lastly, for all four cases, we wish to describe the quiver obtained by 
$\underline{\mu}_{k} \circ \cdots \circ \underline{\mu}_1\circ \underline{\mu}_0(\widehat{\mathcal{Q}_{k,k}})$.  To this end, we decompose the vertices of 
$\widehat{\mathcal{Q}_{k,k}}$ into two sets:
(1) $(\widehat{\mathcal{Q}_{k,k}})_0 \setminus (\overline{R}_k)_0$ and (2) $(\overline{R}_k)_0 \cap (\widehat{\mathcal{Q}_{k,k}})_0$.  
By induction, we have 
$$\underline{\mu}_{k-1} \circ \cdots \circ \underline{\mu}_1\circ \underline{\mu}_0(\widehat{\mathcal{Q}_{k-1,k-1}}) = \widecheck{\mathcal{Q}_{k-1,k-1}}\cdot \sigma_{k-1}$$ and we observe that 
$(\widehat{\mathcal{Q}_{k,k}})_0 \setminus (\overline{R}_k)_0 \subset (\widehat{\mathcal{Q}_{k-1,k-1}})_0$.  The fifth bullet point of Lemma ~\ref{Lemma:mukR} implies that the quiver 
$\underline{\mu}_{k-1} \circ \cdots \circ \underline{\mu}_1\circ \underline{\mu}_0\left(\widehat{\mathcal{Q}_{k,k}}|_{(\widehat{\mathcal{Q}_{k,k}})_0\setminus (\overline{R}_k)_0}\right)\footnote{We define $\widehat{\mathcal{Q}_{k,k}}|_{(\widehat{\mathcal{Q}_{k,k}})_0\setminus(\overline{R}_k)_0}$ (resp. $\widecheck{\mathcal{Q}_{k,k}}|_{(\widehat{\mathcal{Q}_{k,k}})_0\setminus(\overline{R}_k)_0}$)  to be the ice quiver that is a full subquiver of $\widehat{\mathcal{Q}_{k,k}}$ (resp. $\widecheck{\mathcal{Q}_{k,k}}$) on the vertices of ${(\widehat{\mathcal{Q}_{k,k}})_0\setminus(\overline{R}_k)_0}$.}$
 is unchanged by the mutation sequence $\underline{\mu}_k$, and the permutation $\tau_k$ fixes all vertices in $(\widehat{\mathcal{Q}_{k,k}})_0\setminus (\overline{R}_k)_0$.  It follows that 
$$\underline{\mu}_{k} \circ \cdots \circ \underline{\mu}_1\circ \underline{\mu}_0\left(\widehat{\mathcal{Q}_{k,k}}|_{(\widehat{\mathcal{Q}_{k,k}})_0\setminus (\overline{R}_k)_0}\right) = \left(\widecheck{\mathcal{Q}_{k,k}}|_{(\widehat{\mathcal{Q}_{k,k}})_0\setminus (\overline{R}_k)_0}\right)\cdot \sigma_{k}.$$
Additionally, the first bullet point of Lemma~\ref{Lemma:mukR} indicates how the vertices of the second set, i.e. $(\overline{R}_k)_0 \cap (\widehat{\mathcal{Q}_{k,k}})_0$, are affected by $\underline{\mu}_k$.  Comparing Figures~\ref{importantfullsub1} and \ref{mukR}, we see that the vertices of $\overline{R}_{k}$ have been permuted cyclically exactly as described by $\tau_k$.  We conclude that 
$$\underline{\mu}_{k}\circ \cdots \circ \underline{\mu}_1\circ \underline{\mu}_0(\widehat{\mathcal{Q}_{k,k}}) =
\widecheck{\mathcal{Q}_{k,k}}\cdot \tau_k\sigma_{k-1}
= \widecheck{\mathcal{Q}_{k,k}}\cdot \sigma_k
$$ which completes the proof of Lemma \ref{newmainlemma}.\end{proof}

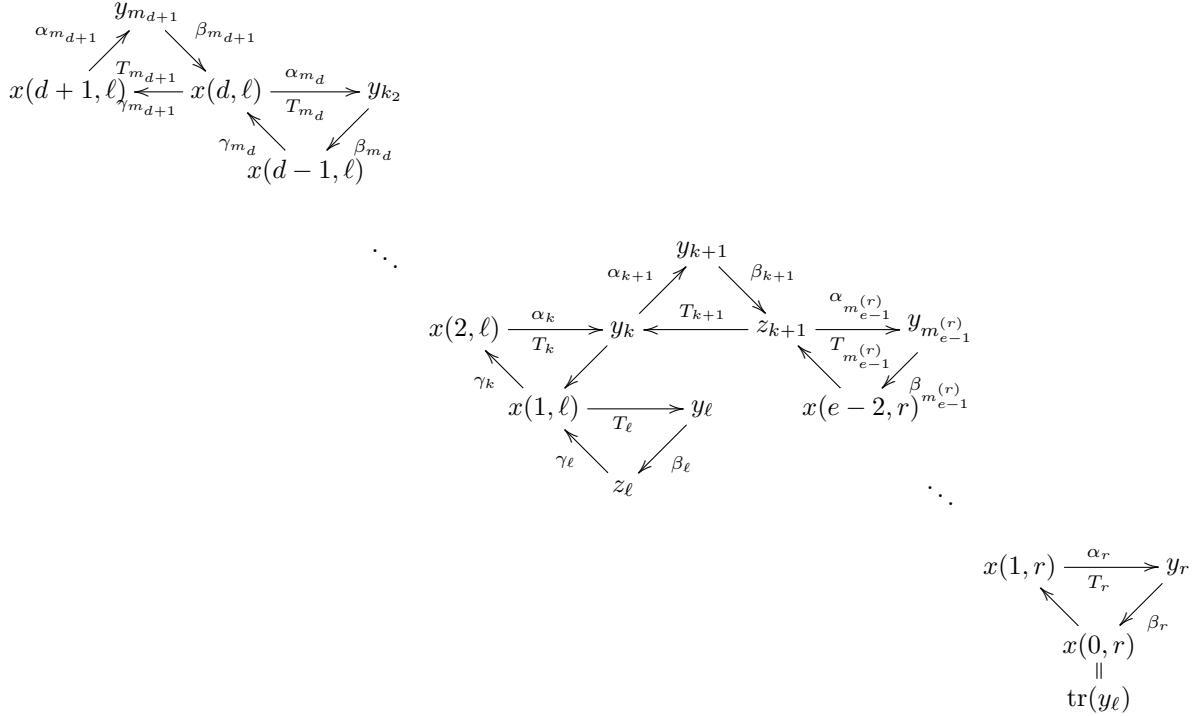
\begin{figure}[h]
$$\begin{xy} 0;<1pt,0pt>:<0pt,-1pt>:: 
(0,30) *+{x({d+1,\ell})} ="0",
(30,0) *+{y_{m_{d+1}}} ="1",
(60,30) *+{x{(d,\ell)}} ="2",
(240,90) *+{y_{k+1}} ="3",
(270,120) *+{z_{k+1}} ="4",
(120,30) *+{y_{k_2}} ="5",
(90,60) *+{x({d-1,\ell})} ="6",
(120,90) *+{\ddots} ="7",
(210,180) *+{z_\ell} ="8",
(240,150) *+{y_\ell} ="10",
(330,120) *+{y_{{m^{(r)}_{e-1}}}} ="11",
(300,150) *+{x(e-2,r)} ="12",
(330,180) *+{\ddots} ="13",
(390,240) *+{x(0,r)} ="14",
(360,210) *+{x(1,r)} ="15",
(420,210) *+{y_r} ="16",
(210,120) *+{y_{k}} ="17",
(150,120) *+{x(2,\ell)} ="18",
(180,150) *+{x(1,\ell)} ="19",
(390,250) *+{\equalto{}{}} ="21",
(390,260) *+{\text{tr}(y_\ell)} ="22",
"0", {\ar^{\alpha_{m_{d+1}}}"1"},
"2", {\ar^{\gamma_{m_{d+1}}}_{T_{m_{d+1}}}"0"},
"1", {\ar^{\beta_{m_{d+1}}}"2"},
"2", {\ar^{\alpha_{m_d}}_{T_{m_{d}}}"5"},
"6", {\ar^{\gamma_{m_{d}}}"2"},
"3", {\ar^{\beta_{k+1}}"4"},
"17", {\ar^{\alpha_{k+1}}"3"},
"4", {\ar^{\alpha_{m^{(r)}_{e-1}}}_{T_{{m^{(r)}_{e-1}}}}"11"},
"12", {\ar"4"},
"4", {\ar_{T_{k+1}}"17"},
"5", {\ar^{\beta_{m_d}}"6"},
"8", {\ar^{\gamma_{\ell}}"19"},
"10", {\ar^{\beta_\ell}"8"},
"19", {\ar_{T_{\ell}}"10"},
"11", {\ar^{\beta_{m^{(r)}_{e-1}}}"12"},
"14", {\ar"15"},
"16", {\ar^{\beta_{r}}"14"},
"15", {\ar^{\alpha_r}_{T_r}"16"},
"18", {\ar^{\alpha_{k}}_{T_{k}}"17"},
"17", {\ar"19"},
"19", {\ar^{\gamma_{k}}"18"},
\end{xy}$$
\caption{A full subquiver of $\mathcal{Q}$ showing one possible configuration of the signed 3-cycles $T_k$, $T_{k+1}$, and $T_\ell$, as described in the proof of Lemma~\ref{newmainlemma} at the end of Case iii).}
\label{st1-relabel}
\end{figure}

\begin{figure}[h]
$$\begin{xy} 0;<1pt,0pt>:<0pt,-1pt>:: 
(-10,65) *+{x(2,\ell)} ="0",
(30,30) *+{y_{k}} ="1",
(70,65) *+{x(1,\ell)} ="2",
(75,-10) *+{y_{k+1}} ="3",
(120,30) *+{z_{k+1}} ="4",
(140,65) *+{y_{\ell}} ="5",
(105,100) *+{z_\ell} ="6",
(210,30) *+{y_{{m^{(r)}_{e-1}}}} ="11",
(165,70) *+{x(e-2,r)} ="12",
(190,90) *+{\ddots} ="13",
(210,120) *+{x(1,r)} ="14",
(300,120) *+{y_{r}} ="15",
(255,160) *+{x(0,r)} ="16",
(255,170) *+{\equalto{}{}} ="21",
(255,180) *+{\text{tr}(y_\ell)} ="22",
"0", {\ar^{\alpha_{k}}"1"},
"2", {\ar^{\gamma_{k}}_{T_{k}}"0"},
"1", {\ar"2"},
"1", {\ar^{\alpha_{k+1}}"3"},
"4", {\ar_{T_{k+1}}"1"},
"2", {\ar_{T_{\ell}}"5"},
"6", {\ar^{\gamma_{\ell}}"2"},
"3", {\ar^{\beta_{k+1}}"4"},
"4", {\ar^{\alpha_{m^{(r)}_{e-1}}}_{T_{m^{(r)}_{e-1}}}"11"},
"12", {\ar"4"},
"5", {\ar^{\beta_\ell}"6"},
"11", {\ar^{\beta_{m^{(r)}_{e-1}}}"12"},
"14", {\ar^{\alpha_{r}}_{T_{r}}"15"},
"16", {\ar"14"},
"15", {\ar^{\beta_{r}}"16"},
\end{xy}$$
\caption{A full subquiver of $\mathcal{Q}$ showing the other possible configuration of the signed 3-cycles $T_k$, $T_{k+1}$, and $T_\ell$, as described in the proof of Lemma~\ref{newmainlemma} at the end of Case iii).}
\label{st2-relabel}
\end{figure}
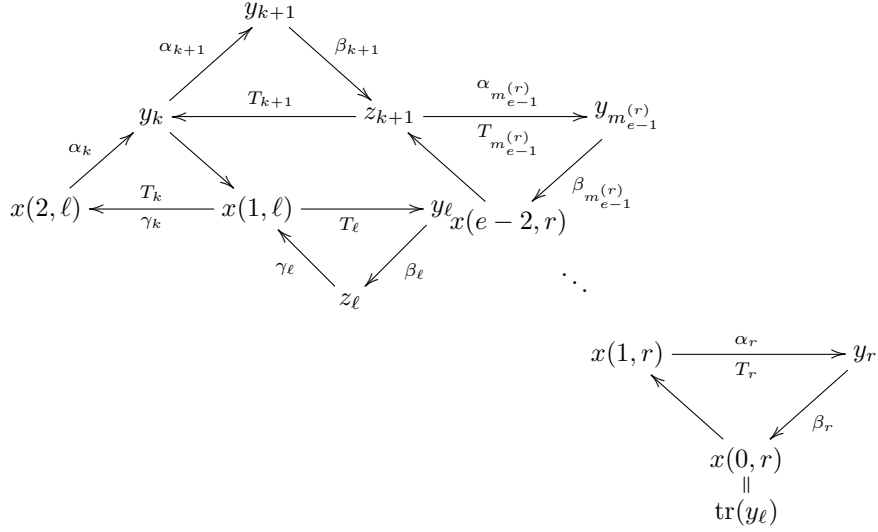


\begin{lemma}\label{TRsigma_c}
Using the notation from the proof of Lemma~\ref{newmainlemma}, for any $s \in [2,e]$ and any $c \in [n]$ satisfying $m^{(r)}_{s}\le c < m^{(r)}_{s-1}$, one has $z_k\cdot \sigma^{-1}_{c} = \text{tr}|_{m^{(r)}_{s},\ell}(y_\ell)$ (see Figure~\ref{scrshot6}). Additionally, for any $c \in [n]$ satisfying $r = m_1^{(r)} \le c < \ell$ we have $z_k\cdot\sigma_c^{-1} = \text{tr}(y_\ell)$.
\end{lemma}
\begin{proof}
For $c = k+1$, we have 

$$\begin{array}{rcll}
z_k\cdot \sigma_{k+1}^{-1} & = & z_k\cdot\sigma_k^{-1}\tau^{-1}_{k+1}\\ & = & \text{tr}(y_k) \cdot \tau^{-1}_{k+1} & \text{(see Figure~\ref{scrshot-aftermu})}\\
& = & \text{tr}(x(1,k+1))\cdot\tau^{-1}_{k+1} & \text{(using that $\text{sgn}(T_{k+1}) = +$)}\\
& = & x(0,k+1) & \text{(by the definition of $\tau_{k+1}$)}\\ 
& = & z_k & \text{(by Definition~\ref{specialT})}\\
& = & \text{tr}|_{k+1,\ell}(y_\ell), & \text{(by Definition~\ref{transportDEF})}
\end{array}$$ 

\noindent as desired. Now suppose that $z_k\cdot \sigma_{c}^{-1} = \text{tr}|_{m^{(r)}_s,\ell}(y_\ell)$ where $s \in [e]$ and $m^{(r)}_{s}\le c < m^{(r)}_{s-1}$. Then for $c \in [n]$ satisfying $m^{(r)}_{s-1}\le c < m^{(r)}_{s-2}$ we have 

$$\begin{array}{rcll}
z_k\cdot \sigma_{c}^{-1} & = & z_k \cdot \sigma_{m^{(r)}_{s-1}-1}^{-1}\tau^{-1}_{{m^{(r)}_{s-1}}} \cdots\tau^{-1}_{c}\\ & = & \text{tr}|_{m_s^{(r)},\ell}(y_\ell) \cdot\tau^{-1}_{{m^{(r)}_{s-1}}} \cdots\tau^{-1}_{c} & \text{(by induction)}\\
& = & x(1,m^{(r)}_{s-1})\cdot\tau^{-1}_{{m^{(r)}_{s-1}}}\tau^{-1}_{{m^{(r)}_{s-1}}+1} \cdots\tau^{-1}_{c} & \text{(note that $x(1,m^{(r)}_{s-1}) = x(s-1,r)$)}\\
& = & x(0,m^{(r)}_{s-1})\cdot\tau^{-1}_{{m^{(r)}_{s-1}}+1} \cdots\tau^{-1}_{c} & \text{(note that $x(0,m^{(r)}_{s-1}) = x(s-2,r)$)}\\
& = & \text{tr}|_{m_{s-1}^{(r)},\ell}(y_\ell)\cdot\tau^{-1}_{{m^{(r)}_{s-1}}+1} \cdots\tau^{-1}_{c} & \text{(note that $x(0,m^{(r)}_{s-1}) = \text{tr}|_{m_{s-1}^{(r)},\ell}(y_\ell)$)}\\
& = & \text{tr}|_{m_{s-1}^{(r)},\ell}(y_\ell),
\end{array}$$ 

\noindent as desired. We remark that the last equality in the previous computation follows from observing that $\text{tr}|_{m_{s-1}^{(r)},\ell}(y_\ell)$ is not mutated in any of the mutation sequences $\underline{\mu}_{m^{(r)}_{s-1}+1}, \ldots, \underline{\mu}_c$, and thus it is unaffected by any of the permutations $\tau^{-1}_{m^{(r)}_{s-1}+1}, \ldots, \tau^{-1}_c$. By induction, this completes the proof.
\end{proof}

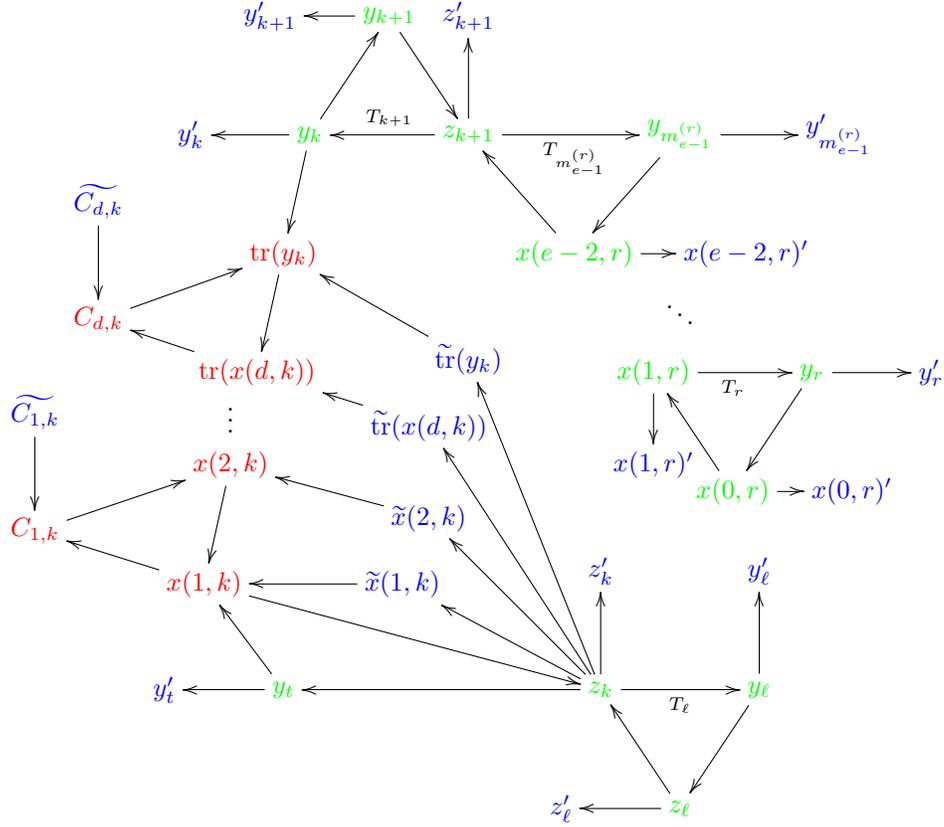
\begin{figure}
$$\begin{xy} 0;<1pt,0pt>:<0pt,-1pt>:: 
(-22,195) *+{\textcolor{red}{C_{1,k}}} ="0",
(52,170) *+{\textcolor{red}{x(2,k)}} ="1",
(62,135) *+{\textcolor{red}{\text{tr}(x(d,k))}} ="40",
(52,150) *+{\vdots} ="42",
(127,155) *+{\textcolor{blue}{\widetilde{\text{tr}}(x(d,k))}} ="41",
(72,90) *+{\textcolor{red}{\text{tr}(y_k)}} ="2",
(2,115) *+{\textcolor{red}{C_{d,k}}} ="3",
(42,215) *+{\textcolor{red}{x(1,k)}} ="4",
(112,0) *+{\textcolor{green}{y_{k+1}}} ="5",
(142,45) *+{\textcolor{green}{z_{k+1}}} ="6",
(222,45) *+{\textcolor{green}{y_{m^{(r)}_{e-1}}}} ="7",
(182,90) *+{\textcolor{green}{x(e-2,r)}} ="8",
(222,110) *+{\ddots} ="9",
(212,135) *+{\textcolor{green}{x(1,r)}} ="10",
(272,135) *+{\textcolor{green}{y_r}} ="11",
(242,180) *+{\textcolor{green}{x(0,r)}} ="12",
(82,45) *+{\textcolor{green}{y_k}} ="13",
(192,255) *+{\textcolor{green}{z_k}} ="14",
(72,255) *+{\textcolor{green}{y_t}} ="48",
(27,255) *+{\textcolor{blue}{y_t^\prime}} ="49",
(-22,150) *+{\textcolor{blue}{\widetilde{C_{1,k}}}} ="15",
(127,190) *+{\textcolor{blue}{\widetilde{x}(2,k)}} ="16",
(117,215) *+{\textcolor{blue}{\widetilde{x}(1,k)}} ="17",
(2,70) *+{\textcolor{blue}{\widetilde{C_{d,k}}}} ="18",
(142,130) *+{\textcolor{blue}{\widetilde{\text{tr}}(y_k)}} ="19",
(67,0) *+{\textcolor{blue}{y_{k+1}^\prime}} ="20",
(142,0) *+{\textcolor{blue}{z_{k+1}^\prime}} ="21",
(282,45) *+{\textcolor{blue}{y_{m^{(r)}_{e-1}}^\prime}} ="22",
(247,90) *+{\textcolor{blue}{x(e-2,r)^\prime}} ="23",
(212,170) *+{\textcolor{blue}{x(1,r)^\prime}} ="25",
(317,135) *+{\textcolor{blue}{y_r^\prime}} ="26",
(287,180) *+{\textcolor{blue}{x(0,r)^\prime}} ="27",
(37,45) *+{\textcolor{blue}{y_k^\prime}} ="28",
(192,210) *+{\textcolor{blue}{z_k^\prime}} ="29",
(252,255) *+{\textcolor{green}{y_\ell}} ="43",
(252,210) *+{\textcolor{blue}{y_\ell^\prime}} ="44",
(222,300) *+{\textcolor{green}{z_\ell}} ="45",
(177,300) *+{\textcolor{blue}{z_\ell^\prime}} ="46",
"0", {\ar"1"},
"4", {\ar"0"},
"15", {\ar"0"},
"2", {\ar"40"},
"40", {\ar"3"},
"1", {\ar"4"},
"16", {\ar"1"},
"3", {\ar"2"},
"13", {\ar"2"},
"19", {\ar"2"},
"5", {\ar"6"},
"18", {\ar"3"},
"4", {\ar"14"},
"17", {\ar"4"},
"5", {\ar"20"},
"6", {\ar_{T_{m^{(r)}_{e-1}}}"7"},
"8", {\ar"6"},
"6", {\ar_{T_{k+1}}"13"},
"6", {\ar"21"},
"7", {\ar"8"},
"7", {\ar"22"},
"8", {\ar"23"},
"10", {\ar_{T_{r}}"11"},
"12", {\ar"10"},
"10", {\ar"25"},
"11", {\ar"12"},
"11", {\ar"26"},
"12", {\ar"27"},
"13", {\ar"28"},
"13", {\ar"5"},
"14", {\ar"16"},
"14", {\ar"17"},
"14", {\ar"19"},
"14", {\ar"29"},
"14", {\ar"41"},
"14", {\ar"48"},
"14", {\ar_{T_\ell}"43"},
"41", {\ar"40"},
"43", {\ar"44"},
"43", {\ar"45"},
"45", {\ar"46"},
"45", {\ar"14"},
"48", {\ar"4"},
"48", {\ar"49"},
\end{xy}$$
\caption{The full subquiver of $\underline{\mu}_{k-1}\circ \cdots \circ \underline{\mu}_1\circ \underline{\mu}_0(\widehat{Q})$ on the vertices and frozen vertices shown here.} 
\label{scrshot0}
\end{figure}

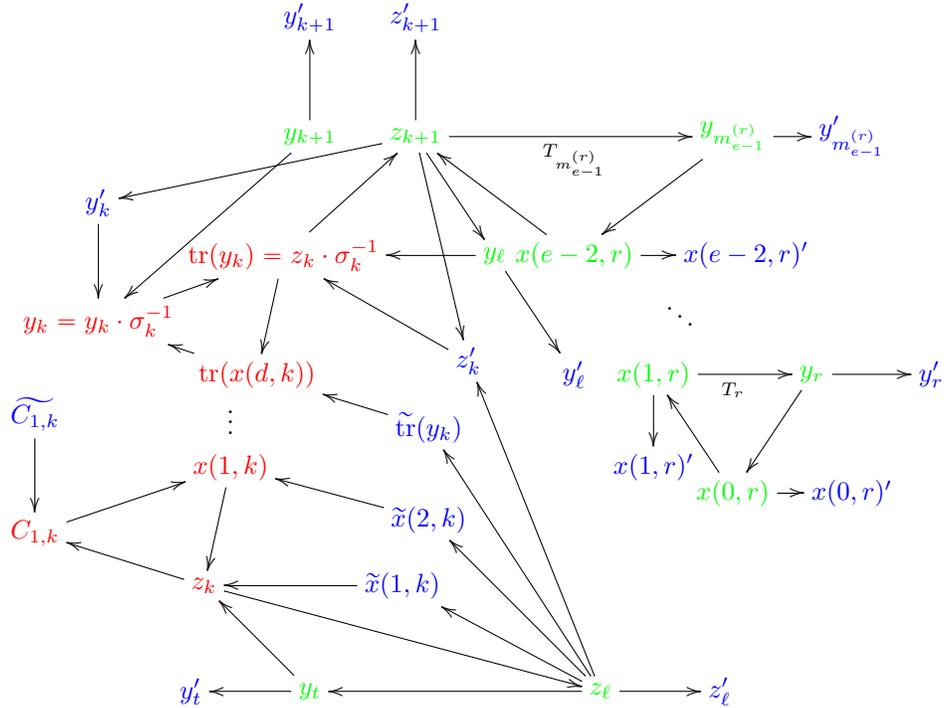
\begin{figure}
$$\begin{xy} 0;<1pt,0pt>:<0pt,-1pt>:: 
(-22,195) *+{\textcolor{red}{C_{1,k}}} ="0",
(52,170) *+{\textcolor{red}{x(1,k)}} ="1",
(62,135) *+{\textcolor{red}{\text{tr}(x(d,k))}} ="40",
(52,150) *+{\vdots} ="42",
(127,155) *+{\textcolor{blue}{\widetilde{\text{tr}}(y_k)}} ="41",  
(72,90) *+{\textcolor{red}{\text{tr}(y_k)=z_k\cdot \sigma_{k}^{-1}}} ="2",
(2,115) *+{\textcolor{red}{y_k = y_k\cdot \sigma_{k}^{-1}}} ="3",
(42,215) *+{\textcolor{red}{z_k}} ="4",
(82,45) *+{\textcolor{green}{y_{k+1}}} ="5",
(122,45) *+{\textcolor{green}{z_{k+1}}} ="6",
(242,45) *+{\textcolor{green}{y_{m^{(r)}_{e-1}}}} ="7",
(182,90) *+{\textcolor{green}{x(e-2,r)}} ="8",
(222,110) *+{\ddots} ="9",
(212,135) *+{\textcolor{green}{x(1,r)}} ="10",
(272,135) *+{\textcolor{green}{y_r}} ="11",
(242,180) *+{\textcolor{green}{x(0,r)}} ="12",
(152,90) *+{\textcolor{green}{y_\ell}} ="13",
(192,255) *+{\textcolor{green}{z_\ell}} ="14",
(82,255) *+{\textcolor{green}{y_t}} ="50",
(37,255) *+{\textcolor{blue}{y_t^\prime}} ="51",
(-22,150) *+{\textcolor{blue}{\widetilde{C_{1,k}}}} ="15",
(127,190) *+{\textcolor{blue}{\widetilde{x}(2,k)}} ="16",
(117,215) *+{\textcolor{blue}{\widetilde{x}(1,k)}} ="17",
(2,70) *+{\textcolor{blue}{y_k'}} ="18", 
(142,130) *+{\textcolor{blue}{z_k'}} ="19", 
(82,0) *+{\textcolor{blue}{y_{k+1}^\prime}} ="20",
(122,0) *+{\textcolor{blue}{z_{k+1}^\prime}} ="21",
(287,45) *+{\textcolor{blue}{y_{m^{(r)}_{e-1}}^\prime}} ="22",
(247,90) *+{\textcolor{blue}{x(e-2,r)^\prime}} ="23",
(212,170) *+{\textcolor{blue}{x(1,r)^\prime}} ="25",
(317,135) *+{\textcolor{blue}{y_r^\prime}} ="26",
(287,180) *+{\textcolor{blue}{x(0,r)^\prime}} ="27",
(182,135) *+{\textcolor{blue}{y_\ell^\prime}} ="28",
(237,255) *+{\textcolor{blue}{z_\ell^\prime}} ="29",
"0", {\ar"1"},
"4", {\ar"0"},
"15", {\ar"0"},
"2", {\ar"40"},
"40", {\ar"3"},
"1", {\ar"4"},
"16", {\ar"1"},
"3", {\ar"2"},
"2", {\ar"6"},
"13", {\ar"2"},
"19", {\ar"2"},
"5", {\ar"3"},
"18", {\ar"3"},
"4", {\ar"14"},
"17", {\ar"4"},
"5", {\ar"20"},
"6", {\ar_{T_{m^{(r)}_{e-1}}}"7"},
"8", {\ar"6"},
"6", {\ar"13"},
"6", {\ar"18"},
"6", {\ar"19"},
"6", {\ar"21"},
"7", {\ar"8"},
"7", {\ar"22"},
"8", {\ar"23"},
"10", {\ar_{T_{r}}"11"},
"12", {\ar"10"},
"10", {\ar"25"},
"11", {\ar"12"},
"11", {\ar"26"},
"12", {\ar"27"},
"13", {\ar"28"},
"14", {\ar"16"},
"14", {\ar"17"},
"14", {\ar"19"},
"14", {\ar"29"},
"14", {\ar"41"},
"14", {\ar"50"},
"41", {\ar"40"},
"50", {\ar"4"},
"50", {\ar"51"},
\end{xy}$$
\caption{The full subquiver of $\underline{\mu}_{k}\circ \cdots \circ \underline{\mu}_1\circ \underline{\mu}_0(\widehat{Q})$ on the vertices and frozen vertices shown here.}
\label{scrshot-aftermu}
\end{figure}

\begin{figure}
$$\begin{xy} 0;<1pt,0pt>:<0pt,-1pt>:: 
(-22,195) *+{\textcolor{red}{C_{1,k}}} ="0",
(52,170) *+{\textcolor{red}{x(2,\ell)}} ="1",
(62,135) *+{\textcolor{red}{\text{tr}(x(d+1,\ell))}} ="40",
(52,150) *+{\vdots} ="42",
(127,155) *+{\textcolor{blue}{\widetilde{\text{tr}}(x(d+1,\ell))}} ="41",
(72,90) *+{\textcolor{red}{\text{tr}|_{k,\ell}(y_\ell) = z_k\cdot \sigma_k^{-1}}} ="2",
(2,115) *+{\textcolor{red}{y_k = y_k\cdot \sigma_k^{-1}}} ="3",
(42,215) *+{\textcolor{red}{x(1,\ell)}} ="4",
(82,45) *+{\textcolor{green}{y_{k+1}}} ="5",
(122,45) *+{\textcolor{green}{z_{k+1}}} ="6",
(242,45) *+{\textcolor{green}{y_{m^{(r)}_{e-1}}}} ="7",
(182,90) *+{\textcolor{green}{x(e-2,r)}} ="8",
(222,110) *+{\ddots} ="9",
(212,135) *+{\textcolor{green}{x(1,r)}} ="10",
(272,135) *+{\textcolor{green}{y_r}} ="11",
(242,180) *+{\textcolor{green}{x(0,r)}} ="12",
(152,90) *+{\textcolor{green}{y_\ell}} ="13",
(192,255) *+{\textcolor{green}{z_\ell}} ="14",
(82,255) *+{\textcolor{green}{y_t}} ="50",
(37,255) *+{\textcolor{blue}{y_t^\prime}} ="51",
(-22,150) *+{\textcolor{blue}{\widetilde{C_{1,k}}}} ="15",
(127,190) *+{\textcolor{blue}{\widetilde{x}(2,\ell)}} ="16",
(117,215) *+{\textcolor{blue}{\widetilde{x}(1,\ell)}} ="17",
(2,70) *+{\textcolor{blue}{y_k^\prime}} ="18",
(152,130) *+{\textcolor{blue}{z_k^\prime}} ="19",
(82,0) *+{\textcolor{blue}{y_{k+1}^\prime}} ="20",
(122,0) *+{\textcolor{blue}{z_{k+1}^\prime}} ="21",
(287,45) *+{\textcolor{blue}{y_{m^{(r)}_{e-1}}^\prime}} ="22",
(247,90) *+{\textcolor{blue}{x(e-2,r)^\prime}} ="23",
(212,170) *+{\textcolor{blue}{x(1,r)^\prime}} ="25",
(317,135) *+{\textcolor{blue}{y_r^\prime}} ="26",
(287,180) *+{\textcolor{blue}{x(0,r)^\prime}} ="27",
(182,135) *+{\textcolor{blue}{y_\ell^\prime}} ="28",
(237,255) *+{\textcolor{blue}{z_\ell^\prime}} ="29",
"0", {\ar"1"},
"4", {\ar"0"},
"15", {\ar"0"},
"2", {\ar"40"},
"40", {\ar"3"},
"1", {\ar"4"},
"16", {\ar"1"},
"3", {\ar@<0.1mm>"2"},
"3", {\ar@<0.05mm>"2"},
"3", {\ar"2"},
"3", {\ar@<-0.05mm>"2"},
"3", {\ar@<-0.1mm>"2"},
"2", {\ar@<0.1mm>"6"},
"2", {\ar@<0.05mm>"6"},
"2", {\ar"6"},
"2", {\ar@<-0.05mm>"6"},
"2", {\ar@<-0.1mm>"6"},
"13", {\ar"2"},
"19", {\ar"2"},
"5", {\ar@<0.1mm>"3"},
"5", {\ar@<0.05mm>"3"},
"5", {\ar"3"},
"5", {\ar@<-0.05mm>"3"},
"5", {\ar@<-0.1mm>"3"},
"18", {\ar"3"},
"4", {\ar"14"},
"17", {\ar"4"},
"5", {\ar"20"},
"6", {\ar_{T_{m^{(r)}_{e-1}}}"7"},
"8", {\ar"6"},
"6", {\ar"13"},
"6", {\ar"18"},
"6", {\ar"19"},
"6", {\ar"21"},
"7", {\ar"8"},
"7", {\ar"22"},
"8", {\ar"23"},
"10", {\ar_{T_{r}}"11"},
"12", {\ar"10"},
"10", {\ar"25"},
"11", {\ar"12"},
"11", {\ar"26"},
"12", {\ar"27"},
"13", {\ar"28"},
"14", {\ar"16"},
"14", {\ar"17"},
"14", {\ar"19"},
"14", {\ar"29"},
"14", {\ar"41"},
"14", {\ar"50"},
"41", {\ar"40"},
"50", {\ar"4"},
"50", {\ar"51"},
\end{xy}$$
\caption{The quiver that appears in Figure~\ref{scrshot-aftermu} with its vertex labels updated so that the part of $\overline{R}_{k+1,\ell}$ that appears here looks like the corresponding part of the quiver $\overline{R}_\ell$.}
\label{scrshot2}
\end{figure}
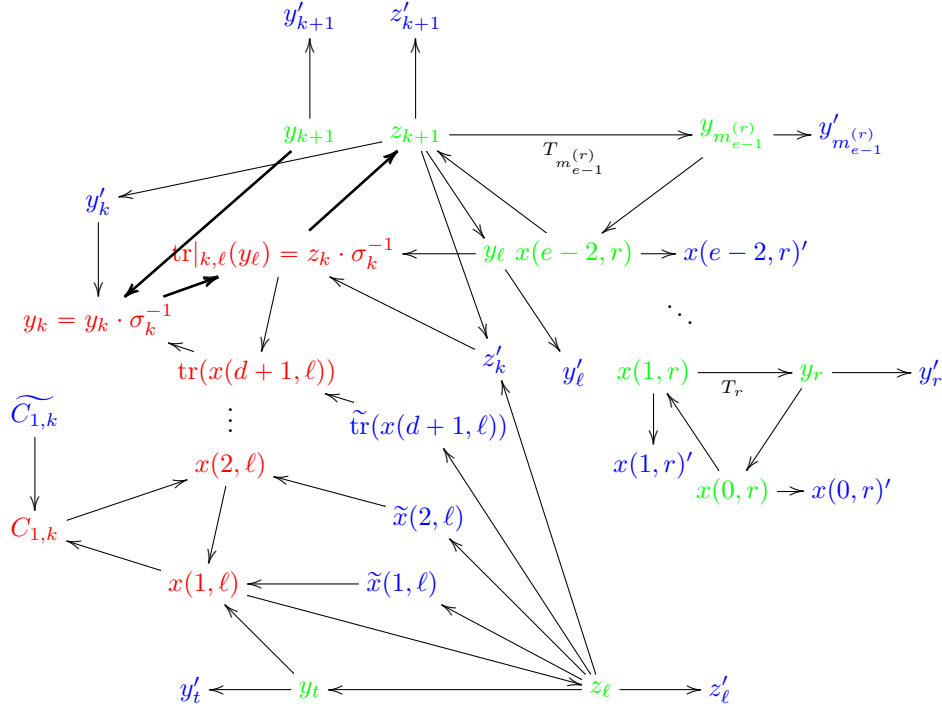

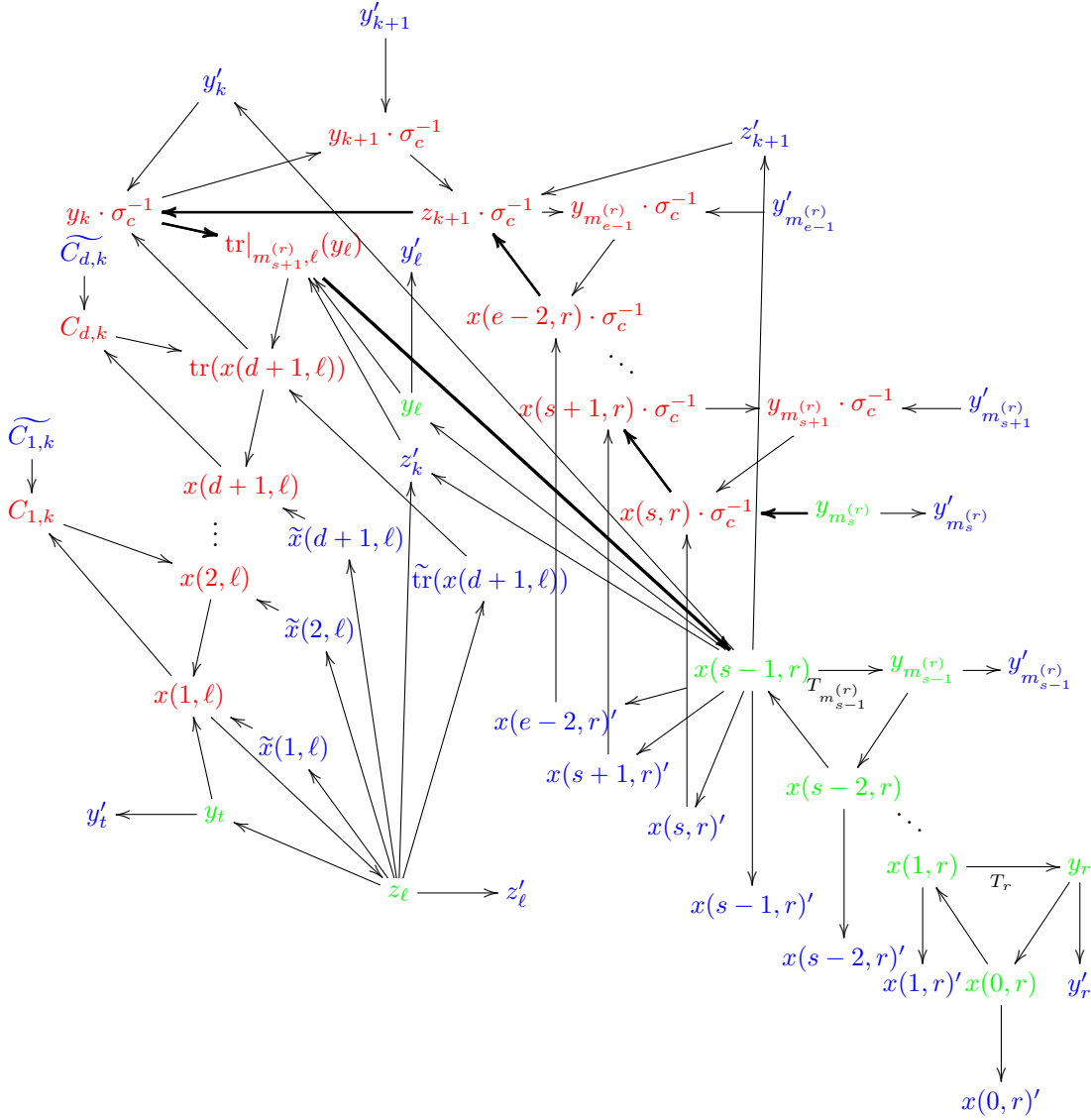
\begin{figure}
$$\begin{xy} 0;<1pt,0pt>:<0pt,-1pt>:: 
(-10,190) *+{\textcolor{red}{C_{1,k}}} ="0",
(60,215) *+{\textcolor{red}{x(2,\ell)}} ="1",
(60,195) *+{{\vdots}} ="47",
(20,75) *+{\textcolor{red}{y_k\cdot\sigma_c^{-1}}} ="46",
(70,180) *+{\textcolor{red}{x(d+1,\ell)}} ="60",
(10,120) *+{\textcolor{red}{C_{d,k}}} ="61",
(10,90) *+{\textcolor{blue}{\widetilde{C_{d,k}}}} ="62",
(110,200) *+{\textcolor{blue}{\widetilde{x}(d+1,\ell)}} ="63",
(60,25) *+{\textcolor{blue}{y_k^\prime}} ="48",
(210,150) *+{\textcolor{red}{x(s+1,r)\cdot\sigma_c^{-1}}} ="2",
(125,45) *+{\textcolor{red}{y_{k+1}\cdot\sigma_{c}^{-1}}} ="3",
(50,260) *+{\textcolor{red}{x(1,\ell)}} ="4",
(60,305) *+{\textcolor{green}{y_t}} ="50",
(15,305) *+{\textcolor{blue}{y_t^\prime}} ="51",
(190,115) *+{\textcolor{red}{x(e-2,r)\cdot\sigma_c^{-1}}} ="5",
(240,190) *+{\textcolor{red}{x(s,r)\cdot\sigma_c^{-1}}} ="6",
(220,75) *+{\textcolor{red}{y_{m^{(r)}_{e-1}}\cdot\sigma_{c}^{-1}}} ="7",
(160,75) *+{\textcolor{red}{z_{k+1}\cdot\sigma_{c}^{-1}}} ="8",
(215,130) *+{\ddots} ="9",
(90,90) *+{\textcolor{red}{\text{tr}|_{m^{(r)}_{s+1},\ell}(y_\ell)}} ="10",
(295,150) *+{\textcolor{red}{y_{m^{(r)}_{s+1}}\cdot\sigma_c^{-1}}} ="11",
(80,135) *+{\textcolor{red}{\text{tr}(x(d+1,\ell))}} ="12",
(135,150) *+{\textcolor{green}{y_\ell}} ="13",
(130,335) *+{\textcolor{green}{z_{\ell}}} ="14",
(-10,160) *+{\textcolor{blue}{\widetilde{C_{1,k}}}} ="15",
(100,235) *+{\textcolor{blue}{\widetilde{x}(2,\ell)}} ="16",
(90,280) *+{\textcolor{blue}{\widetilde{x}(1,\ell)}} ="17",
(135,170) *+{\textcolor{blue}{z_{k}^\prime}} ="18",
(165,215) *+{\textcolor{blue}{\widetilde{\text{tr}}(x(d+1,\ell))}} ="19",
(125,0) *+{\textcolor{blue}{y_{k+1}^\prime}} ="20",
(270,45) *+{\textcolor{blue}{z_{k+1}^\prime}} ="21",
(285,75) *+{\textcolor{blue}{y_{m^{(r)}_{e-1}}^\prime}} ="22",
(190,270) *+{\textcolor{blue}{x(e-2,r)^\prime}} ="23",
(210,290) *+{\textcolor{blue}{x(s+1,r)^\prime}} ="25", 
(360,150) *+{\textcolor{blue}{y_{m^{(r)}_{s+1}}^\prime}} ="26",
(240,310) *+{\textcolor{blue}{x(s,r)^\prime}} ="27",
(135,90) *+{\textcolor{blue}{y_\ell^\prime}} ="28",
(175,335) *+{\textcolor{blue}{z_\ell^\prime}} ="29",
(265,250) *+{\textcolor{green}{x(s-1,r)}} ="30",
(300,295) *+{\textcolor{green}{x(s-2,r)}} ="32",
(330,250) *+{\textcolor{green}{y_{m^{(r)}_{s-1}}}} ="31",
(330,325) *+{\textcolor{green}{x(1,r)}} ="33",
(390,325) *+{\textcolor{green}{y_r}} ="34",
(360,370) *+{\textcolor{green}{x(0,r)}} ="35",
(300,190) *+{\textcolor{green}{y_{m^{(r)}_s}}} ="36",
(375,250) *+{\textcolor{blue}{y_{m^{(r)}_{s-1}}^\prime}} ="38",
(345,190) *+{\textcolor{blue}{y_{m^{(r)}_s}^\prime}} ="37",
(300,360) *+{\textcolor{blue}{x(s-2,r)^\prime}} ="39",
(325,305) *+{\ddots} ="40",
(265,340) *+{\textcolor{blue}{x(s-1,r)^\prime}} ="41",
(330,370) *+{\textcolor{blue}{x(1,r)^\prime}} ="42",
(390,370) *+{\textcolor{blue}{y_r^\prime}} ="43",
(360,415) *+{\textcolor{blue}{x(0,r)^\prime}} ="44",
"0", {\ar"1"},
"4", {\ar"0"},
"15", {\ar"0"},
"1", {\ar"4"},
"10", {\ar@<0.1mm>"30"},
"10", {\ar@<0.05mm>"30"},
"10", {\ar@<-0.05mm>"30"},
"10", {\ar@<-0.1mm>"30"},
"10", {\ar@<-0.15mm>"30"},
"10", {\ar@<0.15mm>"30"},
"10", {\ar"30"},
"16", {\ar"1"},
"6", {\ar@<0.1mm>"2"},
"6", {\ar@<0.05mm>"2"},
"6", {\ar@<-0.05mm>"2"},
"6", {\ar@<-0.1mm>"2"},
"6", {\ar"2"},
"2", {\ar"11"},
"25", {\ar"2"},
"3", {\ar"8"},
"20", {\ar"3"},
"4", {\ar"14"},
"17", {\ar"4"},
"7", {\ar"5"},
"5", {\ar@<0.1mm>"8"},
"5", {\ar@<0.05mm>"8"},
"5", {\ar@<-0.05mm>"8"},
"5", {\ar@<-0.1mm>"8"},
"5", {\ar"8"},
"23", {\ar"5"},
"11", {\ar"6"},
"27", {\ar"6"},
"8", {\ar"7"},
"8", {\ar@<0.1mm>"46"},
"8", {\ar@<0.05mm>"46"},
"8", {\ar@<-0.05mm>"46"},
"8", {\ar@<-0.1mm>"46"},
"8", {\ar"46"},
"22", {\ar"7"},
"21", {\ar"8"},
"10", {\ar"12"},
"18", {\ar"10"},
"26", {\ar"11"},
"12", {\ar"46"},
"12", {\ar"60"},
"13", {\ar"10"},
"19", {\ar"12"},
"13", {\ar"28"},
"14", {\ar"16"},
"14", {\ar"17"},
"14", {\ar"18"},
"14", {\ar"19"},
"14", {\ar"29"},
"14", {\ar"50"},
"14", {\ar"63"},
"30", {\ar_{T_{m^{(r)}_{s-1}}}"31"},
"30", {\ar"13"},
"30", {\ar"18"},
"30", {\ar"21"},
"30", {\ar"23"},
"30", {\ar"25"},
"30", {\ar"27"},
"30", {\ar"41"},
"30", {\ar"48"},
"31", {\ar"32"},
"31", {\ar"38"},
"32", {\ar"30"},
"32", {\ar"39"},
"33", {\ar_{T_r}"34"},
"33", {\ar"42"},
"34", {\ar"35"},
"34", {\ar"43"},
"35", {\ar"33"},
"35", {\ar"44"},
"36", {\ar@<0.1mm>"6"},
"36", {\ar@<0.05mm>"6"},
"36", {\ar@<-0.05mm>"6"},
"36", {\ar@<-0.1mm>"6"},
"36", {\ar"6"},
"36", {\ar"37"},
"46", {\ar@<0.1mm>"10"},
"46", {\ar@<0.05mm>"10"},
"46", {\ar@<-0.05mm>"10"},
"46", {\ar@<-0.1mm>"10"},
"46", {\ar"10"},
"46", {\ar"3"},
"48", {\ar"46"},
"50", {\ar"51"},
"50", {\ar"4"},
"60", {\ar"61"},
"62", {\ar"61"},
"63", {\ar"60"},
"61", {\ar"12"},
\end{xy}$$
\caption{The effect of applying $\underline{\mu}_c\circ \cdots \circ \underline{\mu}_{k+1}$ to $\underline{\mu}_{k}\circ \cdots \circ \underline{\mu}_1\circ \underline{\mu}_0(\widehat{Q})$ where $m^{(r)}_{s+1}\le c < m^{(r)}_s$. If $c = m^{(r)}_s-1$, the mutation sequence $\underline{\mu}_{m^{(r)}_s}$ is indicated by the bold arrows. Note that, as in the statement of Lemma~\ref{TRsigma_c}, we have that $\text{tr}|_{m^{(r)}_{s+1},\ell}(y_\ell) = z_k\cdot \sigma_c^{-1}$.} 
\label{scrshot6}
\end{figure}

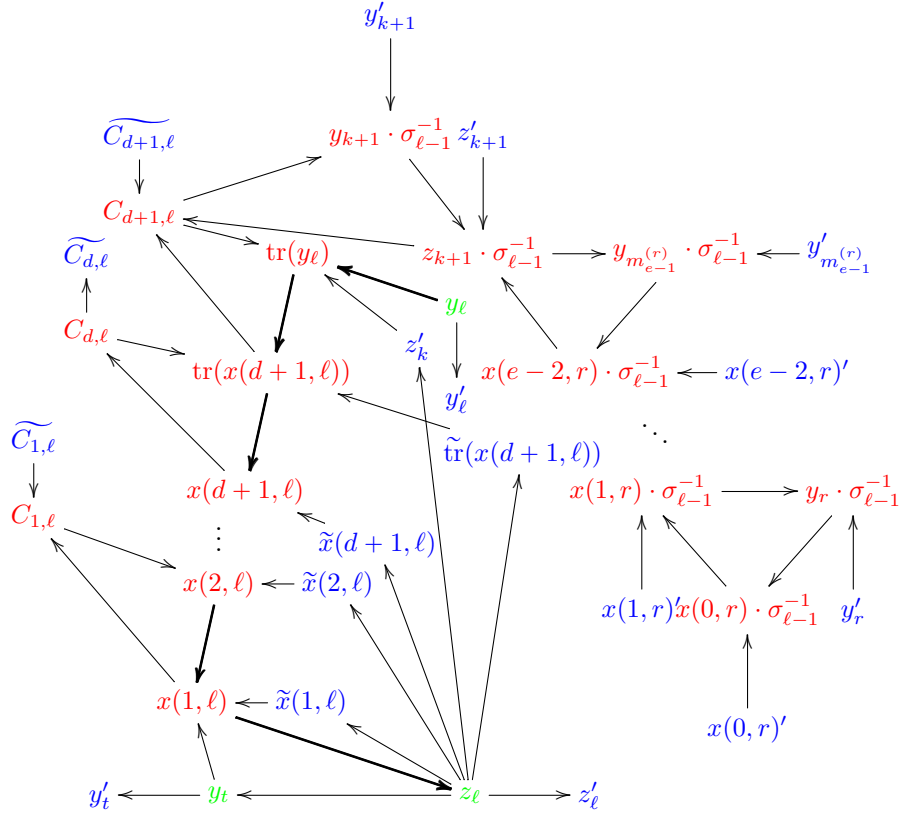
\begin{figure}
$$\begin{xy} 0;<1pt,0pt>:<0pt,-1pt>:: 
(120,200) *+{\textcolor{blue}{\widetilde{x}(d+1,\ell)}} ="63",
(10,90) *+{\textcolor{blue}{\widetilde{C_{d,\ell}}}} ="62",
(10,120) *+{\textcolor{red}{C_{d,\ell}}} ="61",
(70,180) *+{\textcolor{red}{x(d+1,\ell)}} ="60",
(-10,190) *+{\textcolor{red}{C_{1,\ell}}} ="0",
(60,215) *+{\textcolor{red}{x(2,\ell)}} ="1",
(60,195) *+{{\vdots}} ="47",
(30,75) *+{\textcolor{red}{C_{d+1,\ell}}} ="40",
(30,45) *+{\textcolor{blue}{\widetilde{C_{d+1,\ell}}}} ="48", 
(220,180) *+{\textcolor{red}{x(1,r)\cdot\sigma_{\ell-1}^{-1}}} ="2",
(125,45) *+{\textcolor{red}{y_{k+1}\cdot\sigma_{\ell-1}^{-1}}} ="3",
(50,260) *+{\textcolor{red}{x(1,\ell)}} ="4",
(195,135) *+{\textcolor{red}{x(e-2,r)\cdot\sigma_{\ell-1}^{-1}}} ="5",
(260,225) *+{\textcolor{red}{x(0,r)\cdot\sigma_{\ell-1}^{-1}}} ="6",
(235,90) *+{\textcolor{red}{y_{m^{(r)}_{e-1}}\cdot\sigma_{\ell-1}^{-1}}} ="7",
(160,90) *+{\textcolor{red}{z_{k+1}\cdot\sigma_{\ell-1}^{-1}}} ="8",
(225,155) *+{\ddots} ="9",
(90,90) *+{\textcolor{red}{\text{tr}(y_\ell)}} ="10",
(300,180) *+{\textcolor{red}{y_r\cdot\sigma_{\ell-1}^{-1}}} ="11",
(80,135) *+{\textcolor{red}{\text{tr}(x(d+1,\ell))}} ="12",
(150,110) *+{\textcolor{green}{y_\ell}} ="13",
(155,295) *+{\textcolor{green}{z_{\ell}}} ="14",
(60,295) *+{\textcolor{green}{y_{t}}} ="49",
(15,295) *+{\textcolor{blue}{y_{t}^\prime}} ="50",
(-10,160) *+{\textcolor{blue}{\widetilde{C_{1,\ell}}}} ="15",
(105,215) *+{\textcolor{blue}{\widetilde{x}(2,\ell)}} ="16",
(95,260) *+{\textcolor{blue}{\widetilde{x}(1,\ell)}} ="17",
(135,125) *+{\textcolor{blue}{z_{k}^\prime}} ="18",
(175,165) *+{\textcolor{blue}{\widetilde{\text{tr}}(x(d+1,\ell))}} ="19",
(125,0) *+{\textcolor{blue}{y_{k+1}^\prime}} ="20",
(160,45) *+{\textcolor{blue}{z_{k+1}^\prime}} ="21",
(295,90) *+{\textcolor{blue}{y_{m^{(r)}_{e-1}}^\prime}} ="22",
(275,135) *+{\textcolor{blue}{x(e-2,r)^\prime}} ="23",
(220,225) *+{\textcolor{blue}{x(1,r)^\prime}} ="25",
(300,225) *+{\textcolor{blue}{y_r^\prime}} ="26",
(260,270) *+{\textcolor{blue}{x(0,r)^\prime}} ="27",
(150,145) *+{\textcolor{blue}{y_\ell^\prime}} ="28",
(200,295) *+{\textcolor{blue}{z_\ell^\prime}} ="29",
"0", {\ar"1"},
"4", {\ar"0"},
"15", {\ar"0"},
"1", {\ar"4"},
"1", {\ar@<0.1mm>"4"},
"1", {\ar@<0.05mm>"4"},
"1", {\ar@<-0.05mm>"4"},
"1", {\ar@<-0.1mm>"4"},
"12", {\ar"40"},
"12", {\ar"60"},
"12", {\ar@<0.1mm>"60"},
"12", {\ar@<0.05mm>"60"},
"12", {\ar@<-0.05mm>"60"},
"12", {\ar@<-0.1mm>"60"},
"16", {\ar"1"},
"6", {\ar"2"},
"2", {\ar"11"},
"25", {\ar"2"},
"3", {\ar"8"},
"20", {\ar"3"},
"4", {\ar@<0.1mm>"14"},
"4", {\ar@<0.05mm>"14"},
"4", {\ar@<-0.05mm>"14"},
"4", {\ar@<-0.1mm>"14"},
"4", {\ar"14"},
"17", {\ar"4"},
"7", {\ar"5"},
"5", {\ar"8"},
"23", {\ar"5"},
"11", {\ar"6"},
"27", {\ar"6"},
"8", {\ar"7"},
"22", {\ar"7"},
"8", {\ar"40"},
"21", {\ar"8"},
"10", {\ar@<0.1mm>"12"},
"10", {\ar@<0.05mm>"12"},
"10", {\ar@<-0.05mm>"12"},
"10", {\ar@<-0.1mm>"12"},
"10", {\ar"12"},
"18", {\ar"10"},
"26", {\ar"11"},
"13", {\ar@<0.1mm>"10"},
"13", {\ar@<0.05mm>"10"},
"13", {\ar@<-0.05mm>"10"},
"13", {\ar@<-0.1mm>"10"},
"13", {\ar"10"},
"19", {\ar"12"},
"13", {\ar"28"},
"14", {\ar"16"},
"14", {\ar"17"},
"14", {\ar"18"},
"14", {\ar"19"},
"14", {\ar"29"},
"14", {\ar"49"},
"14", {\ar"63"},
"40", {\ar"3"},
"63", {\ar"60"},
"48", {\ar"40"},
"40", {\ar"10"},
"49", {\ar"50"},
"49", {\ar"4"},
"60", {\ar"61"},
"61", {\ar"12"},
"61", {\ar"62"},
\end{xy}$$
\caption{The effect of applying $\underline{\mu}_{\ell-1}\circ \cdots \circ \underline{\mu}_{k+1}$ to $\underline{\mu}_{k}\circ \cdots \circ \underline{\mu}_1\circ \underline{\mu}_0(\widehat{Q})$ where $C_{d+1,\ell} = y_k\cdot\sigma_{\ell-1}^{-1}$ and $\widetilde{C_{d+1,\ell}} = y_k^\prime$, as desired. The mutation sequence $\underline{\mu}_\ell$ is indicated by the bold arrows.}
\label{scrshot8}
\end{figure}

\section{Additional Questions and Remarks} \label{Sec:Conc}
In this section, we give an example to show how our results provide explicit maximal green sequences for quivers that are not of type $\mathbb{A}$. We also discuss ideas we have for further research.

\subsection{Maximal Green Sequences for {Quivers} Arising from Surface Triangulations} \label{Sec:OtherSurf}

The following example shows how our formulas for maximal green sequences for type $\mathbb{A}$ quivers can be used to give explicit formulas for maximal green sequences for {quivers arising from other types of triangulated surfaces}.

\begin{example}\label{apptosurf}
Consider the marked surface $(\textbf{S}, \textbf{M})$ with the triangulation given as $\textbf{T}$ shown in Figure~\ref{pantstriang} on the left. The surface \textbf{S} is a once-punctured pair of pants with triangulation $$\textbf{T} = \textbf{T}_1\sqcup \textbf{T}_2 \sqcup \{\eta, \epsilon, \zeta\}$$ where $\alpha_1,\alpha_2, \alpha_3 \in \textbf{T}_1$ and $\beta_1, \beta_2, \beta_3, \nu \in \textbf{T}_2$. We assume that the boundary arcs $b_i$ with $i \in [5]$ contain no marked points except for those shown in Figure~\ref{pantstriang}. The other boundary arcs may contain any number of marked points. As in Section~\ref{Sec:SurfTriang}, let $Q_\textbf{T}$ be the quiver determined by $\textbf{T}$ and let $v_\delta \in (Q)_0$ denote the vertex corresponding to arc $\delta \in \textbf{T}.$

\begin{figure}[h]
$$\includegraphics[scale=1.3]{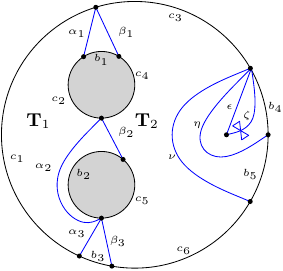} \ \ \ \includegraphics[scale=1.3]{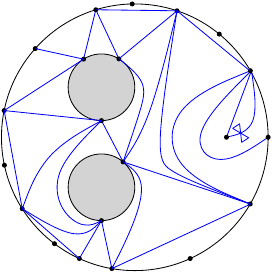}$$
\caption{}
\label{pantstriang}
\end{figure}

We can think of the marked surface $(\textbf{S}_1,\textbf{M}_1)$ determined by $c_1, \beta_1, b_1, c_2, \beta_2, b_2, \beta_3, b_3$ as an $m_1$-gon where $m_1 = \#\textbf{M}_1$ and we can think of $\textbf{T}_1$ as a triangulation of $\textbf{S}_1.$ Similarly, we can think of the marked surface $(\textbf{S}_2,\textbf{M}_2)$ determined by $\alpha_1, c_3, \eta, b_5, c_6, \alpha_3, c_5, b_2, \alpha_2, c_4, b_1$ as an $m_2$-gon where $m_2 = \#\textbf{M}_2$ and we can think of $\textbf{T}_2$ as a triangulation of $\textbf{S}_2$. Thus the quiver $Q_{\textbf{T}_i}$, determined by $\textbf{T}_i$, is a type $\mathbb{A}$ quiver for $i = 1,2.$ Furthermore, we have $$Q_{\textbf{T}} = Q_{\textbf{T}_1}\oplus_{(v_{\alpha_1},v_{\alpha_2},v_{\alpha_3})}^{(v_{\beta_1},v_{\beta_2}, v_{\beta_3})}Q_{\textbf{T}_2}\oplus_{(v_\nu)}^{(v_\eta)}R$$
where
$$\begin{array}{rclll}
\raisebox{-.41in}{$R$} & \raisebox{-.41in}{=} & \begin{xy} 0;<1pt,0pt>:<0pt,-1pt>:: 
(0,30) *+{v_\eta} ="0",
(45,60) *+{v_\epsilon.} ="1",
(45,0) *+{v_{\zeta}} ="2",
"0", {\ar"1"},
"0", {\ar"2"},
\end{xy}
\end{array}$$
By Corollary~\ref{anytypeA}, $Q_{\textbf{T}_1}$ and $Q_{\textbf{T}_2}$ each have a maximal green sequence $\underline{\mu}^{Q_{\textbf{T}_i}}$ for $i = 1,2.$ Since $R$ is acyclic, we can define $\underline{\mu}^{R}$ to be any mutation sequence of $\widehat{R}$ where each mutation occurs at a source (for instance, put $\underline{\mu}^{R} = \mu_{v_{\epsilon}}\circ\mu_{v_{\zeta}}\circ\mu_{v_{\eta}}$). Then $\underline{\mu}^R$ is clearly a maximal green sequence of $R$. Now Theorem~\ref{tcolordirsum} implies that $\underline{\mu}^R\circ\underline{\mu}^{Q_{\textbf{T}_2}}\circ\underline{\mu}^{Q_{\textbf{T}_1}}$ is a maximal green sequence of $Q_{\textbf{T}}$. 

Suppose that $\textbf{T}_1$ and $\textbf{T}_2$ are given by the triangulations shown in Figure~\ref{pantstriang} on the right. Then we have that $Q_{\textbf{T}_1}$ and $Q_{\textbf{T}_2}$ are the quivers shown in Figure~\ref{QT1QT2} where we think of the irreducible parts of $Q_{\textbf{T}_1}$ and $Q_{\textbf{T}_2}$ as signed irreducible type $\mathbb{A}$ quivers with respect to the root 3-cycles $S^{(1)}_1$ and $S^{(2)}_1$, respectively. In this situation, $Q_{\textbf{T}_1}$ and $Q_{\textbf{T}_2}$ have the maximal green sequences 
 \begin{eqnarray}
 \underline{\mu}^{Q_{\textbf{T}_1}} & = & \mu_{w_1}\circ \mu_{w_2}\circ\mu_{v_{\alpha_1}}\circ\underline{\mu}^{(1)}_3\circ\underline{\mu}^{(1)}_2\circ\underline{\mu}^{(1)}_1\circ\underline{\mu}^{(1)}_0 \nonumber \\
\underline{\mu}^{Q_{\textbf{T}_2}} & = & \mu_{w_4}\circ \mu_{w_3}\circ\underline{\mu}^{(2)}_5\circ\underline{\mu}^{(2)}_4\circ\underline{\mu}^{(2)}_3\circ\underline{\mu}^{(2)}_2\circ\underline{\mu}^{(2)}_1\circ\underline{\mu}^{(2)}_0\nonumber\end{eqnarray}
respectively where 
$$\begin{array}{rllll}
\begin{array}{rllll}
\underline{\mu}^{(1)}_0 & = & \mu_{x^{(1)}_1} \nonumber \\
\underline{\mu}^{(1)}_1 & = & \mu_{x^{(1)}_1}\circ \mu_{z^{(1)}_1}\circ \mu_{y^{(1)}_1}\nonumber \\
\underline{\mu}^{(1)}_2 & = & \mu_{x^{(1)}_1}\circ \mu_{z^{(1)}_1}\circ \mu_{z^{(1)}_2}\circ \mu_{y^{(1)}_2}\nonumber \\
\underline{\mu}^{(1)}_3 & = & \mu_{y^{(1)}_2}\circ \mu_{x^{(1)}_1}\circ \mu_{z^{(1)}_3}\circ \mu_{y^{(1)}_3}\nonumber
\end{array}
& \text{ } &  \begin{array}{rllll}
 \underline{\mu}^{(2)}_0 & = & \mu_{x^{(2)}_1} \nonumber \\
 \underline{\mu}^{(2)}_1 & = & \mu_{x^{(2)}_1}\circ \mu_{z^{(2)}_1}\circ \mu_{y^{(2)}_1}\nonumber \\
 \underline{\mu}^{(2)}_2 & = & \mu_{x^{(2)}_1}\circ \mu_{z^{(2)}_1}\circ \mu_{z^{(2)}_2}\circ \mu_{y^{(2)}_2}\nonumber \\
 \underline{\mu}^{(2)}_3 & = & \mu_{y^{(2)}_2}\circ \mu_{x^{(2)}_1}\circ \mu_{z^{(2)}_3}\circ \mu_{y^{(2)}_3}\nonumber \\
 \underline{\mu}^{(2)}_4 & = & \mu_{y^{(2)}_3}\circ \mu_{y^{(2)}_2}\circ\mu_{z^{(2)}_4}\circ\mu_{y^{(2)}_4} \nonumber \\
 \underline{\mu}^{(2)}_5 & = & \mu_{z^{(2)}_4}\circ\mu_{x^{(2)}_1}\circ\mu_{z^{(2)}_3}\circ\mu_{z^{(2)}_5}\circ\mu_{y^{(2)}_5}. \nonumber
 \end{array} 
\end{array}$$
and $\underline{\mu}^R\circ \underline{\mu}^{Q_{\textbf{T}_2}}\circ\underline{\mu}^{Q_{\textbf{T}_1}}$ is a maximal green sequence of $Q_{\textbf{T}}$. In general, if we have a quiver $Q_{\textbf{T}}$ that can be realized as a direct sum of type $\mathbb{A}$ quivers and acyclic quivers, we can write an explicit formula for a maximal green sequence of $Q_{\textbf{T}}$.

\begin{figure}[h]
$$\begin{array}{rclll}
\raisebox{-.41in}{$Q_{\textbf{T}_1}$} & \raisebox{-.41in}{=} &\begin{xy} 0;<1pt,0pt>:<0pt,-1pt>:: 
(30,0) *+{y^{(1)}_1} ="0",
(0,30) *+{x^{(1)}_1} ="1",
(60,30) *+{z^{(1)}_1} ="2",
(120,30) *+{y^{(1)}_2} ="3",
(90,60) *+{z^{(1)}_2} ="4",
(150,0) *+{y^{(1)}_3} ="5",
(180,30) *+{z^{(1)}_3} ="6",
(240,30) *+{w_1} ="7",
(300,30) *+{w_2} ="8",
(360,30) *+{v_{\alpha_1}} ="9",
"1", {\ar"0"},
"0", {\ar"2"},
"2", {\ar_{T^{(1)}_1}"1"},
"2", {\ar_{T^{(1)}_2}"3"},
"4", {\ar"2"},
"3", {\ar"4"},
"3", {\ar"5"},
"6", {\ar_{T^{(1)}_3}"3"},
"5", {\ar"6"},
"6", {\ar"7"},
"8", {\ar"7"},
"9", {\ar"8"},
\end{xy}\end{array}$$
$$\begin{array}{rclll}
\raisebox{-.84in}{$Q_{\textbf{T}_2}$} & \raisebox{-.84in}{=} &\begin{xy} 0;<1pt,0pt>:<0pt,-1pt>:: 
(30,30) *+{y^{(2)}_1} ="0",
(0,60) *+{x^{(2)}_1} ="1",
(60,60) *+{z^{(2)}_1} ="2",
(120,60) *+{y^{(2)}_2} ="3",
(90,90) *+{z^{(2)}_2} ="4",
(150,30) *+{y^{(2)}_3} ="5",
(180,60) *+{z^{(2)}_3} ="6",
(240,60) *+{y^{(2)}_5} ="7",
(150,90) *+{w_3} ="8",
(210,90) *+{z^{(2)}_5} ="9",
(270,90) *+{w_4} ="10",
(180,0) *+{y^{(2)}_4} ="11",
(210,30) *+{z^{(2)}_4} ="12",
"1", {\ar"0"},
"0", {\ar"2"},
"2", {\ar_{T^{(2)}_1}"1"},
"2", {\ar_{T^{(2)}_2}"3"},
"4", {\ar"2"},
"3", {\ar"4"},
"3", {\ar"5"},
"6", {\ar_{T^{(2)}_3}"3"},
"4", {\ar"8"},
"5", {\ar"6"},
"5", {\ar"11"},
"12", {\ar_{T^{(2)}_4}"5"},
"6", {\ar_{T^{(2)}_5}"7"},
"9", {\ar"6"},
"7", {\ar"9"},
"9", {\ar"10"},
"11", {\ar"12"},
\end{xy}\end{array}$$
\caption{}
\label{QT1QT2}
\end{figure}

\end{example}

\begin{problem}\label{clusteralgfromsurfaces}
Find explicit formulas for maximal green sequences for quivers arising from triangulations of surfaces.
\end{problem}

Using Corollary~\ref{cor:tsurf}, we can reduce Problem~\ref{clusteralgfromsurfaces} to the problem of finding explicit formulas for maximal green sequences of irreducible quivers that arise from a triangulated surface. In \cite{ACCERV}, the authors sketch an argument showing the existence of maximal green sequences for quivers arising from triangulated surfaces. However, we would like to prove the existence of maximal green sequences by giving explicit formulas for maximal green sequences of such quivers.

Some progress has already been made in answering Problem~\ref{clusteralgfromsurfaces}. In \cite{L}, Ladkani shows that quivers arising from triangulations of once-punctured closed surfaces of genus $g \ge 1$ have no maximal green sequences. In \cite{Bucher, BucherMills}, explicit formulas for maximal green sequences are given for specific triangulations of closed genus $g \ge 1$ surfaces. In \cite{CDRSW}, a formula is given for the minimal length maximal green sequences of quivers defined by polygon triangulations. It would be interesting to understand, in general, what are the possible lengths that can be achieved by maximal green sequences of a given quiver.

\subsection{Trees of Cycles}

Our study of signed irreducible type $\mathbb{A}$ quivers was made possible by the fact that such quivers are equivalent to labeled binary trees of $3$-cycles (see Lemma~\ref{binarytreelemma}). It is therefore reasonable to ask if one can find explicit formulas for maximal green sequences of quivers that are trees of cycles where each cycle has length at least $k \ge 3$. In our construction, we define a total ordering and a sign function on the set of 3-cycles of an irreducible type $\mathbb{A}$ quiver (with at least one 3-cycle), and this data was important in discovering and describing the associated mutation sequence. One could use a similar technique to construct an analogue of the associated mutation sequence for quivers that are trees of oriented cycles.

\begin{problem}
Find a construction of maximal green sequences for quivers that are trees of oriented cycles.
\end{problem}

\subsection{Enumeration of Maximal Green Sequences}
\label{Sec:Enum} 

In the process of revising this paper, the problems posed in this section have been solved. For posterity, we keep this section as it appeared in the original {\tt arXiv} version. A solution to Problem~\ref{polygonal} appears in \cite{garver2015lattice}. A solution to Problem~\ref{minlengthmgs} appears in \cite{CDRSW}.

\vspace{1em}

For a given signed irreducible type $\mathbb{A}$ quiver $\mathcal{Q}$ with root 3-cycle $T$ and with at least two 3-cycles our construction produces a maximal green sequence $\underline{\mu} = \underline{\mu}(T)$ of $\mathcal{Q}$ for each leaf 3-cycle in $\mathcal{Q}$. It would be interesting to see how many maximal green sequences of ${Q}$ can be obtained from the maximal green sequences $\underline{\mu}$ as the choice of the root 3-cycle $T$ varies.

\begin{problem}\label{polygonal}
Determine what maximal green sequences of $\mathcal{Q}$ can be obtained via commutation relations and Pentagon Identity relations applied to the maximal green sequences in $\{\underline{\mu}(T): T \ \text{is a leaf 3-cycle}\}.$
\end{problem}

Additionally, in \cite{BDP} there are several tables giving the number of maximal green sequences of certain {small rank} quivers by length. These computations may be useful for making progress on the problem of enumerating maximal green sequences of quivers.

As discussed in Remark~\ref{mubarlength}, the associated mutation sequences constructed here are not necessarily the shortest possible maximal green sequences. This motivates the following problem.

\begin{problem}\label{minlengthmgs}
Provide a construction of the maximal green sequences of minimal length, {possibly} by showing how to apply Pentagon Identity relations to the associated mutation sequences.
\end{problem}

\subsection{Further Study of Maximal Green Sequences} \label{Sec:A3}

Note that maximal green sequences of a quiver $Q$ can be thought of as maximal chains (from the unique source to the unique sink) in the \emph{oriented exchange graph} \cite[Section 2]{BDP}.  In the case that $Q$ is of type $\mathbb{A}$, the exchange graph is an orientation of the $1$-skeleton of the associahedron.  The oriented exchange graph is especially nice in the case when $Q$ is a \textbf{Dynkin quiver} (i.e. an acyclic orientation of a Dynkin diagram of type $\mathbb{A},$ $\mathbb{D}$, or $\mathbb{E}$).  For example, it is the Hasse graph of the Tamari lattice in the case $Q$ is linear and equioriented and it is the Hasse graph of a Cambrian lattice (in the sense of Reading \cite{Read}) as is noted in \cite[Section 3]{K3}.  In particular, this means that we consider the finite Coxeter group $G$ whose Dynkin Diagram is the unoriented version of $Q$ and a choice of Coxeter element $c$ compatible with the orientation of $Q$ and then maximal green sequences are in bijection with maximal chains in the Cambrian lattice, a quotient of the weak Bruhat order on $G$.  Note, that this bijection is studied further in \cite{Qiu} where each $c$-sortable word is shown to correspond to a green sequence.

To indicate the difficulty of describing the set of maximal green sequences once we consider quivers with cycles, we focus on the $A_3$ case here.  In the case where $Q$ is a $3$-cycle (with $1 \to 2$, $2 \to 3$, $3 \to 1$), there is not a corresponding Cambrian congruence that one can apply to the weak Bruhat order on the symmetric group $G = S_4$ to obtain the desired Hasse diagram.  In particular, the corresponding Cambrian lattice is constructed from the geometry of the affine $\tilde{A}_2$ root system instead of from a finite Coxeter group.  Intersecting this coarsening of the Coxeter Lattice with the Tits Cone yields 11 regions rather than the 14 we obtain in the acyclic case \cite{ReadSpey}.

Nonetheless, we can still compute maximal green sequences in this case, and see that they are indeed the set of oriented paths through a certain orientation of the $1$-skeleton of the associahedron.  There are six possible maximal green sequences of length $4$: 
$\mu_1\circ\mu_3\circ\mu_2\circ\mu_1$,~$\mu_2\circ\mu_1\circ\mu_3\circ\mu_2$, $\mu_3\circ\mu_2\circ\mu_1\circ\mu_3$,~
$\mu_3\circ\mu_1\circ\mu_2\circ\mu_1$,~ $\mu_1\circ\mu_2\circ\mu_3\circ\mu_2$,~ and $\mu_3\circ\mu_2\circ\mu_1\circ \mu_3$. We can find three more maximal green sequences of length $5$: $\mu_2\circ\mu_3\circ\mu_2\circ\mu_1\circ\mu_2$,~ $\mu_3\circ\mu_1\circ\mu_3\circ\mu_2\circ\mu_3$,~ and
$\mu_1\circ\mu_2\circ\mu_1\circ\mu_3\circ\mu_1$.  As in Figure 22 of \cite{BDP}, there are no other maximal green sequences of this quiver. To obtain these sequences of length 5, we select any of the first three maximal green sequences of length 4. We then apply the relation $\mu_{i+1} \circ \mu_{i} \sim \mu_{i+1} \circ \mu_i \circ \mu_{i+1}$ (where the arithmetic is carried out mod 3) and apply the vertex permutation $(i,i+1)$ to the vertices at which one mutates later in the sequence.

In an attempt to understand this example in terms of the Coxeter group of type $A_3$, i.e. $S_4$, we consider the presentation described in \cite{BarotMarsh} for quivers with cycles.  In this case, if we let $s_1 = (14)$, $s_2 = (24)$, $s_3 = (34)$, we obtain 
$$S_4 = \langle s_1, s_2, s_3: s_1^2=s_2^2=s_3^2 = (s_1s_2)^3 = (s_2s_3)^3 = (s_3s_1)^3 = (s_1s_2s_3s_2)^2=1\rangle.$$
Unlike the acyclic $A_3$ case where the permutation in $S_4$ corresponding to the longest word, i.e. $4321$, is the only element of $S_4$ whose length as a reduced expression, e.g. $s_1s_2s_3s_1s_2s_3$, is of length $6$, in the Barot-Marsh presentation, the permutations 
$4321$, $3412$, $2143$, $1342$, and $1432$ all have reduced expressions of maximal length, namely $4$.  

Further, if we visualize the order complex of $S_4$ under this presentation, we obtain a torus  (see Example 3.1 of \cite{BabR}) rather than a simply-connected surface like the acyclic case and there are no permutations with reduced expression of length $5$, hence reduced expressions in this presentation cannot correspond to maximal green sequences.  We thank Vic Reiner for bringing his paper with Eric Babson to our attention.  Hence, understanding the full collection of maximal green sequences for other quivers with cycles, even those of type $\mathbb{A}$, appears to require more than an understanding of the associated Coxeter groups.

\vspace{1em}

Since our original preprint, in recent work \cite{garver2015lattice} (resp. \cite{garver2015orientedflip}) of the first author and Thomas McConville, an analgoue of the weak order on $S_n$, known as the lattice of \textbf{biclosed subcatgories} (resp. \textbf{biclosed sets} of \textbf{segments}) is constructed. In \cite{garver2015lattice} and \cite{garver2015orientedflip}, it is shown that the oriented exchange graph of quivers of type $\mathbb{A}$ can be obtained as a lattice quotient of these. This lattice congruence generalizes the Cambrian congruence in type $\mathbb{A}$. Additionally, the results in \cite{garver2015lattice} and \cite{garver2015orientedflip} do require techniques from representation theory of finite dimensional algebras, which further suggests that understanding the full collection of maximal green sequences for quivers with cycles may require more than elementary techniques.

\bibliographystyle{plain}
\bibliography{bib_mgs}

\end{document}